\documentclass{commat}

\DeclareMathOperator\diag{diag}

\usepackage[shortlabels]{enumitem}
\usepackage{mathrsfs}

\graphicspath{{./pdfs/}}

\title{Transitions between root subsets associated with Carter diagrams}

\author{Rafael Stekolshchik}

\authorinfo{}{r.stekol@gmail.com}

\abstract{%
    For any two root subsets associated with two Carter diagrams that have the same $ADE$ type and the same size, we construct the transition matrix that maps one subset to the other.  The transition between these two subsets is carried out in some canonical way affecting exactly one root, so that this root is mapped to the minimal element in some root subsystem.  The constructed  transitions are involutions.  It is shown that all root subsets associated with the given Carter diagram are conjugate under the action of the Weyl group.  A numerical relationship is observed between enhanced Dynkin diagrams $\Delta(E_6)$, $\Delta(E_7)$ and $\Delta(E_8)$ (introduced by Dynkin-Minchenko) and Carter diagrams.  This relationship echoes the $2-4-8$ assertions obtained by Ringel, Rosenfeld and Baez in completely different contexts regarding the Dynkin diagrams $E_6$, $E_7$, $E_8$.
}



\VOLUME{30}
\NUMBER{3}
\YEAR{2022}
\firstpage{259}
\DOI{http://doi.org/10.46298/cm.9760}

\begin{document}

\begin{flushright}\textit{In memory of Semyon E. Konstein}\end{flushright}

\section{Introduction}

\subsection{Diagrams with cycles}

In $1972$, R. Carter introduced the so-called admissible diagrams representing semi-Coxeter elements
of conjugacy classes in the Weil group. The admissible  diagrams also represent root subsets\footnotemark[1]
of the root systems associated with the Weyl group. These root subsets sometimes form strange cycles,
strange because the extended Dynkin diagram $\widetilde{A}_l$ cannot be part of any admissible diagram.
The explanation for this fact was that in the case of extended Dynkin diagrams,
the inner products of roots of cycle $\widetilde{A}_l$ are negative, while in the case of admissible diagrams,
there are necessarily both positive and negative inner products, see \cite[Lemma A.1]{18}.
Thus was born the concept of the Carter diagram, see \cite{18}.
They differ from admissible diagrams
in that they take into account the sign of the inner product on the pair of roots;
for this, the language of \textit{solid and dotted} edges is used, see Section \ref{sec_solid_dotted}.
\footnotetext[1]{In the following, the phrase ``root subset'' always means the root subset of linearly independent roots.
}

Carter diagrams distinguish between negative inner products (solid edges) and positive
inner products (dotted edges). The number of solid edges must necessarily be odd, as well as the number
of dotted edges. This agrees well with the fact that, by definition,
the length of an admissible diagram is even. In the theorem on exclusion of long cycles
\cite[Theorem~3.1]{18} it was shown that any Carter diagram with
cycles of arbitrary even length can be reduced to another Carter diagram
containing only $4$-cycles. In proving this theorem, for each particular case,
it is checked that a semi-Coxeter element associated with a Carter diagram with long cycles
is conjugate to a semi-Coxeter element associated with another Carter diagram
containing only cycles of length $4$.

\subsection{Homogeneous classes of Carter diagrams}
  \label{sec_coval_list}
  The Dynkin diagram $A_l$, where $l \ge 1$ (resp. $D_l$, where $l \ge 4$; resp. $E_l$, where $l = 6,7,8$) is said to be the
  \textit{Dynkin diagram of $A$-type}
  (resp. \textit{$D$-type}, resp. \textit{$E$-type}). The Carter diagram $A_l$, where $l \ge 1$ (resp. $D_l$, $D_l(a_k)$,
  where $l \ge 4$, $1 \leq k \leq \big [ \tfrac{l-2}{2} \big ]$; resp. $E_l$, $E_l(a_k)$, where $l = 6,7,8$, $k = 1,2,3,4$)
  is said to be the \textit{Carter diagram of $A$-type} (resp. \textit{$D$-type}, resp. \textit{$E$-type}).

 The Carter diagrams of the same type and the same index constitute a \textit{homogeneous class} of Carter diagrams.
 Denote by $C(\Gamma)$ the homogeneous class containing the Carter diagram $\Gamma$,
 see \eqref{eq_one_type} and Fig. \ref{fig_all_diagr_ED}.
\begin{equation}
 \label{eq_one_type}
  \begin{split}
     C(E_6) = & \{ E_6, E_6(a_1), E_6(a_2)\}, \\
     C(E_7) = & \{ E_7, E_7(a_1), E_7(a_2), E_7(a_3), E_7(a_4)\}, \\
     C(E_8) = & \{ E_8, E_8(a_1), E_8(a_2), E_8(a_3), E_7(a_4), E_8(a_5), E_8(a_6), E_8(a_7), E_7(a_8)\}, \\
     C(D_l) = & \left\{ D_l, D_l(a_1), D_l(a_2), \dots, D_l\left(a_{\left[ \frac{l-2}{2} \right]}\right) \right\}, \text{ where } l \geq 4.
  \end{split}
\end{equation}

\begin{figure}
    \centering
    \includegraphics[height=0.9\textheight]{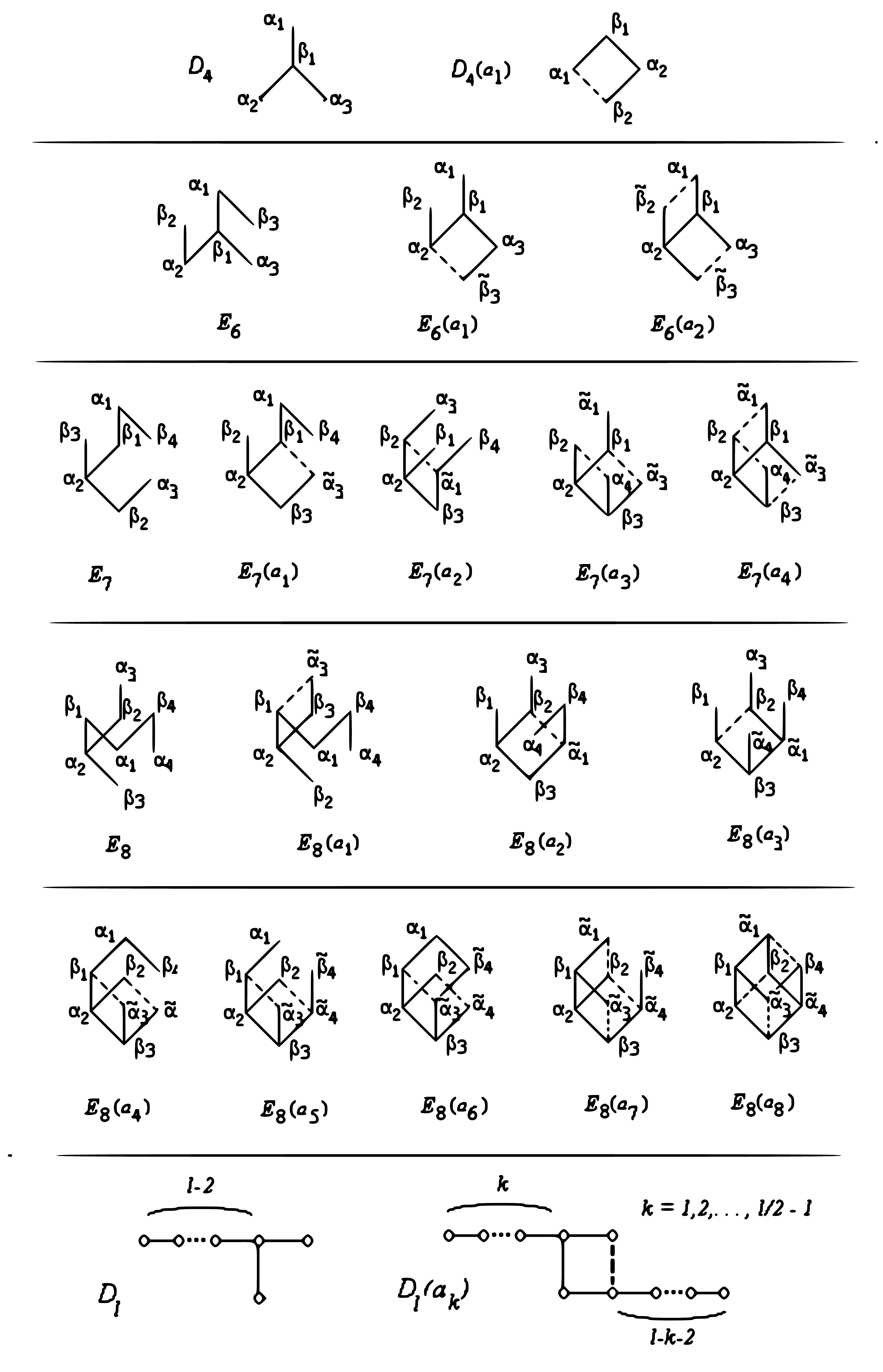}
    \caption{Carter diagrams of $D$ and $E$ types}
    \label{fig_all_diagr_ED}
\end{figure}

  A root subset associated with some diagram $\widetilde\Gamma$ is denoted by a $\widetilde\Gamma$-set.
Let $\widetilde{S}$ (resp. $S$) be a $\widetilde\Gamma$-set (resp. $\Gamma$-set).
In Tables~\ref{table:E6.MI2}-\ref{tab_part_root_syst_4} of Section~\ref{sect_proof_th},
the \textit{transition matrix} $M_I : \widetilde{S} \longmapsto S$ is constructed for the following
homogeneous pairs  $\{ \widetilde\Gamma, \Gamma \}$:
  \begin{equation}
    \label{eq_pairs_trans}
     \begin{array}{llll}
          (1) & \quad \{D_4(a_1), D_4\} & \qquad\qquad (9) & \quad \{E_8(a_1), E_8\}\\
          (2) & \quad \{D_l(a_k), D_l\} & \qquad\qquad (10) & \quad \{E_8(a_2), E_8\} \\
          (3) & \quad \{E_6(a_1), E_6\} & \qquad\qquad  (11) & \quad \{E_8(a_3), E_8(a_2)\} \\
          (4) & \quad \{E_6(a_2), E_6(a_1)\} & \qquad\qquad  (12) & \quad \{E_8(a_4), E_8(a_1)\}\\
          (5) & \quad \{E_7(a_1), E_7\} & \qquad\qquad  (13) & \quad \{E_8(a_5), E_8(a_4)\} \\
          (6) & \quad \{E_7(a_2), E_7\} & \qquad\qquad  (14) & \quad \{E_8(a_6), E_8(a_4)\} \\
          (7) & \quad \{E_7(a_3), E_7(a_1)\} & \qquad\qquad  (15) & \quad \{E_8(a_7), E_8(a_5)\}\\
          (8) & \quad \{E_7(a_4), E_7(a_3)\} & \qquad\qquad (16) & \quad \{E_8(a_8), E_8(a_7)\} \\
     \end{array}
  \end{equation}
  The diagrams obtained as images
  of the mapping $M_I$ are considered up to equivalence of Carter diagrams,
  see \cite[Section 1.3]{18}. The list \eqref{eq_pairs_trans} is called the \textit{adjacency list}.
  The adjacency list is not a complete list, but a minimal list that chains all Carter diagrams of
  to the same homogeneous class using the transition matrices
  constructed in Theorem \ref{th_invol}.

\subsection{A group of transitions $\mathscr{M}$}

Let $C(\Gamma)$ be a homogeneous class of Carter diagrams
out of \eqref{eq_one_type}, and $\widehat{S}$ be the set of all root subsets associated with diagrams of $C(\Gamma)$.
Let $\mathscr{M}$ be some subgroup of the group generated by transitions of type $M_I$
operating on $\widehat{S}$. The Weyl group $W$ operates simply transitively
on the set of bases of the corresponding root system, all bases are associated with
a single Dynkin diagram (which is also the Carter diagram).
This is not true for the group $\mathscr{M}$: any element $M_t \in \mathscr{M}$ (of type $M_I$)
can map some root subset $S_1 \in \widehat{S}$
to another root subset $S_2 \in \widehat{S}$.
The root subsets $S_1$ and $S_2$ are associated,
generally speaking, with different Carter diagrams $\Gamma_1 \in C(\Gamma)$
and $\Gamma_2 \in C(\Gamma)$. We will say that $S_1$ (resp. $S_2$) is a $\Gamma_1$-set
(resp. $\Gamma_2$-set).

\subsection{Theorems on transition matrix}

In this paper, we show that each homogeneous class of Carter diagrams essentially depicts
the same subset of the root system given in different bases. There are several chains of diagrams
containing homogeneous classes of Carter diagrams. The transition between neighboring bases
in any chain can be performed in some \textit{canonical way},
\textit{affecting exactly one root},
which is mapped to the minimum element in some root subsystem.
The transition matrix connecting two adjacent diagrams is an involution.

The main result of this paper can be formulated as follows:
Let $\{\widetilde\Gamma, \Gamma \}$ be a pair out of the adjacency list \eqref{eq_pairs_trans},
and let $\widetilde{S}$ (resp. $S$) be a $\widetilde\Gamma$-set (resp. $\Gamma$-set).
We construct the matrix $M_I$ having the following properties
({Theorem \ref{th_invol}}, {Theorem \ref{th_conjugate}}):
\begin{enumerate}[(a)]
\item
  The matrix $M_I$ is the \underline{transition matrix} transforming the roots of $\widetilde{S}$
  to the roots of~$S$.

\item
  The matrix $M_I$ transforms only one element in $\widetilde\alpha \in \widetilde{S}$,
  the remaining elements of $\widetilde{S}$ are left fixed. The element
  $\widetilde\alpha$ is transformed
  into the \underline{minimal element} of some Dynkin subset $S(\widetilde\alpha) \subset \widetilde{S}$,

\item
  The matrix $M_I$ acts as \underline{involution} on $\widetilde{S}$:
\begin{equation}
  \label{th_invol_3}
  M_I\widetilde\alpha = \alpha = -\widetilde\alpha + \sum\limits_{\tau_i \in \widetilde{S}}t_i\tau_i,  \quad
  M_I\tau_i = \tau_i \text{ for } \tau_i \neq \widetilde\alpha,
\end{equation}
  and $M_I$ acts also as \underline{involution} on $S$:
\begin{equation}
  \label{th_invol_4}
  M_I\alpha = \widetilde\alpha = -\alpha + \sum\limits_{\tau_i \in S}t_i\tau_i, \quad
  M_I\tau_i = \tau_i \text{ for } \tau_i \neq \alpha,
\end{equation}
  The values of $t_i$ in \eqref{th_invol_3} and \eqref{th_invol_4}
   are given in Tables~\ref{table:E6.MI2}-\ref{tab_part_root_syst_4} of Section~\ref{sect_proof_th}.

\item
 In most cases of Section~\ref{sect_proof_th},
  the mapping $M_I$ given in \eqref{th_invol_3} \underline{eliminates} one circle (or one endpoint),
  the mapping $M_I$ given in \eqref{th_invol_4} \underline{builds} one circle (or one endpoint).
  In case~(16) of Section~\ref{sect_proof_th},  $M_I$ \underline{eliminates} $3$ cycles.
\end{enumerate}

Only $ADE$ root systems are considered.

\subsection{Relationship with other diagrams containing cycles}

\subsubsection{Dynkin-Minchenko diagrams: Procedure of completion} \label{sec_DM_compl}

Let $\Gamma$ be a Dynkin diagram of the complex semisimple
Lie algebra $\mathfrak{g}$. Subdiagrams of $\Gamma$ are Dynkin diagrams of regular
subalgebras of $\mathfrak{g}$.
However, not all regular subalgebras can be obtained in this way.
Besides, non-conjugate regular subalgebras can have identical Dynkin diagrams.
Both problems are efficiently solved by using enhanced Dynkin diagrams.
In \cite{8}, Dynkin and Minchenko introduced a \underline{canonical enlargement}
of the basis called the \textit{enhanced basis} and \textit{enhanced Dynkin diagrams} representing an enhanced basis.
They constructed an enhancement of $\Gamma$ by a recursive procedure which they call the \textit{completion}:
At each step of the procedure, find a $D_4$-subset in the already introduced nodes,
add the maximal\footnotemark[2] root of this subset, and connect it by edges to the corresponding part
of the already introduced nodes.
For the enhanced Dynkin diagrams $\Delta(E_6)$ and $\Delta(E_7)$, see Section \ref{sect_A}.
\footnotetext[2]{Adding a minimal root (instead of a maximal one) leads to a
topologically isomorphic enhanced Dynkin diagram,
but distinguished by solid and dotted edges.}

   In \cite{19},\cite{20} Vavilov and Migrin combined both types of considered diagrams: \textit{Carter diagrams} and \textit{enhanced Dynkin diagrams}, they applied the language of solid and dotted edges to enhanced Dynkin diagrams.
 The obtained diagrams are called \textit{signed enhanced Dynkin diagrams}.
 They showed that any Carter diagram of the homogeneous class containing
 the Dynkin diagrams $E_6$, $E_7$, $E_8$ can be embedded into
 the signed enhanced Dynkin diagram $\Delta(\Gamma)$ associated with $\Gamma$
 such that the ``solid and dotted'' correspondence is preserved.

\subsubsection{Values 2,4,8 and diagrams $E_6$, $E_7$, $E_8$}

To the Vavilov-Margin observation mentioned above, I would like to add the following easily verifiable fact:
\begin{remark}
  \label{obs_corresp_DM}
   The number of extra nodes obtained by the Dynkin-Minchenko completion procedure
   for a simply-laced Dynkin diagram coincides with the number of Carter diagrams
   (with cycles) of the same type, see Fig.~\ref{fig_all_diagr_ED}, Table~\ref{tab_extra nodes}.
\end{remark}

\begin{table}
  \begin{center}
  \caption{Cardinality of extra nodes}
  \label{tab_extra nodes}
  \renewcommand{\arraystretch}{1.3}
  \begin{tabular} {|c|c|c|c|c|}
\hline
  $D_{2m}$  & $ D_{2m+1}$ & $E_6$  & $E_7$  & $E_8$ \\
\hline
  $m-1$       & $m-1$         & $2$     &  $4$    &  $8$  \\
\hline
  \end{tabular}
  \end{center}
\end{table}

For a further discussion of the relationship between
enhanced Dynkin diagrams and Carter diagrams, see Section \ref{sec_conjectures}.

In \cite{13}, in the context of Auslander-Reiten quivers, Ringel observed a completely different relation
between values $2$,$4$,$8$ and diagrams $E_6$, $E_7$, $E_8$, see Section \ref{sec_Ringel}.

In \cite{1}, Baez (in relation to Rosenfeld's idea in \cite{14}) points to another connection
between values $2$,$4$,$8$ and diagrams $E_6$, $E_7$, $E_8$, see Section \ref{sec_Baez}.

\subsubsection{McKee-Smyth diagrams: Eigenvalues in $(-2,2)$}

Much to my surprise, I found a complete list of $8$-vertex
Carter diagrams with circles in the paper of McKee and Smyth  \cite[Figs. 12-14]{11}.
The  $\{0, 1\}$-matrices with zeros on the diagonal can be regarded as adjacency matrices of graphs.
Assume that the off-diagonal elements of such a matrix to be chosen from the set $\{-1, 0, 1\}$.
Then, we get so-called a \textit{signed} graph, a non-zero $(\alpha,\beta)$th entry denotes a sign of $-1$ or $1$
on the edge connecting vertices $\alpha$ and $\beta$.
The signed graphs exactly correspond to our diagrams with solid and dotted edges.
The  matrix with zeros on the diagonal is called an \textit{uncharged matrix}.
By \cite[Theorem 4]{11}, the signed graphs maximal with respect to having all their eigenvalues in $(-2,2)$
are exactly Carter diagrams $E_8(a_i), 1 \le i \le 8$ and $D_l(a_i), i < l/2 - 1$, see Fig. \ref{fig_all_diagr_ED}.
If the diagonal matrix $2I$ is added to such an uncharged matrix,
then exactly \textit{partial Cartan matrix} will be obtained, see Section \ref{sec_partial}.
Then, the eigenvalues of the partial Cartan matrices
should lie in the interval $(0,4)$. Using eigenvalues one can get an invariant description of Carter diagrams,
see \cite[Section 4.4]{16}. For some details on the relationship between Carter diagrams and eigenvalues of partial Cartan matrices, see Section \ref{sec_eigenv}.

Similar results to \cite{11} were also obtained by Mulas and Stanic in \cite{12}.

\section{Diagrams containing cycles}

\subsection{Admissible diagrams: Conjugacy classes of $W$} \label{eq_def_adm}

Let $\varPhi$ be the root system  corresponding to the Weyl group $W$.
Each element $w \in W$ can be expressed in the form
 \begin{equation}
   \label{any_roots_0}
    w  = s_{\alpha_1} s_{\alpha_2} \dots s_{\alpha_k}, \text{ where } \alpha_i \in \varPhi \text{ for all } i.
 \end{equation}

Carter proved that $k$ in the decomposition \eqref{any_roots_0} is the smallest if
and only if the subset of roots $\{\alpha_1, \alpha_2, \dots , \alpha_k\}$ is linearly independent;
such a decomposition is said to be \textit{reduced}. The admissible diagram
corresponding to the given element $w$ is not unique, since the reduced
decomposition of the element $w$ is not unique.

Denote by $l_C(w)$ the smallest value $k$
 corresponding to any reduced decomposition~\eqref{any_roots_0}.
 The corresponding set of roots $\{\alpha_1,\alpha_2,\dots,\alpha_k\}$ consists of linearly
 independent and \underline{not ne-}
 \underline{cessarily simple} roots, see Lemma \ref{lem_lin_indep}.
 If $l(w)$ is the smallest value
 $k$ in any expression like \eqref{any_roots_0} such that all roots $\alpha_i$ are \textit{simple},
 then $l_C(w) \leq l(w)$.

 \begin{lemma}[{\cite[Lemma 3]{6}}]
  \label{lem_lin_indep}
   Let $\alpha_1, \alpha_2, \dots, \alpha_k \in \varPhi$.
   Then, $s_{\alpha_1} s_{\alpha_2} \dots s_{\alpha_k}$ is reduced
   if and only if $\alpha_1, \alpha_2, \dots, \alpha_k$ are linearly
   independent. \qed
 \end{lemma}

The Cartan matrix (resp. quadratic form) associated with $\varPhi$
is denoted by ${\bf B}$ (resp. $\mathscr{B}$).
The inner product induced by $\mathscr{B}$ is denoted by $(\cdot,\cdot)$.

Let us take the subset of linearly independent, but not necessarily simple roots $S \subset \varPhi$.
To the subset $S$ we associate some diagram $\Gamma$ that provides one-to-one correspondence
between roots of $S$ and nodes of $\Gamma$.
The diagram $\Gamma$ is said to be \textit{admissible} if the following two conditions hold:
\begin{enumerate}[(a)]
    \item \label{eq_def_adm(a)}
    The nodes of $\Gamma$ correspond to a set of linearly independent roots in $\varPhi$.

    \item \label{eq_def_adm(b)}
    If a subdiagram of $\Gamma$ is a cycle, then it contains an even number of nodes.
\end{enumerate}

Note that the admissible diagram may contain cycles, since the roots of $S$ are non necessarily simple,
see \cite[Section 1.2.1]{17}. Let us fix some basis of roots corresponding to the given
admissible diagram $\Gamma$:
\begin{equation*}
\label{basis_S}
  S = \{\alpha_1,\dots,\alpha_k,\beta_1,\dots,\beta_h\}.
\end{equation*}
The admissible diagram is bicolored, i.e., the set of nodes can be partitioned
into two disjoint subsets $S_\alpha = \{\alpha_1,\dots,\alpha_k \}$ and $S_\beta = \{\beta_1,\dots,\beta_h\}$,
where roots of $S_\alpha$ (resp. $S_\beta$) are  mutually orthogonal. The element
\begin{equation*}
\label{semi_Cox_elem_c}
   c = w_{\alpha}w_{\beta}, \quad \text{ where }
   w_{\alpha} = \prod\limits_{i=1}^k s_{\alpha_i}, \quad
   w_{\beta} = \prod\limits_{j=1}^h s_{\beta_j}
\end{equation*}
is called the \textit{semi-Coxeter element}; it represents the conjugacy class associated with
the admissible diagram $\Gamma$  and root subset $S$ (not necessarily root system).

\subsection{Carter diagrams: Language of ``solid and dotted'' edges}
  \label{sec_solid_dotted}

In \cite{18}, it was observed that the cycles in the admissible diagrams with necessity contains at least
one pair of roots  $\{\alpha_1, \beta_1\}$ with $(\alpha_1, \beta_1) > 0$ and at least one pair of roots
$\{\alpha_2, \beta_2\}$ with $(\alpha_2, \beta_2) < 0$, where $(\cdot, \cdot)$ is the Tits bilinear form
associated with the root system $\varPhi$. This observation motivated me to distinguish such pairs of roots:
Let us draw the \textit{dotted} (resp. \textit{solid}) edge $\{\alpha, \beta\}$
if $(\alpha, \beta) > 0$ (resp. $(\alpha, \beta) < 0)$.  The admissible diagram with dotted and solid
edges is said to be the \textit{Carter diagram.}
Up to dotted edges, the classification of Carter diagrams coincides with the classification of
admissible diagrams.

In the theorem on exclusion of long cycles \cite{18}, it was shown that any Carter diagram
with cycles of arbitrary even length can be reduced to diagrams with cycles of length $4$ only.
This explains why the admissible diagrams
$D_l(b_{\frac{1}{2}l-1})$, $E_7(b_2)$, $E_8(b_3)$, $E_8(b_5)$ listed
in \cite[Table 2]{6} do not appear in the lists of conjugacy classes.
The Carter diagrams with conjugate semi-Coxeter elements are said to
be \textit{equivalent}. The Carter diagrams (with cycles) representing non-Coxeter conjugacy classes are given in
Fig. \ref{fig_all_diagr_ED}. For another view of these diagrams, see \cite[Table 1]{18}.

\subsection{Carter diagrams: Eliminating the cycle.}

The semi-Coxeter elements generated by reflections $\{s_{\alpha_1}, s_{\alpha_2}, s_{\beta_1}, s_{\beta_2} \}$
constitute exactly two conjugacy classes with representatives $w_t$ and $w_o$, see Fig. \ref{D4_two_CCL}.
Semi-Coxeter elements $w_t$ and $w_o$ are distinguished by orders of reflections in the decomposition of $w_t$,
and $w_o$. Here, $t$ is the bicolored order $\{\alpha_1, \alpha_2, \beta_1, \beta_2 \}$, and
$o$ is the cyclic order $\{\alpha_1, \beta_1, \alpha_2, \beta_2 \}$, see \cite[Section 1.2]{18}.

\begin{figure}
    \centering
    \includegraphics[scale=0.3]{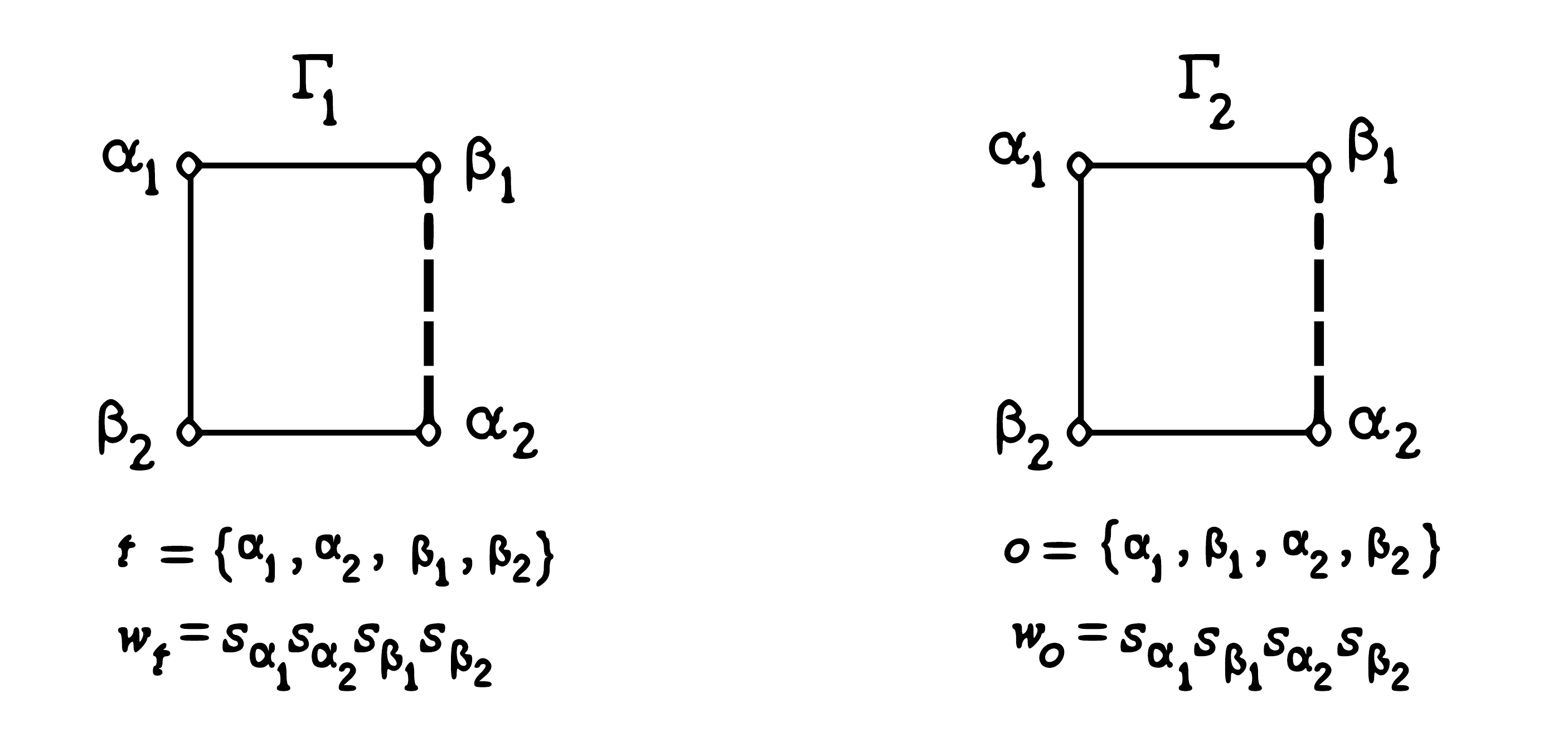}
    \caption{Diagram $\Gamma_1$ (resp. $\Gamma_2$) of type $D_4(a_1)$ (resp. equivalent to $D_4$)}
    \label{D4_two_CCL}
\end{figure}

The element $w_o$ is conjugate to the
$w_{\widetilde{o}} = s_{\alpha_1}s_{\alpha_2}s_{-(\alpha_1+\beta_1+ \beta_2)}s_{\beta_2}$
since
\begin{equation*}
\begin{split}
   w_o  = & s_{\alpha_1}s_{\beta_1}s_{\alpha_2} s_{\beta_2} = s_{\alpha_1+\beta_1}s_{\alpha_1}s_{\alpha_2}s_{\beta_2}
   \stackrel{s_{\alpha_1+\beta_1}}{\simeq}  s_{\alpha_1}s_{\alpha_2}s_{\beta_2}s_{\alpha_1+\beta_1}= \\
   & s_{\alpha_1}s_{\alpha_2}s_{\alpha_1+\beta_1+ \beta_2}s_{\beta_2} =
   s_{\alpha_1}s_{\alpha_2}s_{-(\alpha_1+\beta_1+ \beta_2)}s_{\beta_2} = w_{\widetilde{o}}.
 \end{split}
\end{equation*}
The elements $w_o$ and $w_{\widetilde{o}}$ are conjugate, the corresponding sets of roots are as follows:
\begin{equation*}
 S_1 = \{\alpha_1, \beta_1, \alpha_2, \beta_2 \}  \text{ and }
 S_2 = \{\alpha_1, \alpha_2, -(\alpha_1+\beta_1+ \beta_2), \beta_2 \}.
\end{equation*}
 There is a map $M: S_1 \longmapsto S_2$ acting as follows:
\begin{equation*}
  M\alpha_1 = \alpha_1, \quad M\alpha_2 = \alpha_2, \quad   M\beta_2 = \beta_2, \quad
  \widetilde\beta_1 = M\beta_1 = -(\alpha_1+\beta_1+ \beta_2).
\end{equation*}
Note that $M$ is an involution, $M : S_1 \to S_2$ and $M : S_2 \to S_1$, since
\begin{gather*}
    M^2\beta_1 = -(\alpha_1+M\beta_1+ \beta_2) = -(\alpha_1+ \beta_2) + (\alpha_1+\beta_1+ \beta_2) = \beta_1, \\
    M\beta_1 =  \widetilde\beta_1
    \quad \textup{ and } \quad
    M\widetilde\beta_1 = \beta_1.
\end{gather*}
Thus, $M$ transforms the root $\beta_1$ into the minimal element of the root subsystem $\{ \alpha_1, \beta_1, \beta_2 \}$\hspace*{-0.5pt}. In this paper, we will encounter a number of \textit{involution mappings} $M$ that map a certain element to the \textit{minimal element} of some root subsystem of $\varPhi$. So, we observe that there are two different orders of reflections:
\begin{itemize}
  \item The cyclic order of reflections $o$. Then, we get a $4$-cycle leading to the Coxeter class $D_4$
  of $W(D_4)$, see Fig. \ref{D4_equiv_diagr}.
  \item The bicolored order of reflections $t$. Then, we get an \textit{indestructible} $4$-cycle leading to the
  semi-Coxeter class $D_4(a_1)$.
\end{itemize}

\begin{figure}
    \centering
    \includegraphics[scale=0.3]{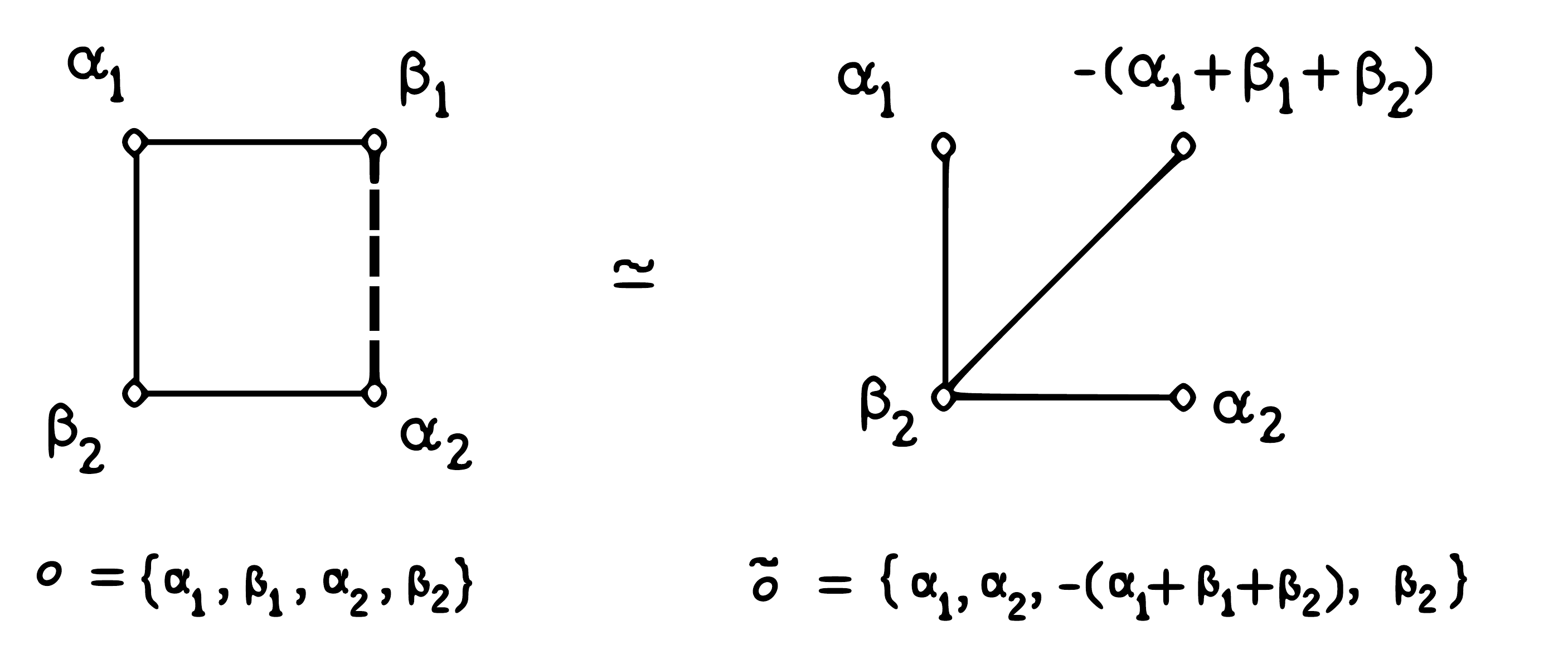}
    \caption{Conjugate elements $\{ w_o , w_{\widetilde{o}}\}$ corresponding to the Coxeter class $D_4$}
    \label{D4_equiv_diagr}
\end{figure}

\subsection{Connection diagrams}
  \label{sec_conn_diagr}
In \cite{18}, in addition to Carter diagrams, the so-called \textit{connection diagrams} were introduced.
Let $S$ be a set of linearly independent and not necessarily simple roots, $o$ be the order of reflections in the
decomposition \eqref{any_roots_0} of element $w$ associated with the set of roots $S$.
The connection diagram is a pair $(\Gamma, o)$, where $\Gamma$ corresponds to the set $S$.
In the connection diagram $\Gamma$, edges are also solid and dotted as in Carter diagrams.
The connection diagrams serve to transform one Carter diagram into another,
since in the process of transformation we can get non-Carter diagrams -- the evenness
of the cycles may be violated, see \cite[Section 1.2.2]{18}.

In \cite{18}, three equivalence transformations operating on
the connection diagrams and Carter diagrams were introduced: similarities, conjugations and s-permutations.
The Carter diagrams are studied there up to equivalence. In what follows, we only need \textit{similarity}.
Let $\alpha$ be a root in the $\Gamma$-set $S$.
The similarity transformation $L_{\alpha}$ reflects the root~$\alpha$:
\begin{equation}
   \label{eq_similar}
    L_{\alpha}: \alpha \longmapsto -\alpha.
\end{equation}
Two connection diagrams obtained from each other by a sequence of reflections \eqref{eq_similar},
are called \textit{similar} connection diagrams, see Fig. \ref{fig_simil}.

\begin{figure}
    \centering
    \includegraphics[scale=0.3]{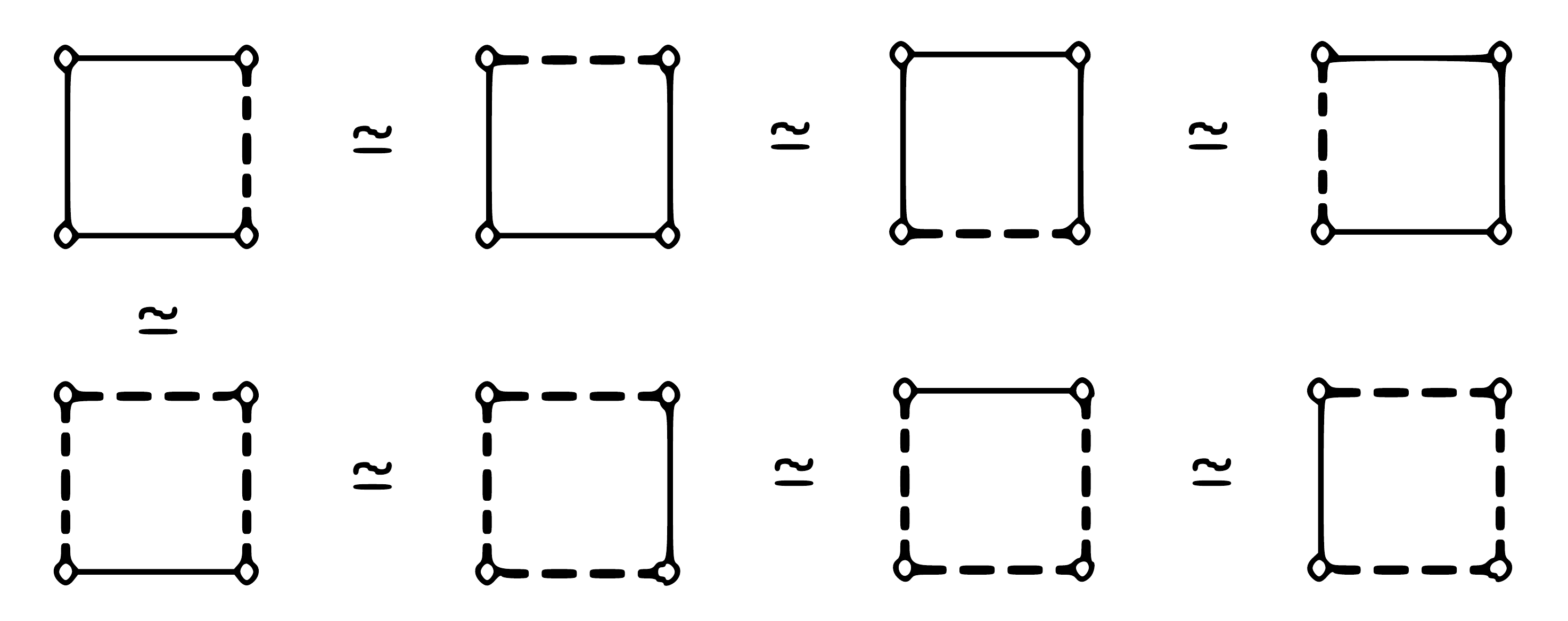}
    \caption{Eight similar $4$-cycles equivalent to $D_4(a_1)$)}
    \label{fig_simil}
\end{figure}

\subsection{Bicolored partition}
  Let $\Gamma$ be a Carter diagram and
\begin{equation}
   \label{eq_alpha_bet}
     S = \{\alpha_1, \alpha_2, \dots, \alpha_k, \beta_1, \beta_2, \dots, \beta_h \}
 \end{equation}
 be any $\Gamma$-set (of not necessarily simple roots),
 where roots of the set 
 \[
    S_{\alpha}
    := \{\alpha_i \mid i = 1,\dots,k\}
    \qquad (\textup{resp. }
    S_{\beta}
    := \{\beta_j \mid j = 1,\dots,h\})
\]
are mutually orthogonal.
 According to condition~\ref{eq_def_adm}\ref{eq_def_adm(a)}, there exists a certain set \eqref{eq_alpha_bet} of linearly independent roots, and thanks to condition~\ref{eq_def_adm}\ref{eq_def_adm(b)}, there exists a partition $S = S_{\alpha} \coprod S_{\beta}$
 which is said to be the \textit{bicolored partition}.

  Let $w = w_1 w_2$ be the decomposition of $w$ into the product of two involutions.
 By \cite[Lemma 5]{6} each of $w_1$ and $w_2$ can
 be expressed as a product of reflections as follows:
 \begin{equation}
   \label{two_invol}
      w = w_1{w}_2, \quad \text{ where } \quad
      w_1 = s_{\alpha_1} s_{\alpha_2} \dots s_{\alpha_k}, \quad
      w_2 = s_{\beta_1} s_{\beta_2} \dots s_{\beta_h},
 \end{equation}
 where subset $S_{\alpha} = \{\alpha_1,\dots,\alpha_k\}$ (resp. $S_{\beta} = \{\beta_1,\dots,\beta_h\}$)
 consists of mutually orthogonal roots.
 Let 
 \begin{equation}
   \label{root_subset_L}
       \Pi_w = \{ \alpha_1, \alpha_2, \dots, \alpha_k, \beta_1, \beta_2, \dots, \beta_h \}
 \end{equation}
 be the linearly independent root subset. Then, the decomposition \eqref{two_invol} is reduced, see Lemma \ref{lem_lin_indep}, and $k + h = l_C(w)$.
 The decomposition \eqref{two_invol} is said to be a \textit{bicolored decomposition}.

\section{The Cartan matrix}
  \label{sect_Cartan}

\subsection{The generalized Cartan matrix}

  The $n\times{n}$ matrix $K = (k_{ij})$, where $1 \leq i,j \leq n$, such that
 \begin{equation*}
   \begin{split}
     (C1) & \quad k_{ii} = 2 \text{ for } i = 1,\dots, n, \\
     (C2) & \quad -k_{ij} \in \mathbb{Z} = \{0, 1, 2, \dots \} \text{ for } i \neq j, \\
     (C3) & \quad  k_{ij} = 0 \text{ implies } k_{ji} = 0 \text{ for } i, j = 1, \dots, n,
    \end{split}
 \end{equation*}
  is called a \textit{generalized Cartan matrix}, \cite{10}, \cite[Section 2.1]{16}.
  For the Carter diagram $\Gamma$, which is not a Dynkin diagram,
  the condition (C2) fails: The elements $k_{ij}$ associated with dotted edges are positive.

 If the Carter diagram does not contain any cycle,
 then the Carter diagram is the Dynkin diagram, the corresponding conjugacy class
 is the conjugacy class of the Coxeter element, and
 the partial Cartan matrix is the classical Cartan matrix, which is
 the particular case of a generalized Cartan matrix.

 \subsection{The partial Cartan matrix}
   \label{sec_partial}

 Similarly to the Cartan matrix associated with Dynkin diagrams, we determine the Cartan matrix
 for each pair \{$\Gamma$, S\} consisting of the connection or Carter diagram $\Gamma$ and $\Gamma$-set $S$:

 \begin{equation}
   \label{canon_dec_2}
   B_{\Gamma} :=
      \left (
        \begin{array}{cccccc}
         (\tau_1, \tau_1) & \dots & (\tau_1, \tau_n) \\
         \dots                & \dots & \dots \\
         (\tau_n, \tau_1) & \dots & (\tau_n, \tau_n) \\
        \end{array}
      \right ),
 \end{equation}
 where $S = \{\tau_1,\dots,\tau_n\}$.
 We call the matrix $B_{\Gamma}$ a \textit{partial Cartan matrix} corresponding to
 the diagram $\Gamma$.
 The partial Cartan matrix $B_{\Gamma}$ is well-defined since products $(\tau_i, \tau_j)$ in \eqref{canon_dec_2}
 do not depend on the choice of the $\Gamma$-set $S$. The elements of the partial
 Cartan matrix are uniquely determined by the diagram $\Gamma$ as follows:

 \begin{equation}
  \label{bilin_form_B}
   (\tau_i, \tau_j) =
     \begin{cases}
        2, & \text{if } \tau_i = \tau_j, \\
        0, & \text {if } \tau_i \text{ and } \tau_j \text{ are not connected}, \\
        -1, & \text{if edge } \{\tau_i, \tau_j\} \text{ is solid}, \\
        1, & \text{if edge } \{\tau_i, \tau_j\} \text{ is dotted}. \\
     \end{cases}
 \end{equation}
Let $L$ be the subspace spanned by the vectors $\{ \tau_1,\dots,\tau_n \}$.
We write this fact as follows:
\begin{equation}
  \label{space_L}
  L = [\tau_1,\dots,\tau_n].
\end{equation}
The subspace $L$ is said to be the $S$-\textit{associated subspace}.
Let ${\bf B}$ be the Cartan matrix corresponding to the primary
root system $\varPhi$.

\begin{proposition} \label{restr_forms_coincide}
    \begin{enumerate}[(i)]
        \item The restriction of the bilinear form associated with the Cartan matrix ${\bf B}$ on the subspace $L$ coincides with the bilinear form associated with the partial Cartan matrix $B_{\Gamma}$, i.e., for any pair of vectors $v, u \in L$, we have
         \begin{equation}
           \label{restr_q}
               (v, u)_{\Gamma} = (v, u), \text{ and }
               \mathscr{B}_{\Gamma}(v) = \mathscr{B}(v).
         \end{equation}
    
        \item For every Carter diagram, the matrix $B_{\Gamma}$ is positive definite.
    \end{enumerate}
\end{proposition}
 
\begin{proof}
    \begin{enumerate}[(i)]
        \item From \eqref{canon_dec_2} we deduce:
         \begin{equation*}
             (v, u)_{\Gamma}
             = \left(\sum_i {t_i{\tau_i}}, \sum_j {q_j{\tau_j}}\right)_{\Gamma} 
             = \sum_{i,j} t_i{q}_j(\tau_i, \tau_j)_{\Gamma}
             = \sum_{i,j} t_i{q}_j(\tau_i, \tau_j)
             = (v, u).
         \end{equation*}
        \item This follows from (i).
    \qedhere
    \end{enumerate}
\end{proof}

 If $\Gamma$ is a Dynkin diagram, the partial Cartan matrix $B_{\Gamma}$ is the Cartan matrix
 associated with $\Gamma$.  By \eqref{restr_q} the matrix $B_{\Gamma}$ is positive definite.
 The symmetric bilinear form associated with $B_{\Gamma}$ is denoted by $(\cdot, \cdot)_{\Gamma}$ and the
 corresponding quadratic form is denoted by $\mathscr{B}_{\Gamma}$.

\begin{remark}
    \begin{enumerate}[(i)]
        \item D.~Leites noticed that there are a number of other cases, where some off-diagonal elements of the Cartan matrices are positive. For example, this is so in the case of Lorentzian algebras, see \cite{4}, \cite{7}. However, in these cases the Cartan matrices are of \textit{hyperbolic type}, whereas the partial Cartan matrices are \textit{positive definite}.

        \item I would like to quote S. Brenner's article: ``\dots it is amusing to note that there is a surprisingly large intersection between the finite type quivers with commutativity conditions and the diagrams by Carter in describing conjugacy classes of the classical Weyl groups \dots'', \cite[p.43]{3}. On various other cases arising in the representation theory of quivers, algebras and posets with Cartan matrices containing positive off-diagonal elements, see \cite{2}, \cite[10.7]{5}, \cite{15}.
    \end{enumerate}
\end{remark}

\section{Transitions}

\subsection{First transition theorem}
  Let $\{\widetilde\Gamma, \Gamma \}$ be a homogeneous pair of Carter diagrams,
   $\widetilde{S}$ be a $\widetilde\Gamma$-set, and $S$ be a $\Gamma$-set.
  In this section, we construct a mapping connecting $\widetilde{S}$ and $S$.
  This mapping represents the transition matrix
  connecting $\widetilde{S}$ and $S$ as bases in the linear spaces.
  The transition matrix has some good properties that are presented in Theorems \ref{th_invol}
  and \ref{th_conjugate}.
  Let $\Gamma'$ be the subdiagram of $\Gamma$, and subset $S' \subset S$ be a $\Gamma'$-set.
  If $\Gamma'$ is  the Dynkin diagram, we call $S'$ the Dynkin subset.

 \begin{theorem} \label{th_invol}
  For each pair of diagrams $\{\widetilde\Gamma, \Gamma \}$ out of the list \eqref{eq_pairs_trans},
  there exists the linear transformation matrix $M_I$ mapping each $\widetilde\Gamma$-set $\widetilde{S}$  to
  some $\Gamma$-set $S$ being the image of $M_I$,
  see Tables~\ref{table:E6.MI2}-\ref{tab_part_root_syst_4} of Section~\ref{sect_proof_th}.
  The matrix $M_I$ is the \underline{transition matrix} binding $\widetilde{S}$ and $S$ as bases in the linear spaces.
    \begin{enumerate}[(i)]
    \item The matrix $M_I$ acts only on one element $\widetilde\alpha \in \widetilde{S}$ and does not change remaining elements in $\widetilde{S}$; $M_I$ transforms $\widetilde\alpha$ into the \underline{minimal element} $\alpha$ of some  Dynkin subset $S(\widetilde\alpha)$ in $\widetilde{S}$:
    \begin{equation*}
      \begin{cases}
         \widetilde\alpha \in S(\widetilde\alpha) \subset \widetilde{S}, \\
          M_I \tau_i = \tau_i \text{ for all }
               \tau_i \in \widetilde{S}, \tau_i \neq \widetilde\alpha, \\
          M_I \widetilde\alpha = \alpha =
             -\widetilde\alpha + \sum t_i \tau_i,
             \text{ where the sum is taken over }
               \tau_i \in \widetilde{S}, \tau_i \neq \widetilde\alpha, t_i \in \mathbb{Z}, \\
          \alpha \text{ - minimal element in } S(\widetilde\alpha).
      \end{cases}
    \end{equation*}
    The image $S = M_I\widetilde{S}$ is the set $\{\widetilde{S}\backslash{\widetilde{\alpha}\}} \sqcup \{\alpha\}$ that satisfies to  the orthogonality relations of the Carter diagram $\Gamma$.

    \item The transformation $M_I: \widetilde{S} \longmapsto S$ is an \underline{involution}\footnotemark[1] on the set $\widetilde{S} \sqcup \{\alpha\}$:
    \begin{equation*}
        M_I \widetilde\alpha = \alpha \quad \text{ and } \quad M_I\alpha = \widetilde\alpha.
    \end{equation*}
  \end{enumerate}
\end{theorem}

\footnotetext[1]{The index $I$ in the symbol $M_I$ originally appeared to indicate that $M_I$ is an involution.
In what follows, we will use also other notation for the matrix $M$ and its index,
which are more related to a specific situation.}

For each pair of diagrams $\{\widetilde\Gamma, \Gamma \}$ from list \eqref{eq_pairs_trans},
the matrix $M_I$ is defined in Tables~\ref{table:E6.MI2}-\ref{tab_part_root_syst_4} of Section~\ref{sect_proof_th}.
The matrix $M_I$ is the transition matrix transforming each basis $\widetilde{S}$ into some basis $S$.

The proof of Theorem \ref{th_invol} is given in Section \ref{sect_proof_th}. 
It is carried out separately for each pair
 $\{\widetilde\Gamma, \Gamma\}$ in the adjacency list \eqref{eq_pairs_trans}.

\subsection{The chain of homogeneous pairs}

  Let $\widetilde{\Gamma}$ be a Carter diagram. Denote by $C(\widetilde{\Gamma})$ 
  the homogeneous class containing $\widetilde{\Gamma}$.
  For any Carter diagram $\widetilde{\Gamma}$ and the Dynkin diagram $\Gamma$ 
  from  $C(\widetilde{\Gamma})$, there exists the chain of homogeneous pairs 
  connecting $\widetilde{\Gamma}$ and $\Gamma$ as follows:
\begin{equation}
  \label{eq_comm_chain}
   \{ \{\widetilde{\Gamma}, \Gamma_1\}, \{\Gamma_1, \Gamma_2\},\dots,
     \{\Gamma_{k-1}, \Gamma_k\},\{\Gamma_k, \Gamma\} \}.
\end{equation}
  This fact follows easily from consideration of the adjacency list \eqref{eq_pairs_trans}.

\subsubsection{Example: from $E_8(a_8)$ to $E_8$ }
  \label{sec_prod_trans_matr}
 There are $16$ cases in Section~\ref{sect_proof_th}.
 Denote by $M_I(n)$ the transition matrix of the $n$-th case. The similarity transformation $L_{\tau_i}$
 from \eqref{eq_similar} is the diagonal matrix of the form
 \begin{equation*}
    \diag(1,1\dots,1,-1,1,\dots,1)
 \end{equation*}
  with $-1$ in the $\{i,i\}$th entry.
 The homogeneous pairs are bound by matrices $M_I$ and
 similarity matrices $L_{\tau_i}$. Consider, for example, the chain diagrams
 $E_8(a_8)$, $E_8(a_7)$, $E_8(a_5)$, $E_8(a_4)$, $E_8(a_1)$, $E_8$, see eq. \eqref{eq_chain_trans}
 and Section \ref{sec_coval_list}.
 
\begin{equation}
\label{eq_chain_trans}
\begin{split}
 E_8(a_8) \overset{M_I(16)}{\xrightarrow{\hspace{0.8cm}}}
 E_8(a_7) \overset{L_{\beta_2}}{\xrightarrow{\hspace{0.8cm}}} &
 E_8(a_7) \overset{M_I(15)}{\xrightarrow{\hspace{0.8cm}}}
 E_8(a_5)  \overset{L_{\widetilde\alpha_3}}{\xrightarrow{\hspace{0.8cm}}} E_8(a_5)  \\
  \overset{M_I(13)}{\xrightarrow{\hspace{0.8cm}}} E_8(a_4) &
\overset{M_I(12)}{\xrightarrow{\hspace{0.8cm}}} E_8(a_1)  \overset{M_I(9)}{\xrightarrow{\hspace{0.8cm}}} E_8
 \overset{L_{\alpha_4}L_{\beta_4}}{\xrightarrow{\hspace{1.2cm}}} E_8
 \end{split}
\end{equation}

In eq.~\eqref{eq_chain_trans}, we mean that instead of each diagram $\Gamma$ there is some $\Gamma$-set.
The matrices $M_I(n)$ are given in Appendix \ref{sec_trans}.
Consider the product of matrices of \eqref{eq_chain_trans}:

\begin{equation*}
   F = L_{\alpha_4}L_{\beta_4}{}M_I(9)M_I(12)M_I(13)L_{\widetilde\alpha_3}
       M_I(15)L_{\beta_2}M_I(16).
\end{equation*}

The matrix $F$ maps the $E_8(a_8)$-basis $\widetilde{S}$  to a certain $E_8$-basis $S = F\widetilde{S}$:
\begin{equation*}
F:  \widetilde{S} = \{\widetilde\alpha_1, \alpha_2, \widetilde\alpha_3, \widetilde\alpha_4,
               \beta_1, \beta_2, \beta_3, \beta_4\} \longmapsto
S =  \{\alpha_1, \alpha_2, \alpha_3, \alpha_4, \beta_1, \beta_2, \beta_3, \beta_4\}
\end{equation*}

The chain \eqref{eq_chain_trans} is parallel
to the \underline{ascending chain} of maximal eigenvalues of the corresponding partial Cartan matrices,
see Section \ref{sec_eigenv}.

\subsubsection{Alternative transitions}
  \label{sec_altern}

The transition matrices from the adjacency list \eqref{eq_pairs_trans}
do not constitute a complete set of possible transitions. For example,
to the transition $\{ E_8(a_6),  E_8(a_4) \}$, one can add more pairs containing $E_8(a_6)$:
$\{E_8(a_6),  E_8(a_5) \}$ and $\{E_8(a_6),  E_8(a_1)\}$, see Fig. \ref{fig_altern_trans}.
The marked vertex corresponds to the root, which is converted with the transition matrix $M_I$.
To keep the solid and dotted edges corresponding to cases (9) and (13)
in Tables \ref{tab_part_root_syst_3} -- \ref{tab_part_root_syst_4},
each such transition is followed by actions of similarities which result
in the desired \textit{edge type}\footnotemark[1].
Note that the vertex names for $E_8(a_1)$ and $E_8(a_5)$
are different from those in Tables \ref{tab_part_root_syst_3} -- \ref{tab_part_root_syst_4}:
\begin{equation*}
  \begin{split}
     \{ E_8(a_6), & E_8(a_4) \}:  ~ M_I \widetilde\beta_4 = \beta_4 =
                        -(\widetilde\beta_4 + 2\alpha_1 + 2\beta_1 + 2\alpha_2 + \beta_2 + \beta_3). \\
     \\
     \{ E_8(a_6), & E_8(a_5) \}:  ~ M_I \beta_2 = \widetilde\beta_2 =
                       -(\widetilde\beta_4 + 2\alpha_1 + 2\beta_1 + 2\alpha_2 + \beta_2 + \beta_3), \\
                   & \text{ followed by actions of }  L_{\tau},  \text{ with }
                    \tau = \alpha_1, \alpha_2, \beta_1, \widetilde\beta_2.    \\
     \\
     \{ E_8(a_6), & E_8(a_1) \}: ~ M_I \beta_3 = \widetilde\beta_3 =
                        -(\widetilde\beta_4 + 2\alpha_1 + 2\beta_1 + 2\alpha_2 + \beta_2 + \beta_3). \\
                    & \text{ followed  by actions of }  L_{\tau}, \text{ with }
                    \tau = \widetilde\alpha_4. 
  \end{split}
\end{equation*}
\footnotetext[1]{The property of the edge to be solid or dotted is called the edge type.}

\begin{figure}
    \centering
    \includegraphics[scale=0.35]{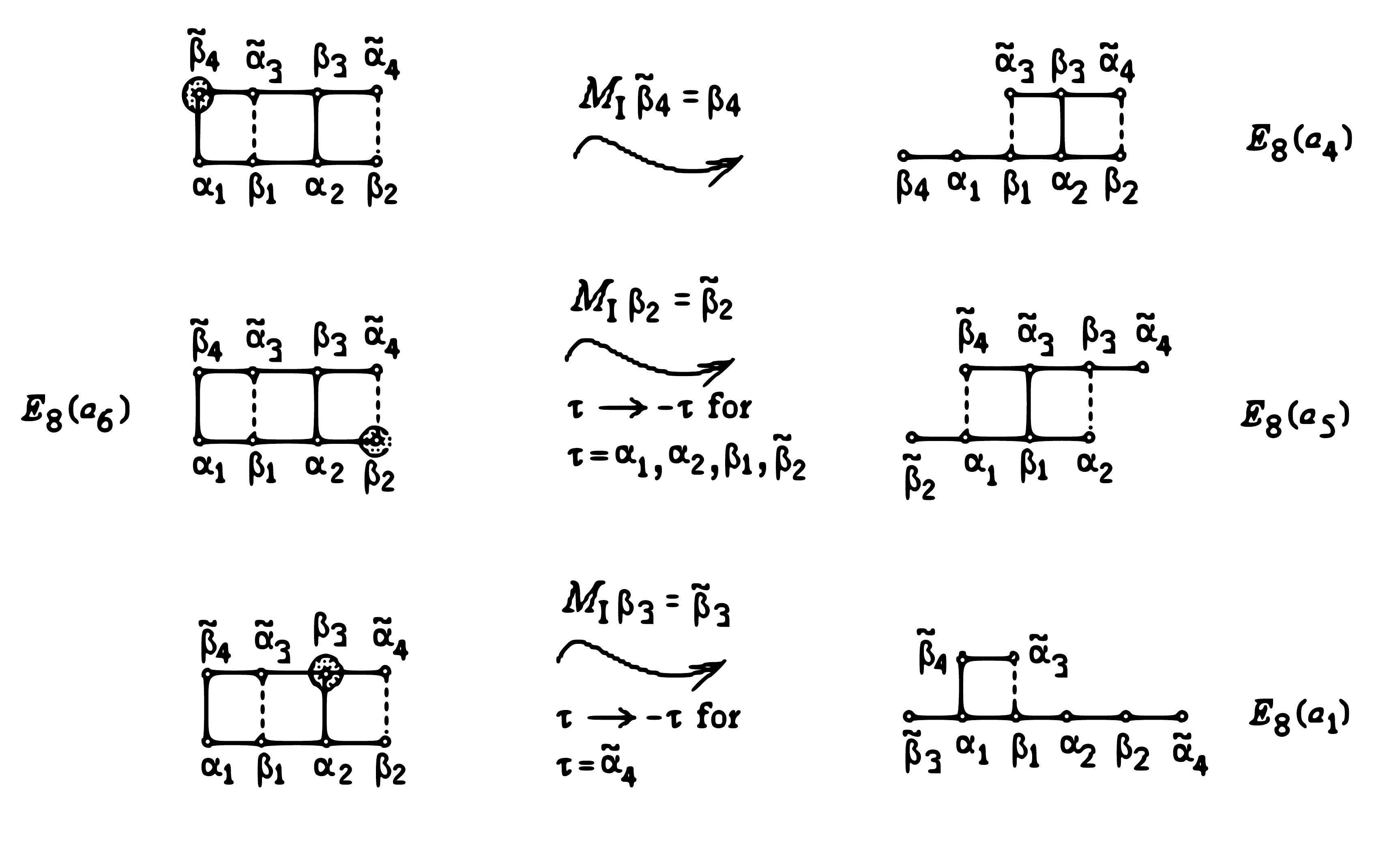}
    \caption{Alternative transitions $\{E_8(a_6),  E_8(a_5) \}$ and $\{E_8(a_6),  E_8(a_1) \}$)}
    \label{fig_altern_trans}
\end{figure}

\subsubsection{The product of transition matrices}
 As in Section \ref{sec_prod_trans_matr}, for any chain \eqref{eq_comm_chain}, one can construct
 the matrices $F$ and $F^{-1}$, where $F$ is the product of corresponding transition matrices $M_I(n)$
 and similarity matrices $L_{\tau_i}$. The matrix $F$ is invertible since all $M_I(n)$  and $L_{\tau_i}$
 are invertible.
\begin{equation}
  \label{eq_F_and_Fmin1}
   F : \widetilde{\Gamma} \longmapsto  \Gamma, \qquad F^{-1} : \Gamma \longmapsto  \widetilde{\Gamma}.
\end{equation}
This means that for any $\widetilde{\Gamma}$-set $\widetilde{S}$ there exists $\Gamma$-set $S$
such that
\begin{equation}
  \label{eq_F_and_Fmin1_2}
   F\widetilde{S} = S, \qquad F^{-1}{S} = \widetilde{S}.
\end{equation}
The matrix $F$ does not depend on $\widetilde{S}$ and $S$.
If $S = \{\tau_1,\dots, \tau_n\}$, $F$ transforms $\tau_i$ so that
\begin{equation}
  \label{eq_coed_fij}
   F\tau_i = \sum\limits_{j = 1}^n f_{ji}\tau_i,
\end{equation}
where $f_{ji}$  are some coefficients that depend only on diagrams $\widetilde{\Gamma}$ and $\Gamma$.

\subsection{Action of the Weyl group on a Carter diagram}

We suppose that $\{\widetilde\Gamma, \Gamma\}$ is the homogeneous pair of Carter diagrams,
where $\Gamma$ is the simply-laced Dynkin diagram, and $W$ is the Weyl group associated
with $\Gamma$.

\begin{lemma}
  \label{lem_commute}
  Let $F$, $F^{-1}$ be the matrices described in \eqref{eq_F_and_Fmin1}, \eqref{eq_F_and_Fmin1_2}.
  Then,
  \begin{enumerate}[(i)]
    \item The matrix $F$ commutes with the Weyl group $W$ on any $\Gamma$-set $S$ as follows:
    \begin{equation}
      \label{lem_F_commute}
        wF = Fw \text{ for any } w \in W.
    \end{equation}
    
    \item Let $\Gamma$-sets $S$ and $S'$ be conjugate by some $w \in W$: $wS = S'$. Then, $FS$ and $FS'$ are conjugate by the same element $w \in W$, i.e,
    \begin{equation*}
           wS = S' \quad \text{ implies }  \quad  wFS = FS'.
    \end{equation*}
  \end{enumerate}
\end{lemma}

\begin{proof}
    \begin{enumerate}[(i)]
        \item It suffices to prove eq. \eqref{lem_F_commute} for each element $\tau_i \in S$. Each element $w \in W$ transforms basis $S$ to another basis $S' = wS$, where $w\tau_i = \tau'_i$, and $S' = \{\tau'_1,\dots, \tau'_n\}$. In our case,
        \begin{equation*}
            \begin{split}
                & Fw\tau_i = F\tau'_i = \sum\limits_{j = 1}^n f_{ji}\tau'_i \quad \text{ and } \\
                & wF\tau_i = w\sum\limits_{j = 1}^n f_{ji}\tau_i = \sum\limits_{j = 1}^n f_{ji}w\tau_i =
                            \sum\limits_{j = 1}^n f_{ji}\tau'_i.
            \end{split}
        \end{equation*}
        Therefore, $Fw\tau_i = wF\tau_i$ for any $\tau_i \in S$.

        \item If $wS = S'$ then by (i), we have $wFS = FwS = FS'$, i.e., $FS$ and $FS'$ conjugate by the same element $w \in W$.
    \qedhere
    \end{enumerate}
\end{proof}

\subsection{Second transition theorem}

\begin{theorem} \label{th_conjugate}
    \begin{enumerate}[(i)]
        \item For any Carter diagram $\widetilde\Gamma$, all $\widetilde\Gamma$-sets are \underline{conjugate} under the Weyl group~$W$.

        \item Let $\{ \widetilde\Gamma, \Gamma \}$ be any homogeneous pair of Carter diagrams, where $\Gamma$ is the Dynkin diagram, and let $\widetilde{S}$  be \underline{any} $\widetilde\Gamma$-set, $S$  be \underline{any} $\Gamma$-set. Then, there exists $F$, the product of transition matrices $M_I(n)$ and some matrices of similarity maps like $L_{\tau_i}$ as in Theorem \ref{th_invol} and Section \ref{sec_prod_trans_matr} and $w \in W$ such that $S = wF\widetilde{S}$.
    \end{enumerate}
\end{theorem}

\begin{proof}
    \begin{enumerate}[(i)]
        \item The Carter diagram $\widetilde\Gamma$ belongs to some homogeneous class $C(\widetilde\Gamma)$.
        Every homogeneous class contains a Dynkin diagram $\Gamma$.  As in Section \ref{sec_prod_trans_matr}, there exists
        the mapping $F$ from any $\widetilde\Gamma$-set $\widetilde{S}$ to some $\Gamma$-set $S$.
        By Theorem \ref{th_invol},
        the mapping $F$ is the product of transition matrices $M_I(n)$ and some matrices of similarity maps like $L_{\tau_i}$.
        
        Let $\widetilde{S}'$ and $\widetilde{S}''$ be any $\widetilde\Gamma$-sets. We will prove that
        $\widetilde{S}'$ and $\widetilde{S}''$ are conjugate under the Weyl group $W$,
        i.e., 
        \begin{equation}
          \label{eq_conj_Carter}
              w\widetilde{S}' = \widetilde{S}'', \quad \text{ for some } \quad w \in W.
        \end{equation}
        There exist $\Gamma$-sets $S'$ and $S''$ such that
        \begin{equation}
          \label{eq_F_matr}
              F\widetilde{S}' = S', \qquad F\widetilde{S}'' = S'',
        \end{equation}
        see Fig. \ref{fig_map_F}. Let $\varPhi$ be the root system associated with $\Gamma$.
        All bases in $\varPhi$ are conjugate, see \cite[Theorem 1.4]{9}.
        Then, there exists $w \in W$, such that $wS' = S''$. By \eqref{eq_F_matr},
        \begin{equation}
          \label{eq_w_F}
             wF\widetilde{S}' = wS' = S'' = F\widetilde{S}''.
        \end{equation}
        By Lemma \ref{lem_commute}, transformations $w$ and $F$ in \eqref{eq_w_F} commute, so
        \begin{equation}
          \label{eq_F_invert}
             Fw\widetilde{S}' = F\widetilde{S}'', \quad  \text{ and } \quad  w\widetilde{S}' = \widetilde{S}''.
        \end{equation}
    
        \item First, by Theorem \ref{th_invol}, we transform $\widetilde{S}$ to some $\Gamma$-set $S'$ by the mapping $F$ as in (i), see Fig.~\ref{fig_map_F}. Further, as in (i), there exists $w \in W$ such that  $wS' = S$. Thus,
        \begin{equation*}
             wF\widetilde{S} = w{S}' = S.
             \qedhere
        \end{equation*}
    \end{enumerate}
\end{proof}

\begin{figure}
    \centering
    \includegraphics[scale=0.3]{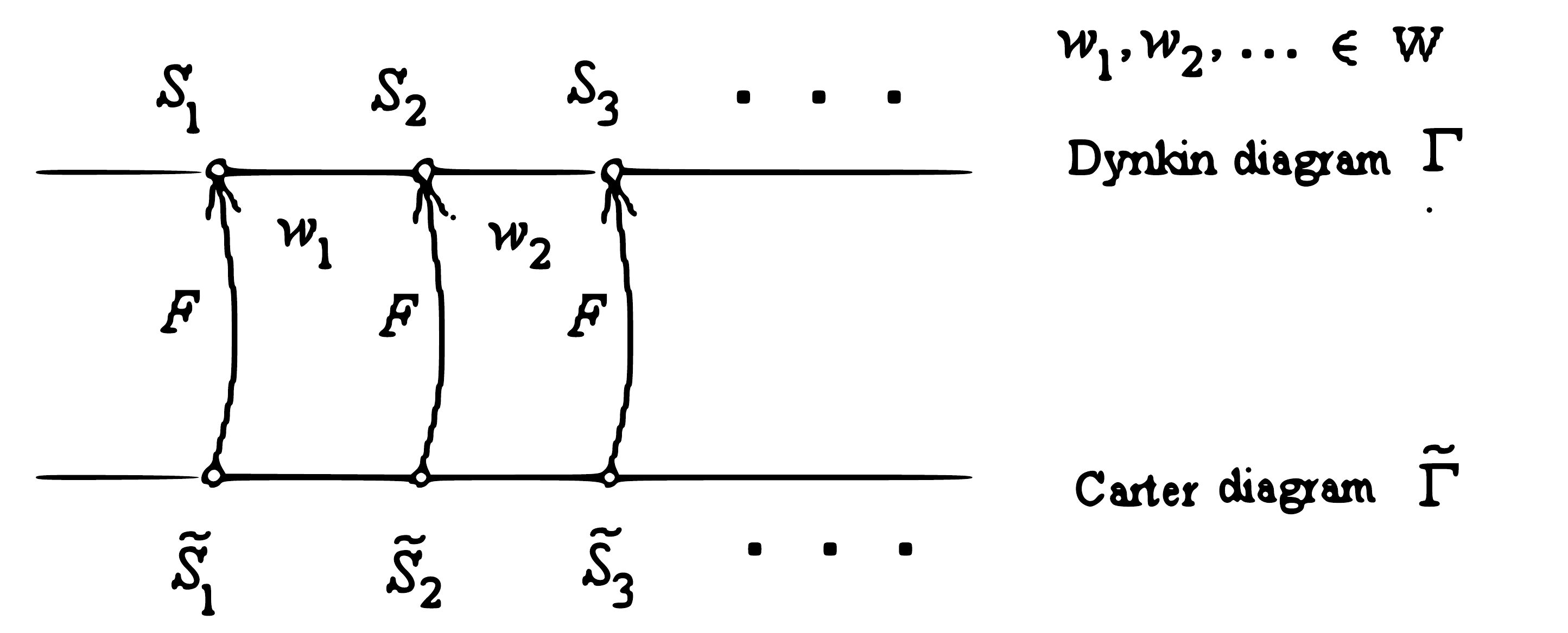}
    \caption{$F\widetilde{S}_i = S_i$, where $\widetilde{S}_i$ (resp. $S_i$), are some $\widetilde\Gamma$-sets (resp. $\Gamma$-sets), $i = 1,2,3,\dots$}
    \label{fig_map_F}
\end{figure}

\subsection{Conjugacy of all $\widetilde\Gamma'$-sets}
As above, $\{\widetilde\Gamma, \Gamma\}$ is a homogeneous pair of Carter diagrams,
$\Gamma$ is the simply-laced Dynkin diagram, $W$ is the Weyl group associated
with $\Gamma$.

Similarly to the fact that all $\Gamma$-bases (where $\Gamma$ is the Dynkin diagram)
are conjugate under the Weyl group \cite[Theorem 1.4]{9},
the same fact holds for Carter diagrams.

\begin{corollary}
   All $\widetilde\Gamma$-sets are conjugate under the Weyl group $W$
associated with the Dynkin diagram $\Gamma$.
\end{corollary}

\begin{proof}
Let $\widetilde{S}$ and $\widetilde{S'}$ be two $\widetilde\Gamma'$-sets.
By Theorem \ref{th_conjugate}, there exist $\Gamma$-bases $S$ and $S'$ and $w \in W$ such that
\begin{equation*}
   F: \widetilde{S} \longmapsto S, \quad
   w: S \longmapsto S', \quad
   F: \widetilde{S'} \longmapsto S'.
\end{equation*}
In other words,
\begin{equation*}
 \begin{split}
   & Fw\widetilde{S} = wF\widetilde{S} = S', \quad F\widetilde{S'} = S', \quad \text{ i.e.,} \\
   & Fw\widetilde{S} = F\widetilde{S'}, \quad \text{ and } \quad w\widetilde{S} = \widetilde{S'}.
    \qedhere
 \end{split}
\end{equation*}
\end{proof}

Let $C(\Gamma)$, as above, be a homogeneous class of Carter diagrams out of \eqref{eq_one_type}, where $\Gamma$ is the Dynkin diagram. Denote by $\nu(\Gamma')$ the number of all $\Gamma'$-sets.

\begin{corollary}
  The number $\nu(\Gamma')$ is the constant number for any homogeneous class.
  For any Carter diagram $\Gamma' \in C(\Gamma)$, the number $\nu(\Gamma')$ coincides with the number of bases
  in the root system associated with $\Gamma$,  and coincides with the number of elements in
  the Weyl group:
    \begin{equation}
      \mid \nu(\Gamma') \mid \quad  = \quad  \mid \nu(\Gamma) \mid \quad = \quad \mid W \mid.
      \qedhere
    \end{equation}
\end{corollary}

\subsection{Proof of Theorem \ref{th_invol}} \label{sect_proof_th}

\begin{enumerate}[(i)]

    \item
    Let us construct the matrix $M_I$ for every pair $\{\widetilde\Gamma, \Gamma \}$ in the adjacency list \eqref{eq_pairs_trans}.
    
    \begin{enumerate}[(1)]
        \item 
        Pair $\{D_4(a_1), D_4\}$: $M_I$ maps $D_4(a_1)$-set to $D_4$-set.
        
        \begin{itemize}
            \item Mapping: $M_I \widetilde\alpha_3 = \alpha_3 = -(\alpha_1  + \alpha_2 + \widetilde\alpha_3)$.

            \item Root system: $S = \{\alpha_2, \alpha_1, \widetilde\alpha_3 \}$ and $\varPhi$ is a root system of type $A_3$.

            \item Minimal root: $\alpha_3$ is the minimal root in $\varPhi$.

            \item Eliminated edges: $\{\widetilde\alpha_3, \alpha_1 \}$ and $\{\widetilde\alpha_3, \alpha_4\}$.

            \item Emerging edge: $\{\alpha_3, \alpha_2 \}$.

            \item Checking relations:
            \[
                \begin{cases}
                    \alpha_3 \perp \alpha_1  \quad & (\alpha_3 \text{ is the minimal root} ) \\
                    \alpha_3 \perp \alpha_4  \quad & (\alpha_2 + \widetilde\alpha_3 \perp \alpha_4 ) \\
                    (\alpha_3, \alpha_2) = -1 \quad & (\alpha_3 \text{ is the minimal root} )
                  \end{cases}
            \]
        \end{itemize}

        \item 
        Pair $\{D_l(a_k), D_l\}$: $M_I$ maps $D_l(a_k)$-set to $D_l$-set.
        
        \begin{itemize}
            \item Mapping:
            \[
                M_I\overline{\tau}_{k+1} =
                \begin{cases}
                    \beta_2 = -(\tau_1 + \tau_2 + \overline{\tau}_2), 
                    & \textup{ for } k = 1, \\
                    \beta_{k+1} = -(\tau_1 + 2\sum_{i=2}^k \tau_i + \tau_{k+1} + \overline{\tau}_{k+1}),
                    & \textup{ for } k \ge 2.
                \end{cases}
            \]

            \item Root systems:
            \begin{gather*}
                \begin{cases}
                    S_1 = \{\tau_1, \dots, \tau_k, \tau_{k+1}, \overline{\tau}_{k+1} \},
                    & \textup{for } k \ge 2, \\
                    S_2 = \{\tau_1, \tau_2, \overline{\tau}_2 \},
                    & \textup{for } k = 1,
               \end{cases}
               \\
                \begin{cases}
                    \varPhi(S_1) \textup{ is a root system of type } D_{k+2}, 
                    & \textup{for } k \ge 2, \\
                    \varPhi(S_2) \textup{ is a root system of type } A_3,
                    & \textup{for } k = 1.
               \end{cases}
            \end{gather*}
                
            \item Minimal roots: $\beta_{k+1}$ is the minimal root in $\varPhi(S_1)$ and $\beta_2$ is the minimal root in $\varPhi(S_2)$.
                
            \item Eliminated edges: $\{\overline{\tau}_{k+1}, \tau_k \}$ and $\{\overline{\tau}_{k+1}, \tau_{k+2}\}$ .

            \item Emerging edge: $\{\beta_{k+1}, \tau_2 \}$.

\enlargethispage{.62cm}

            \item Checking relations: for $k \ge 2$,
            \[\begin{cases}
                \beta_{k+1} \perp \tau_i  \ (1 \leq i \leq k+1, i \neq 2) 
                 & (\beta_{k+1} \textup{ is the minimal root\footnotemark[1]})
                \\
                \beta_{k+1} \perp \tau_{k+2} 
                 & ((\beta_{k+1}, \tau_{k+2}) = (\tau_{k+1} + \overline{\tau}_{k+1}, \tau_{k+2}) = 0) 
                \\
                \beta_{k+1} \perp \tau_i \ (i > k+2)
                 & (\textup{disconnected\footnotemark[2]}) \\
                (\beta_{k+1}, \tau_2) = -1
                 & (\beta_{k+1} \textup{ is the minimal root}),
                \end{cases}
            \]
            \footnotetext[1]{Hereinafter, the reason for the relation is indicated in parentheses.}
            \footnotetext[2]{Let $\alpha$ be the sum of several roots $\alpha_i$: $\alpha = \sum\alpha_i$. Hereinafter, the line ``$\alpha \perp \beta$ (disconnected)'' means the case, where each summand $\alpha_i$ in $\alpha$ is orthogonal to $\beta$.}
            and for $k = 1$,
            \[\begin{cases}
                \beta_2 \perp \tau_1, \quad (\beta_2, \tau_2) = -1
                 & (\beta_2 \text{ is the minimal root}) \\
                \beta_2 \perp \tau_i \quad (i > 3)
                 & (\textup{disconnected}).
                \end{cases}
            \]
            
        \end{itemize}

        \item 
        Pair $\{E_6(a_1), E_6\}$: $M_I$ maps $E_6(a_1)$-set to $E_6$-set.
    
        \begin{itemize}
            \item Mapping: $M_I\widetilde\beta_3 = \beta_3 = -(\alpha_1 + \beta_1  + \alpha_3 + \widetilde\beta_3)$.
            
            \item Root system: $S = \{\alpha_1, \beta_1, \alpha_3, \widetilde\beta_3\}$ and $\varPhi$ is a root system of type $A_4$.
            
            \item Minimal root: $\beta_3$ is the minimal root in $\varPhi$.
            
            \item Eliminated edges: $\{\widetilde\beta_3, \alpha_3\}$, $\{\widetilde\beta_3, \alpha_2\}$.
            
            \item Emerging edge: $\{\beta_3, \alpha_1\}$.
            
            \item Checking relations:
            \[\begin{cases}
                \beta_3 \perp \alpha_3, \beta_1 & (\beta_3 \text{ is the minimal root}) \\
                \beta_3 \perp \alpha_2 & (\widetilde\beta_3 + \beta_1 \perp \alpha_2) \\
                \beta_3 \perp \beta_2 & (\text{disconnected}) \\
                (\beta_3, \alpha_1) = -1 & (\beta_3 \text{ is the minimal root})
            \end{cases}\]
        \end{itemize}

        \begin{figure}
            \centering
            \includegraphics[scale=0.45]{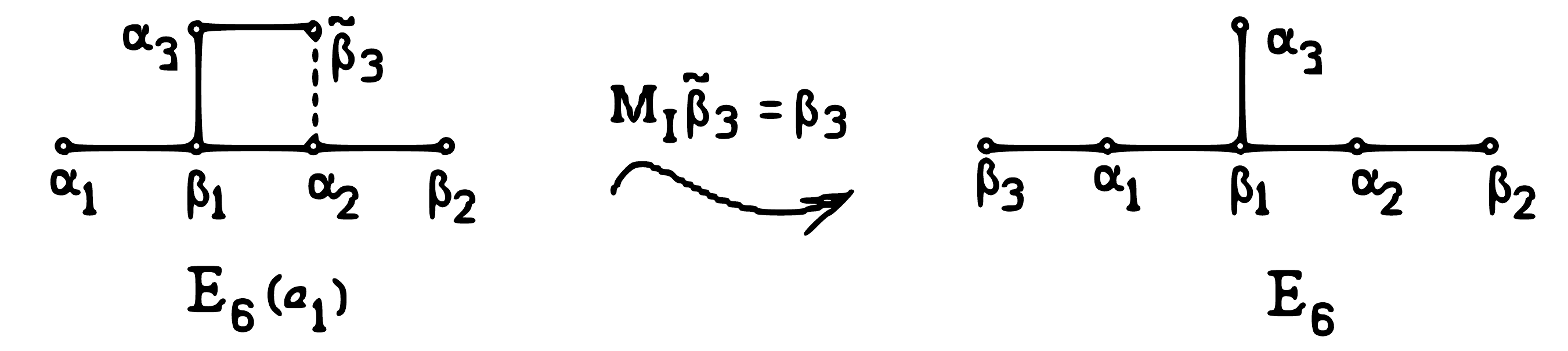}
            \caption{Case (3). Mapping $M_I: E_6(a_1) \longmapsto E_6$.}
            \label{map_E6a1_to_E6}
        \end{figure}

        \item 
        Pair $\{E_6(a_2), E_6(a_1)\}$: $M_I$ maps $E_6(a_2)$-set to $E_6(a_1)$-set.

        \begin{itemize}
            \item Mapping: $M_I\widetilde\beta_2 = \beta_2 = -(\alpha_3 + \beta_1 + \alpha_2 + \widetilde\beta_2)$.
            
            \item Root systems: $S = \{\alpha_3, \beta_1, \alpha_2, \widetilde\beta_2 \}$ and $\varPhi(S)$ is a root system of type $A_4$.
            
            \item Minimal roots: $\beta_2$ is the minimal root in $\varPhi(S)$.
            
            \item Eliminated edges: $\{\widetilde\beta_2, \alpha_3 \}$ and $\{\widetilde\beta_2, \alpha_1\}$.
            
            \item Emerging edges: $\{\beta_2, \alpha_2 \}$.
            
            \item Checking relations:
            \[\begin{cases}
               \beta_2 \perp \alpha_3, \beta_1 \quad & (\beta_2 \text{ -- minimal root}) \\
               \beta_2 \perp \alpha_1, \widetilde\beta_3
               \quad & (\widetilde\beta_2 + \beta_1 \perp \alpha_1 \text{ and }  \alpha_2 + \alpha_3 \perp \widetilde\beta_3 ) \\
               (\beta_2, \alpha_2) = -1 \quad & (\beta_2 \text{ -- minimal root})
         \end{cases}\]
        \end{itemize}

        \begin{figure}[!h]
            \centering
            \includegraphics[scale=0.45]{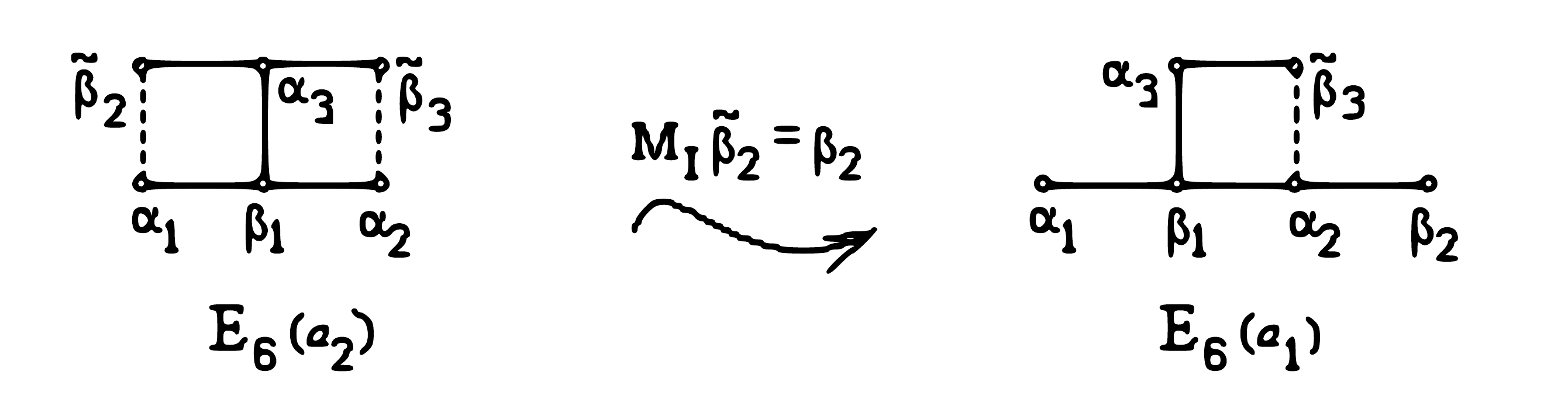}
            \caption{ Case (4). Mapping $M_I: E_6(a_2) \longmapsto E_6(a_1)$.}
            \label{map_E6a2_to_E6a1}
        \end{figure}

        \begin{table}
            \centering
            \begin{tabular}{|c|c|c|}
                                \hline
                 $\begin{array}{c}
                   {\bf D_4} \\
                 \end{array}$
                 &
                 $\begin{array}{c}
                  \includegraphics[scale=0.3]{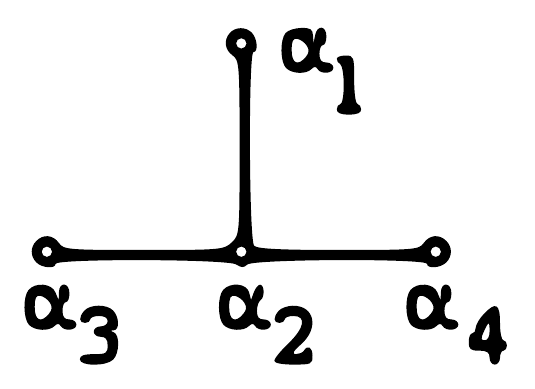} \\
                  \end{array}$       
                  &
                  $\begin{array}{c}
                  S = \{\alpha_1, \alpha_2, \alpha_3, \alpha_4 \} \\
                  \end{array}$
                  \\ \hline
                 $\begin{array}{c}
                 \\
                 \footnotesize{(1)}
                 \\
                 \\
                 \footnotesize{\bf D_4(a_1)}
                 \\
                 \end{array}$
                 &
                 $\begin{array}{c}
                   \\
                   \footnotesize{\{\widetilde\Gamma, \Gamma\} = \{D_4(a_1), D_4\}} \\ \\
                  \includegraphics[scale=0.3]{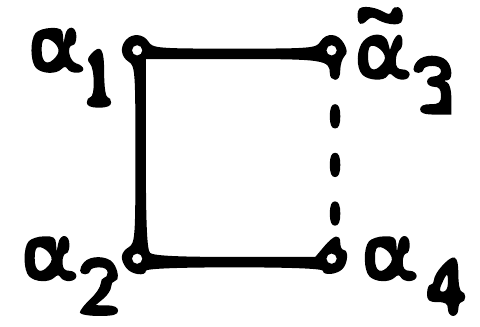}
                  \end{array}$       
                  &
                  \footnotesize
                  $\begin{array}{l}
                  \\
                   \widetilde{S} = \{\alpha_1, \alpha_2, \widetilde\alpha_3, \alpha_4\}, \\
                   M_I\tau = \tau \text{ for } \tau \in  \{ \alpha_1, \alpha_2, \alpha_4\}, \\
                   M_I\widetilde\alpha_3 = \alpha_3 = -(\alpha_1 + \alpha_2 + \widetilde\alpha_3)
                   \\
                   \end{array}$ 
                \\ \hline
                 \footnotesize
                 $\begin{array}{c}
                   {\bf D_l} \\
                 \end{array}$
                 &
                  $\begin{array}{c}
                  \includegraphics[scale=0.3]{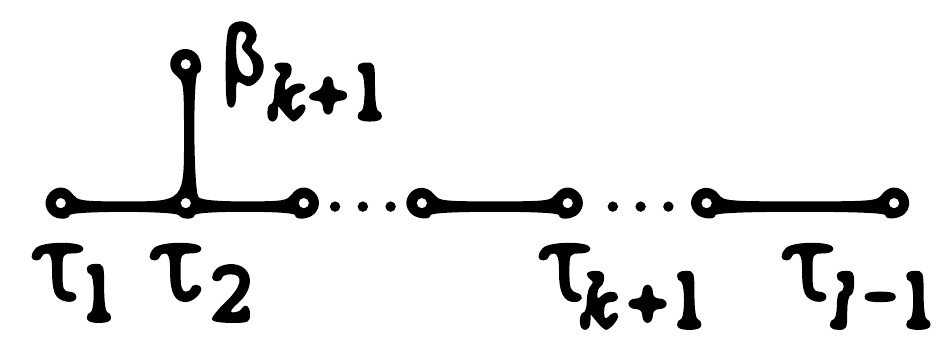} \\
                  \end{array}$  &
                  $\begin{array}{c}
                     S = \{\tau_1, \tau_2, \dots, \tau_{l-1}, \beta_{k+1}\} \\
                  \end{array}$
                     \\
                \hline
                 $\begin{array}{c}
                 \\
                 \footnotesize{(2)}
                 \\
                  \\
                   \footnotesize{\bf D_l(a_k)} \\
                  \\
                 \end{array}$
                 &
                 $\begin{array}{c}
                   \{\widetilde\Gamma, \Gamma\} = \{D_l(a_k), D_l\} \\
                   \\
                  \includegraphics[scale=0.3]{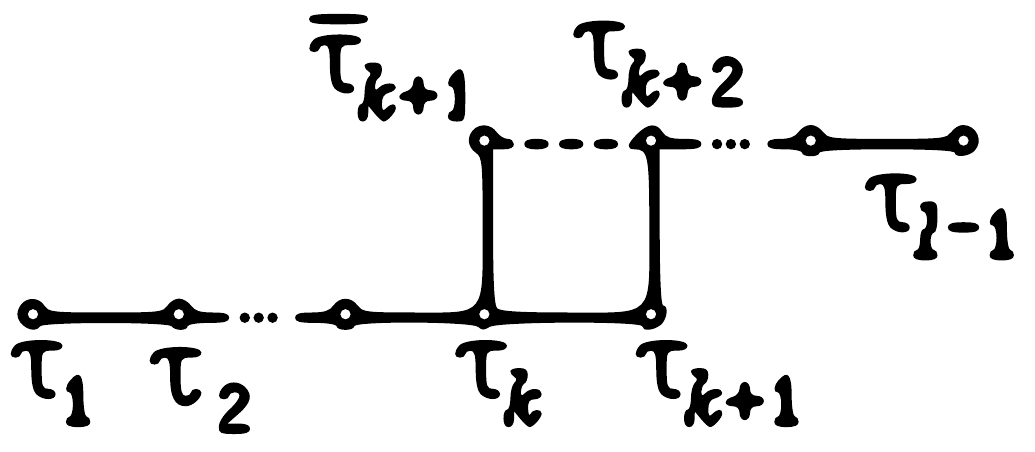} \\
                  \end{array}$  &
                 $\begin{array}{l}
                  \\
                  \widetilde{S} = \{\tau_1, \dots, \tau_k, \tau_{k+1},
                       \overline\tau_{k+1}, \tau_{k+2}\dots, \tau_{l-1}\},  \\
                   M_I\tau = \tau \text{ for } \tau \in  \{ \tau_1, \tau_2, \dots, \tau_{l-1}\}, \\
                   M_I\overline{\tau}_{k+1} = \\
                      \begin{cases}
                         \beta_{k+1} =
                           -(\tau_1 + 2\sum\limits_{i=2}^k\tau_i +  \tau_{k+1} + \overline{\tau}_{k+1}), \\
                           \qquad\qquad\qquad\qquad\qquad k \geq 2, \\
                          \beta_2 = -(\tau_1 + \tau_2 + \overline{\tau}_2).
                      \end{cases}
                      \\
                 \end{array}$ \\
                 & & \\
                \hline
                $\begin{array}{c}
                {\bf E_6} \\
                \end{array}$
                &
                $\begin{array}{c}
                \includegraphics[scale=0.3]{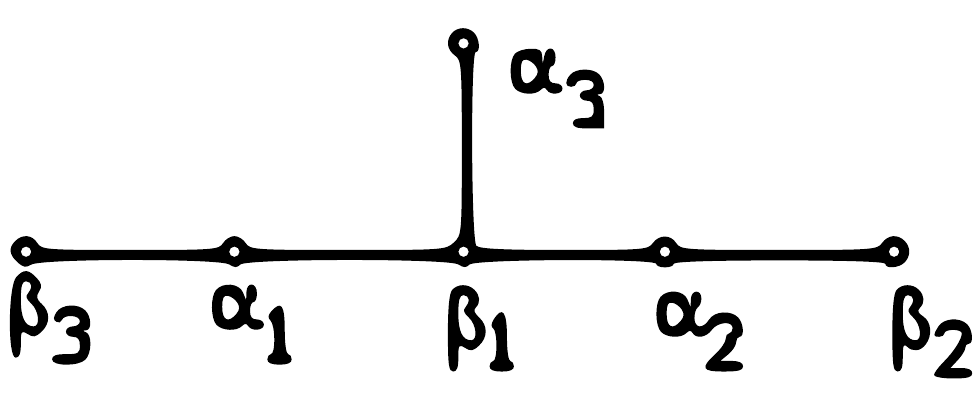} \\
                \end{array}$ &
                $\begin{array}{c}
                \\
                S = \{\alpha_1, \alpha_2, \alpha_3, \beta_1, \beta_2, \beta_3 \}  \\ \\
                \end{array}$
                \\
                \hline
                $\begin{array}{c}
                 \\
                 \footnotesize{(3)}
                 \\
                \\
                \footnotesize{\bf E_6(a_1)}
                \\
                \end{array}$
                &
                $\begin{array}{c}
                \\
                \footnotesize{\{\widetilde\Gamma, \Gamma\} = \{E_6(a_1), E_6\}}
                \\ \\
                \includegraphics[scale=0.3]{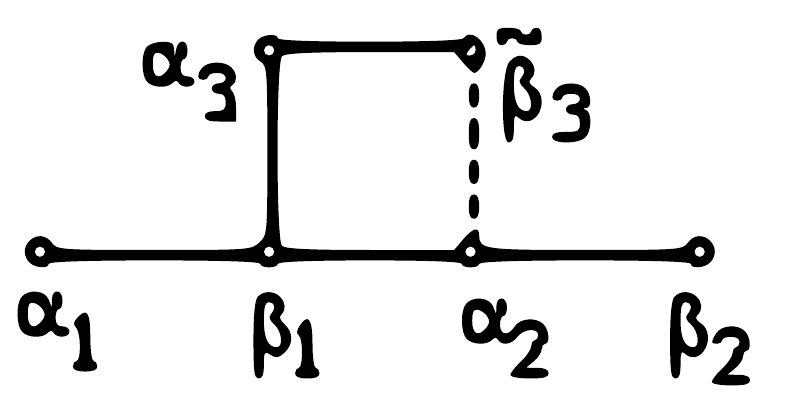} \\
                \end{array}$ &
                $\begin{array}{l}
                \widetilde{S} = \{\alpha_1, \alpha_2, \alpha_3, \beta_1, \beta_2, \widetilde\beta_3 \}, \\
                M_I\tau = \tau \text{ for } \tau \in \{\alpha_1, \alpha_2, \alpha_3, \beta_1, \beta_2\}, \\
                M_I\widetilde\beta_3 = \beta_3 = -(\alpha_1 + \beta_1  + \alpha_3 + \widetilde\beta_3)
                \end{array}$  \\
                \hline
                $\begin{array}{c}
                 \\
                 \footnotesize{(4)}
                 \\
                \\
                \footnotesize{\bf E_6(a_2)}
                \\
                \end{array}$
                &
                $\begin{array}{c}
                \\
                \footnotesize{\{\widetilde\Gamma, \Gamma\} = \{E_6(a_2), E_6(a_1) \}} \\ \\
                \includegraphics[scale=0.3]{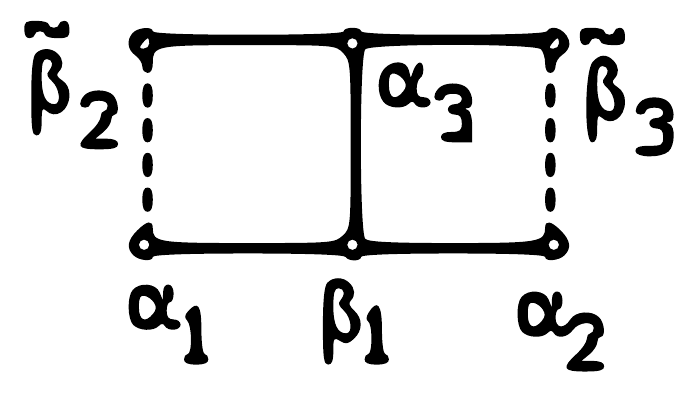} \\
                \end{array}$ &
                $\begin{array}{l}
                \\
                \widetilde{S} = \{\alpha_1, \alpha_2, \alpha_3, \beta_1, \widetilde\beta_2, \widetilde\beta_3 \}, \\
                M_I\tau = \tau \text{ for } \tau \in  \{\alpha_1, \alpha_2, \alpha_3, \beta_1, \widetilde\beta_3 \} \\
                M_I\widetilde\beta_2 = \beta_2 = -(\alpha_3 + \beta_1  + \alpha_2 + \widetilde\beta_2)
                \\
                \end{array}$
                \\
                \hline
            \end{tabular}
            \caption{$\{D_4(a_1), D4\}$, $\{D_l(a_k), D_l\}$, $\{E_6(a_1), E_6\}$, $\{E_6(a_2), E_6(a_1)\}$}
            \label{table:E6.MI2}
        \end{table}

        \item
        Pair $\{E_7(a_1), E_7\}$: $M_I$ maps $E_7(a_1)$-set to $E_7$-set.

        \begin{itemize}
          \item Mapping: $M_I\widetilde\alpha_3 = \alpha_3 = -(\beta_2 + \alpha_2 + \beta_3 + \widetilde\alpha_3)$.
        
          \item Root system:
              $S = \{\beta_2, \alpha_2, \beta_3, \widetilde\alpha_3\}$ and $\varPhi$ is a root system of type $A_4$.
        
          \item Minimal root: $\alpha_3$ is the minimal root in $\varPhi$.
        
          \item Eliminated edges: $\{\widetilde\alpha_3, \beta_1\}$ and $\{\widetilde\alpha_3, \beta_3\}$.
          
          \item Emerging edge: $\{\beta_2, \alpha_3\}$.
        
          \item Checking relations:
              \[\begin{cases}
                  \alpha_3 \perp \alpha_2, \beta_3 \quad & (\alpha_3 \text{ is the minimal root}) \\
                  \alpha_3 \perp \beta_1 \quad & (\widetilde\alpha_3 + \alpha_2 \perp \beta_1 ) \\
                  \alpha_3 \perp \alpha_1, \beta_4 \quad & (\text{disconnected}) \\
                  (\alpha_3, \beta_2) = -1 \quad & (\alpha_3 \text{ is the minimal root})
              \end{cases}\]
        \end{itemize}

        \begin{figure}
            \centering
            \includegraphics[scale=0.45]{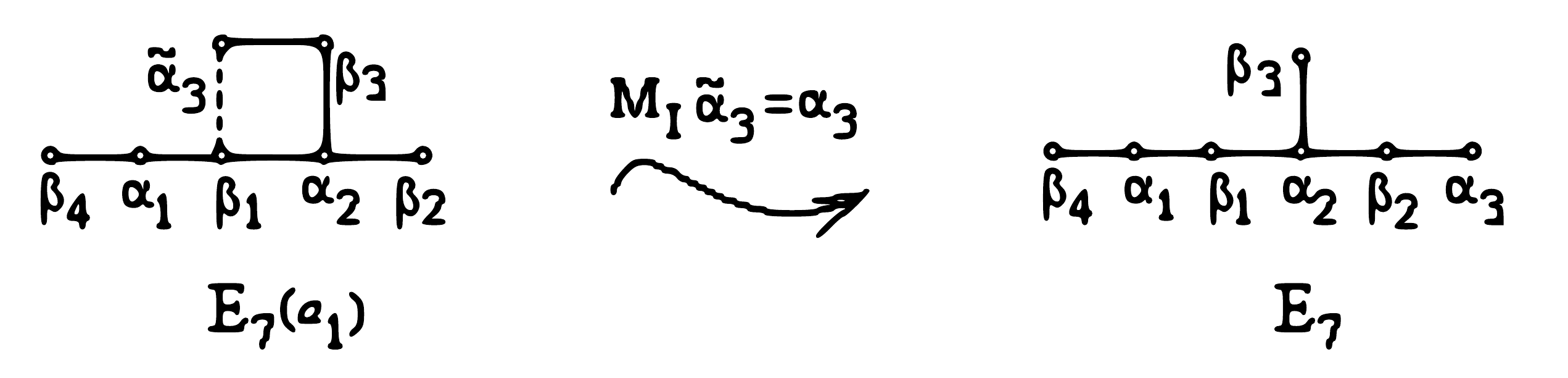}
            \caption{ Case (5). Mapping $M_I: E_7(a_1) \longmapsto E_7$.}
            \label{map_E7a1_to_E7}
        \end{figure}

        \item
        Pair $\{E_7(a_2), E_7\}$: $M_I$ maps $E_7(a_2)$-set to $E_7$-set.

        \begin{itemize}
            \item Mapping: $M_I\widetilde\alpha_1 = \alpha_1 = -(\beta_1 + \alpha_2 + \beta_3 + \widetilde\alpha_1)$.
            
            \item Root system: $S = {\beta_1, \alpha_2, \beta_3, \widetilde\alpha_1}$ and $\varPhi$ is a root system of type $A_4$.
            
            \item Minimal root: $\alpha_1$ is the minimal root in $\varPhi$.
            
            \item Eliminated edges: ${\widetilde\alpha_1, \beta_3}$, ${\widetilde\alpha_1, \beta_2}$ and
            ${\widetilde\alpha_1, \widetilde\beta_4}$.
            
            \item Emerging edges: ${\alpha_1, \beta_1}$ and ${\alpha_1, \widetilde\beta_4}$.
            
            \item Checking relations:
            \[\begin{cases}
                \alpha_1 \perp \alpha_2, \beta_3 & (\alpha_1 \text{ is the minimal root}) \\
                \alpha_1 \perp \beta_2 & (\alpha_2 + \widetilde\alpha_1 \perp \beta_2) \\
                \alpha_1 \perp \alpha_3 & (\text{disconnected}) \\
                (\alpha_1, \beta_1) = -1 & (\alpha_1 \text{ is the minimal root}) \\
                (\beta_4, \alpha_1) = -1 & ((\beta_4, \alpha_1) = (\beta_4, \widetilde\alpha_1))
            \end{cases}\]
        \end{itemize}

        \begin{figure}[!h]
            \centering
            \includegraphics[scale=0.45]{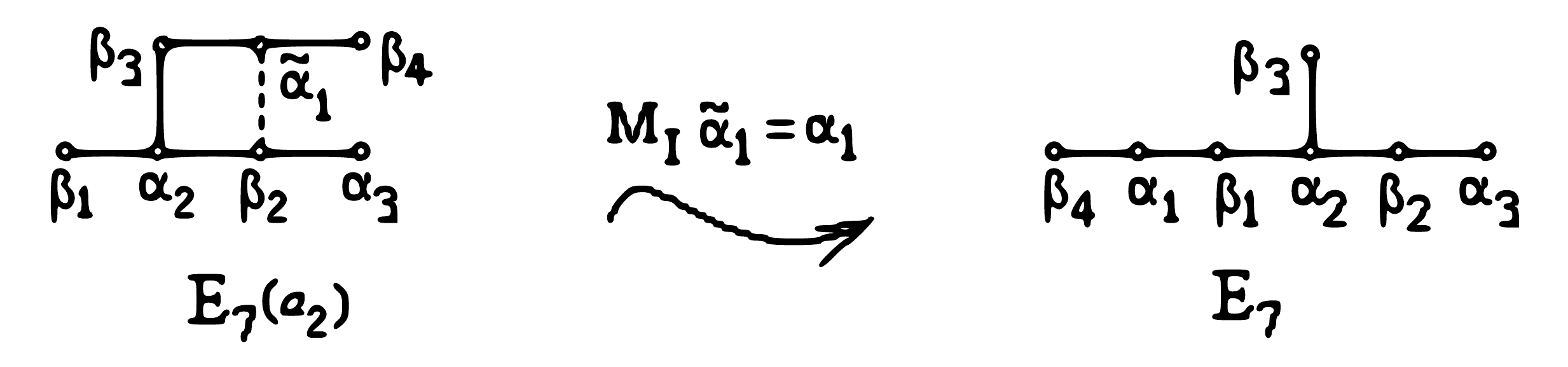}
            \caption{ Case (6). Mapping $M_I: E_7(a_2) \longmapsto E_7$.}
            \label{map_E7a2_to_E7}
        \end{figure}

        \item
        Pair $\{E_7(a_3), E_7(a_1)\}$: $M_I$ maps $E_7(a_3)$-set to $E_7(a_1)$-set.

        \begin{itemize}
            \item Mapping: $M_I\alpha_4 = \beta_4 = -(\alpha_4 + \beta_3 + \alpha_2 + \beta_1 + \alpha_1)$.
            
            \item Root system: $S = \{\alpha_4, \beta_3, \alpha_2, \beta_1, \alpha_1 \}$ and $\varPhi(S)$ is a root system of type $A_5$.
            
            \item Minimal root: $\beta_4$ is the minimal root in $\varPhi(S)$.
            
            \item Eliminated edges: $\{\alpha_4, \beta_3 \}$ and $\{\alpha_4, \beta_2\}$.
            
            \item Emerging edge: $\{\beta_4, \alpha_1\}$.
            
            \item Checking relations:
            \[\begin{cases}
              \beta_4 \perp \beta_1, \alpha_2, \beta_3 \quad & (\beta_4 \text{ is the minimal root}) \\
              \beta_4 \perp \widetilde\alpha_3, \beta_2 \quad & (\beta_1 + \beta_3 \perp  \widetilde\alpha_3 \text{ and } \alpha_2 + \alpha_4 \perp \beta_2) \\
              (\beta_4, \alpha_1) = -1 \quad  & (\beta_4 \text{ is the minimal root})
           \end{cases}\]
        \end{itemize}

       \begin{figure}
            \centering
            \includegraphics[scale=0.45]{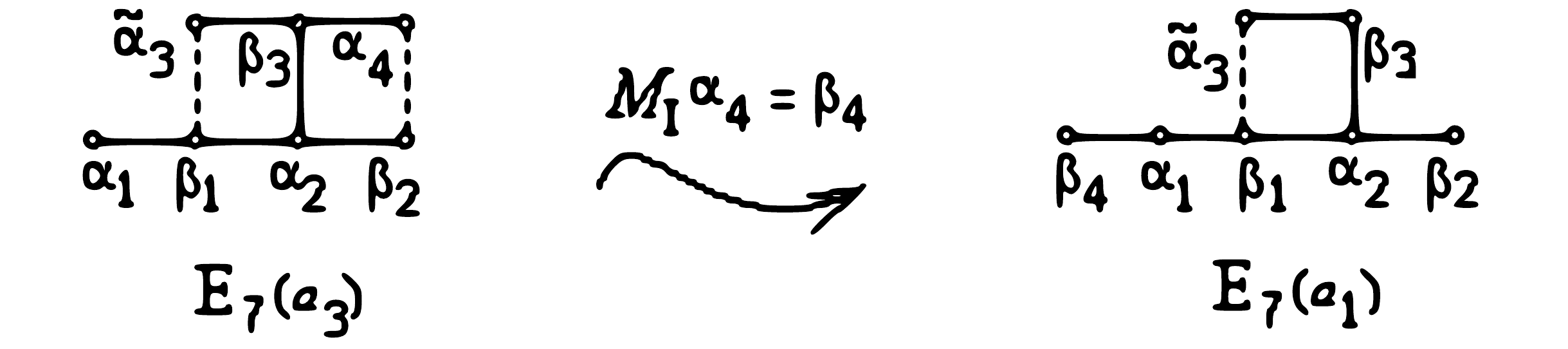}
            \caption{ Case (7). Mapping $M_I: E_7(a_3) \longmapsto E_7(a_1)$. }
            \label{map_E7a3_to_E7a1}
        \end{figure}

        \item
        Pair $\{E_7(a_4), E_7(a_3)\}$: $M_I$ maps $E_7(a_4)$-set to $E_7(a_3)$-set.

        \begin{itemize}
              \item Mapping: $M_I\widetilde\alpha_1 = \alpha_1 = -(2\beta_1 + \widetilde\alpha_1 + \alpha_2 + \widetilde\alpha_3)$.
            
              \item Root system: $S = \{\beta_1, \widetilde\alpha_1, \alpha_2, \alpha_3 \}$ and $\varPhi(S)$ is a root system of type $D_4$.
            
              \item Minimal root: $\alpha_1$ is the minimal root in $\varPhi(S)$.
            
              \item Eliminated edges: $\{\widetilde\alpha_1, \beta_1 \}$ and $\{\widetilde\alpha_1, \beta_2\}$.
            
              \item Emerging edge: $\{\alpha_1, \beta_1\}$.
            
              \item Checking relations:
              \[\begin{cases}
                  \alpha_1 \perp  \alpha_2, \widetilde\alpha_3 \quad  & (\beta_1 \text{ is the minimal root}) \\
                  \alpha_1 \perp \beta_3, \beta_2 \quad & (\alpha_2 + \widetilde\alpha_3 \perp \beta_3 \text{ and } \widetilde\alpha_1 + \alpha_2 \perp \beta_2) \\
                  \alpha_1 \perp \alpha_4 \quad & (\alpha_4 \perp \widetilde\alpha_1,  \widetilde\alpha_3, \alpha_2, \beta_1) \\
                  (\alpha_1, \beta_1) = -1 \quad &(\beta_1 \text{ is the minimal root})
              \end{cases}\]
        \end{itemize}

        \begin{figure}[!h]
            \centering
            \includegraphics[scale=0.4]{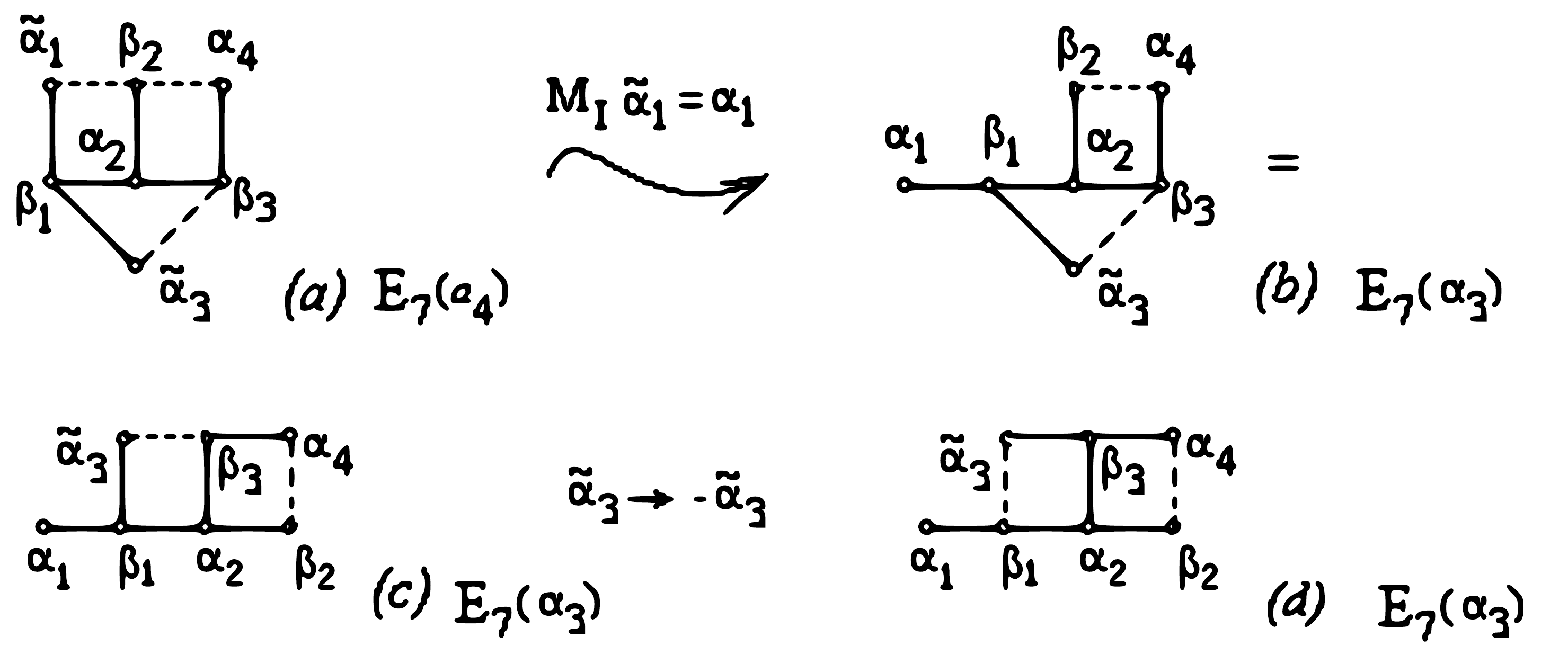}
            \caption{ Case (8). Mapping $M_I: E_7(a_4) \longmapsto E_7(a_3)$ (diagram (c)). Diagrams (c) and (d)  are equivalent
              (by similarity $\widetilde\alpha_3 \longmapsto -\widetilde\alpha_3$) }
            \label{map_E7a4}
        \end{figure}

        \begin{table}
            \centering
            \caption{$\{ E_7(a_1), E_7\}$, $\{ E_7(a_2), E_7\}$,  $\{ E_7(a_3), E_7(a_1)\}$, $\{ E_7(a_4), E_7(a_3)\}$}
            \label{tab_part_root_syst_2}
            \begin{tabular} {|c|c|c|}
            \hline
               & & \\
               & $\widetilde\Gamma$-basis $\widetilde{S}$ and $\Gamma$-basis $S$
                         & Mapping $M_I$ \\
               \hline
              $\begin{array}{c}
               \footnotesize{\bf E_7} \\
             \end{array}$
              &
              $\begin{array}{c}
               \\
               \includegraphics[scale=0.3]{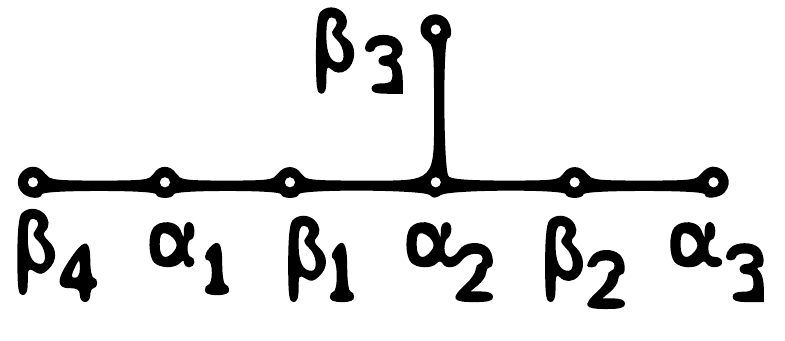}  \\
               \end{array}$
               &
               $\begin{array}{c}
               \\
               S = \{\alpha_1, \alpha_2, \alpha_3, \beta_1, \beta_2, \beta_3, \beta_4 \} \\ \\
               \end{array}$
              \\
            \hline
              $\begin{array}{c}
              \\
              \footnotesize{(5)}
              \\
              \\
               \footnotesize{\bf E_7(a_1)}
              \\
             \end{array}$
              &
              $\begin{array}{c}
                 \\
                 \{\widetilde\Gamma, \Gamma\} = \{E_7(a_1), E_7\} \\ \\
               \includegraphics[scale=0.3]{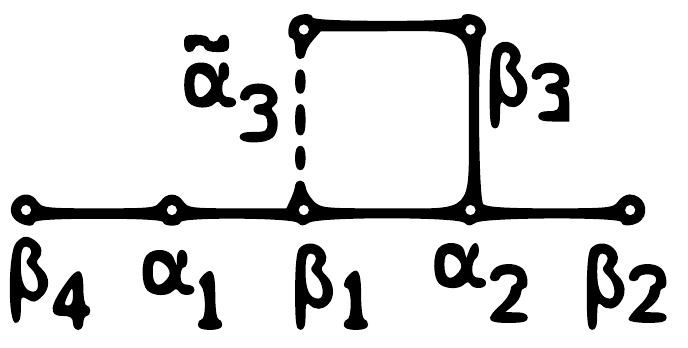}  \\
               \end{array}$
                &
                 $\begin{array}{l}
                 \\
                 \widetilde{S} = \{\alpha_1, \alpha_2, \widetilde\alpha_3, \beta_1, \beta_2, \beta_3, \beta_4 \}, \\
                 M_I\tau = \tau \text{ for } \tau \in \{\alpha_1, \alpha_2, \beta_1, \beta_2, \beta_3, \beta_4\}, \\
                 M_I\widetilde\alpha_3 = \alpha_3 = -(\beta_2 + \alpha_2 + \beta_3 + \widetilde\alpha_3) 
                 \\
                 \end{array}$
                 \\
              \hline
              $\begin{array}{c}
              \\
              \footnotesize{(6)}
              \\
              \\
               \footnotesize{\bf E_7(a_2)}
               \\
             \end{array}$
              &
              $\begin{array}{c}
                 \\
                 \{\widetilde\Gamma, \Gamma\} = \{E_7(a_2), E_7\} \\ \\
              \includegraphics[scale=0.35]{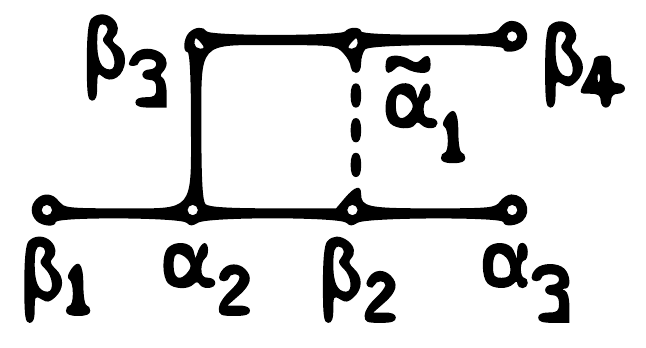} \\
              \end{array}$
              &
               $\begin{array}{l}
                  \\
                 \widetilde{S} = \{\widetilde\alpha_1, \alpha_2, \alpha_3, \beta_1, \beta_2, \beta_3, \beta_4 \},  \\
                 M_I\tau = \tau \text{ for } \tau \in \{\alpha_2, \alpha_3, \beta_1, \beta_2, \beta_3, \beta_4\} \\
                 M_I\widetilde\alpha_1 = \alpha_1 = -(\beta_1 + \alpha_2 + \beta_3 + \widetilde\alpha_1) 
                 \\
                 \end{array}$
                  \\
             \hline
             $\begin{array}{c}
              \\
              \footnotesize{(7)}
              \\
              \\
               \footnotesize{\bf E_7(a_3)}
               \\
             \end{array}$
              &
            $\begin{array}{c}
              \\
              \{\widetilde\Gamma, \Gamma\} = \{E_7(a_3), E_7(a_1)\} \\
              \\
              \includegraphics[scale=0.45]{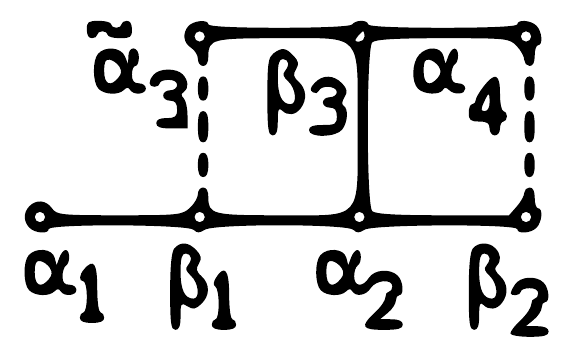} \\
              \end{array}$
              &
               $\begin{array}{l}
               \\
              \widetilde{S} = \{\alpha_1, \alpha_2, \widetilde\alpha_3, \alpha_4, \beta_1, \beta_2, \beta_3 \}, \\
               M_I\tau = \tau \text{ for } \tau \in \{ \alpha_1, \alpha_2, \widetilde\alpha_3,
                        \beta_1, \beta_2, \beta_3\}, \\
               M_I\alpha_4 = \beta_4 =  -(\alpha_4 + \beta_3 + \alpha_2 + \beta_1 + \alpha_1)
               \\
                 \end{array}$   \\
              \hline
              $\begin{array}{c}
              \\
              \footnotesize{(8)}
              \\
              \\
               \footnotesize{\bf E_7(a_4)}
               \\
             \end{array}$
              &
             $\begin{array}{c}
               \\
                 \{\widetilde\Gamma, \Gamma\} = \{E_7(a_4), E_7(a_3)\} \\
               \\
              \includegraphics[scale=0.35]{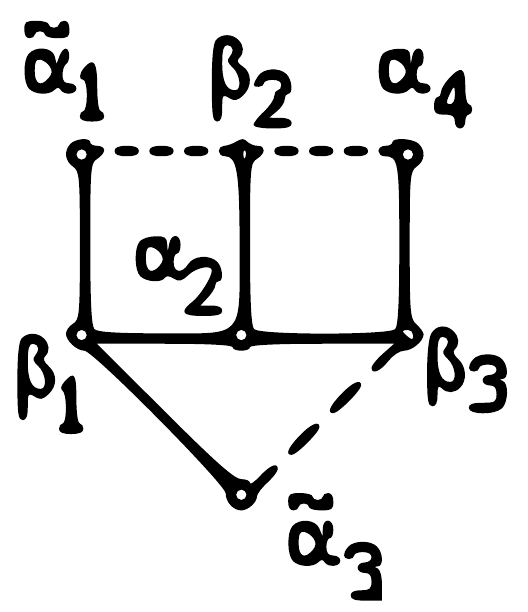} \\
              \end{array}$
               &
                $\begin{array}{l}
                \\
                \widetilde{S} = \{\widetilde\alpha_1, \alpha_2, \widetilde\alpha_3, \alpha_4,
                                    \beta_1, \beta_2, \beta_3 \},  \\
                 M_I\tau = \tau \text{ for } \tau \in \{\alpha_2, \widetilde\alpha_3,\alpha_4, \beta_1, \beta_2, \beta_3\}, \\
                 M_I\widetilde\alpha_1 = \alpha_1 = -(2\beta_1 + \widetilde\alpha_1 + \alpha_2 + \widetilde\alpha_3)
                 \\
                 \end{array}$
                 \\
              \hline
            \end{tabular}
        \end{table}

        \item
        Pair $\{E_8(a_1), E_8\}$: $M_I$ maps $E_8(a_1)$-set to $E_8$-set.

        \begin{itemize}
            \item Mapping: $M_I\widetilde\alpha_3 = \alpha_3 = -(\widetilde\alpha_3 + \beta_3 + \alpha_2 + \beta_2)$.
            
            \item Root system: $S = \{\widetilde\alpha_3, \beta_3, \alpha_2, \beta_2\}$ and $\varPhi$ is a root system of type $A_4$.
            
            \item Minimal root: $\alpha_3$ is the minimal root in $\varPhi$.
            
            \item Eliminated edges: $\{\widetilde\alpha_3, \beta_3\}$ and $\{\widetilde\alpha_3, \widetilde\beta_3\}$.
            
            \item Emerging edge: $\{\beta_2, \alpha_3\}$.
            
            \item Checking relations:
            \[\begin{cases}
                \alpha_3 \perp \alpha_2, \beta_3 & (\alpha_3 \text{ is the minimal root}) \\
                \alpha_3 \perp \beta_1 & (\widetilde\alpha_3 + \alpha_2 \perp \beta_1) \\
                \alpha_3 \perp \alpha_1, \alpha_4, \beta_1 & (\text{disconnected}) \\
                (\alpha_3, \beta_2) = -1 & (\alpha_3 \text{ is the minimal root})
            \end{cases}\]
        \end{itemize}

        \begin{figure}[!h]
            \centering
            \includegraphics[scale=0.45]{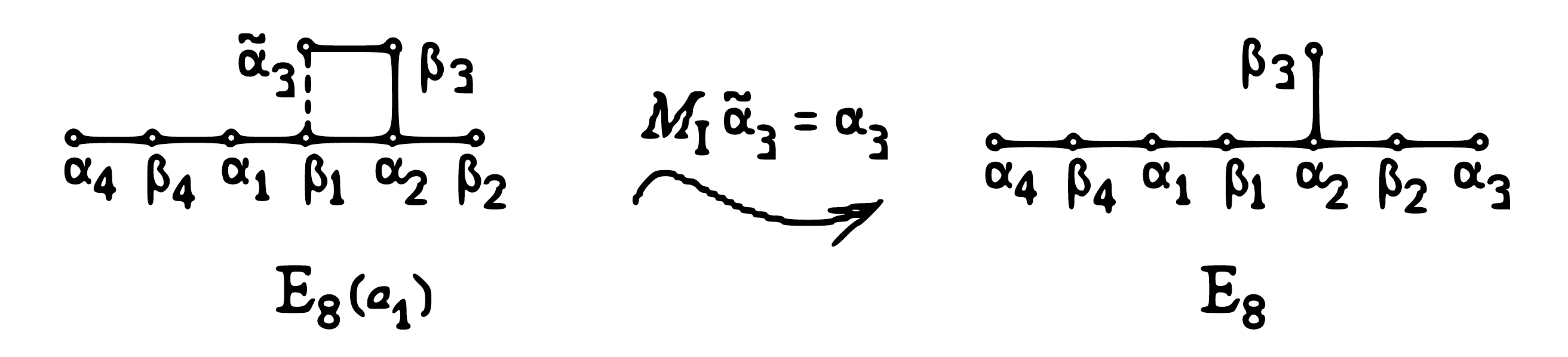}
            \caption{ Case (9). Mapping $M_I: E_8(a_1) \longmapsto E_8$. }
            \label{map_E8a1_to_E8}
        \end{figure}

        \item
        Pair $\{E_8(a_2), E_8\}$: $M_I$ maps $E_8(a_2)$-set to $E_8$-set.
        
        \begin{itemize}
            \item Mapping: $M_I\widetilde\alpha_1 = \alpha_1 = -(\widetilde\alpha_1 + \beta_3 + \alpha_2 + \beta_1)$.

            \item Root system: $S = \{\widetilde\alpha_1, \beta_3, \alpha_2, \beta_1 \}$ and $\varPhi$ is a root system of type $A_4$.

            \item Minimal root: $\alpha_1$ is the minimal root in $\varPhi$.

            \item Eliminated edges: $\{\beta_3, \widetilde\alpha_1\}$, $\{\widetilde\alpha_1, \beta_4\}$ and $\{\widetilde\alpha_1, \beta_2 \}$.

            \item Emerging edges: $\{\alpha_1, \beta_1\}$ and $\{\alpha_1, \beta_4\}$.

            \item Checking relations:
                \[\begin{cases}
                    \alpha_1 \perp \alpha_2, \beta_3 & (\alpha_1 \text{ is the minimal root}) \\
                    \alpha_1 \perp \beta_2 & (\widetilde\alpha_1 + \alpha_2 \perp \beta_2) \\
                    \alpha_1 \perp \alpha_4 &  (\text{disconnected}) \\
                    (\alpha_1, \beta_1) = -1 & (\alpha_1 \text{ is the minimal root}) \\
                    (\beta_4, \alpha_1)  = +1 & ((\beta_4, \alpha_1) = -(\beta_4, \widetilde\alpha_1))
                \end{cases}\]
       \end{itemize}

        \begin{figure}
            \centering
            \includegraphics[scale=0.45]{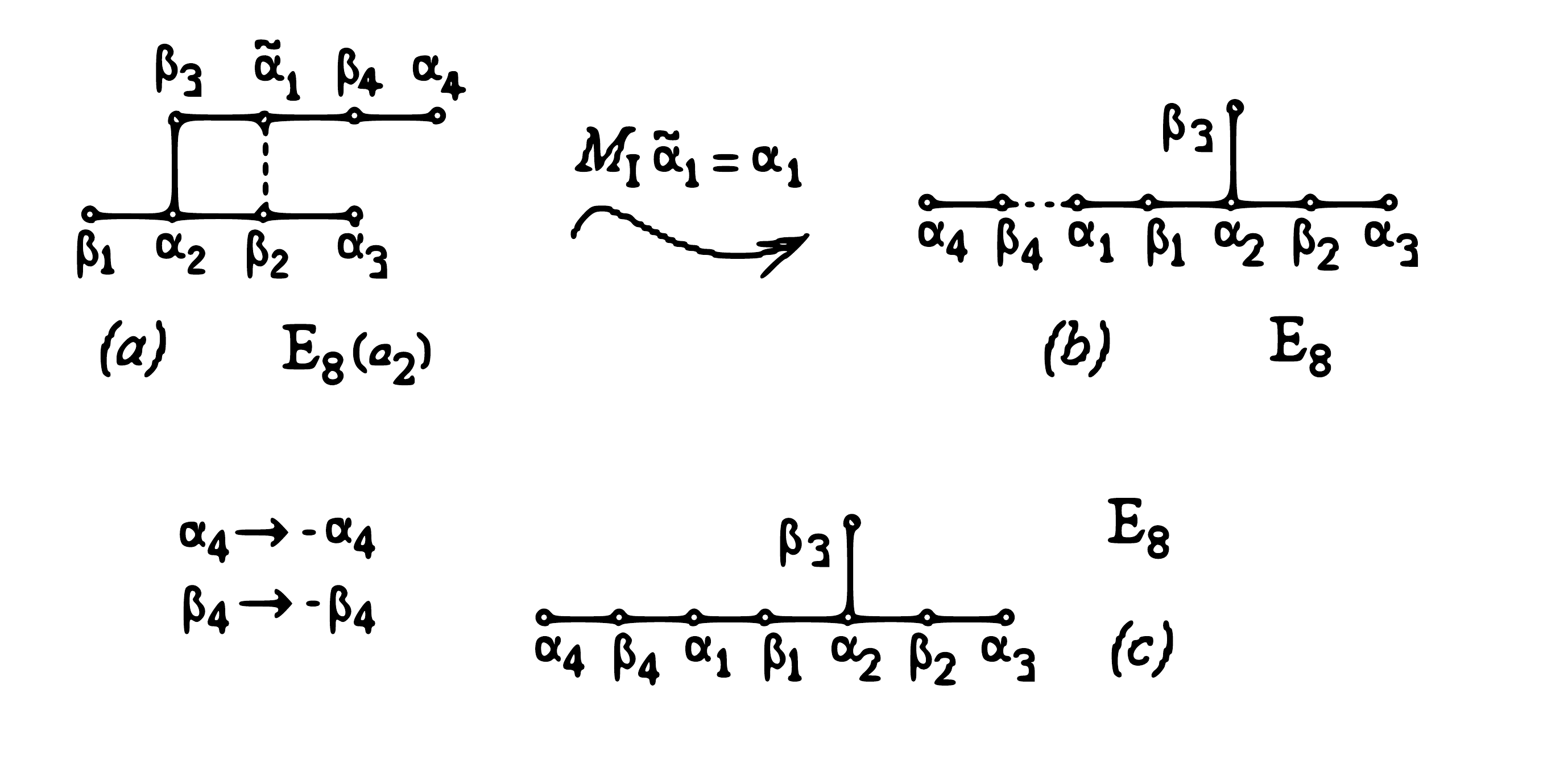}
            \caption{ Case (10). Mapping $M_I: E_8(a_2) \longmapsto E_8$ (diagram (b)). Diagrams (b) and (c) are equivalent
              (by similarity $\alpha_4 \longmapsto -\alpha_4$, $\beta_4 \longmapsto -\beta_4$) }
            \label{map_E8a2_to_E8}
        \end{figure}

        \item
        Pair $\{E_8(a_3), E_8(a_2)\}$: $M_I$ maps $E_8(a_3)$-set to $E_8(a_2)$-set.
        
        \begin{itemize}
            \item Mapping: $M_I\widetilde\alpha_4 = \alpha_4 =
            -(3\widetilde\alpha_1 + 2\beta_2 + 2\beta_3 + 2\beta_4 + \alpha_3  + \widetilde\alpha_4)$.
        
            \item Root system: $S = \{\widetilde\alpha_1, \beta_2, \beta_3, \beta_4, \alpha_3, \widetilde\alpha_4 \}$ and $\varPhi$ is a root system of type $E_6$.
        
            \item Minimal root: $\alpha_4$ is the minimal root in $\varPhi$.  The minimal root $M_I\widetilde\alpha_4 = \alpha_4$ is connected only with $\beta_4$.
        
            \item Eliminated edge: $\{\widetilde\alpha_4, \beta_3\}$.

            \item Emerging edge: $\{\beta_4, \alpha_4\}$.
        
            \item Checking relations:
            \[\begin{cases}
                \alpha_4 \perp \widetilde\alpha_1, \alpha_3, \beta_2, \beta_3 \quad
                  & (\alpha_4 \text{ - minimal root}) \\
                \alpha_4 \perp  \alpha_2  & (\beta_2 + \beta_3 \perp \alpha_2) \\
                \alpha_4 \perp  \beta_1 \quad & (\text{disconnected}) \\
                (\alpha_4, \beta_4) = -1 & (\alpha_4 \text{ - minimal root}) \\
           \end{cases}\]
        \end{itemize}

        \begin{figure}[!h]
            \centering
            \includegraphics[scale=0.45]{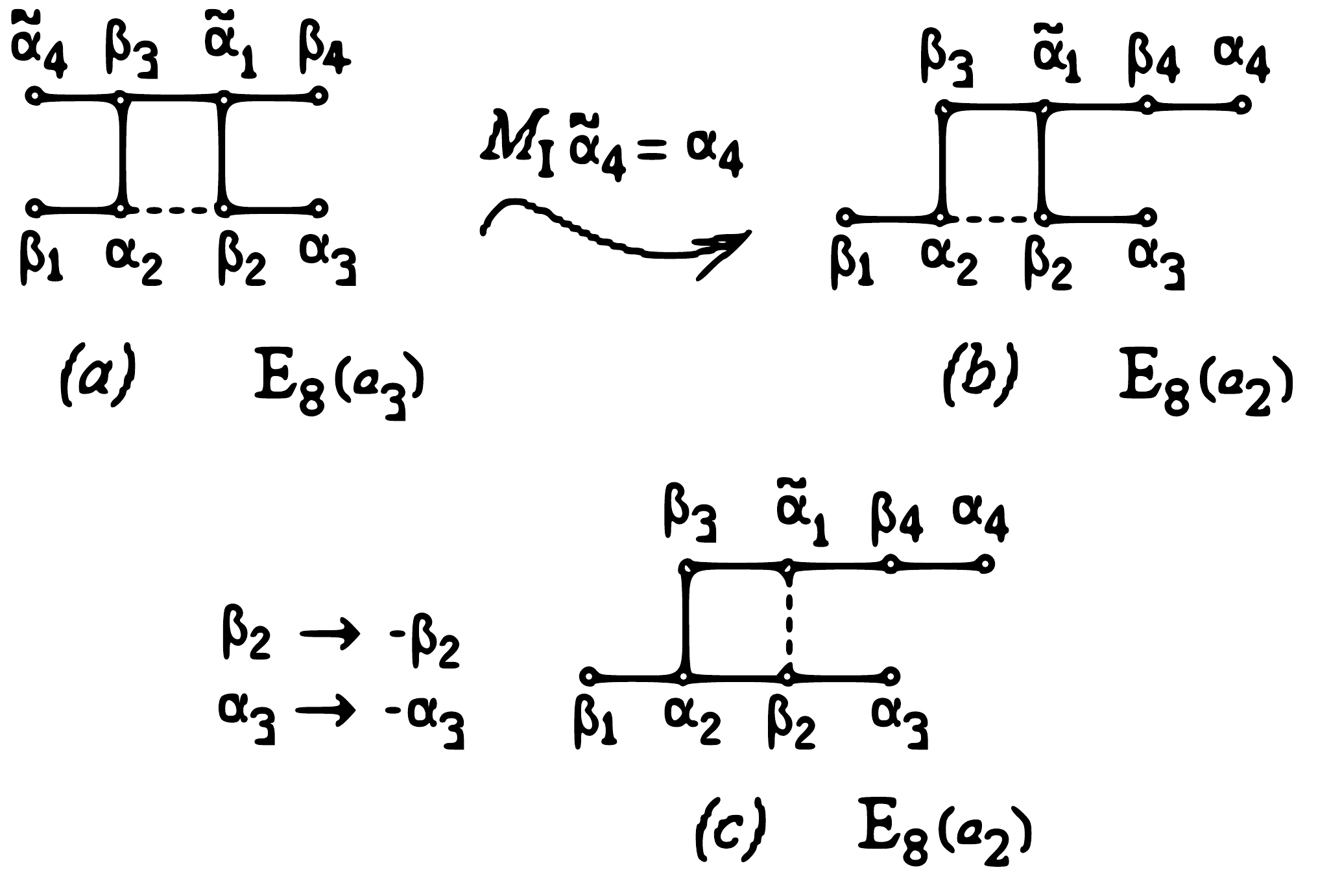}
            \caption{ Case (11). Mapping $M_I: E_8(a_3) \longmapsto E_8(a_2)$ (diagram (b)). Diagrams (b) and (c) are equivalent
            (by similarity $\beta_2 \longmapsto -\beta_2$, $\alpha_3 \longmapsto -\alpha_3$ )}
            \label{map_E8a3}
        \end{figure}

        \item
        Pair $\{E_8(a_4), E_8(a_1)\}$: $M_I$ maps $E_8(a_4)$-set to $E_8(a_1)$-set.

        \begin{itemize}
            \item Mapping: $M_I\widetilde\alpha_4 = \alpha_4 = -(\widetilde\alpha_4 + \beta_3 + \alpha_2 + \beta_1 + \alpha_1 + \beta_4)$.
        
            \item Root systems: $S = \{\widetilde\alpha_4, \beta_3, \alpha_2, \beta_1, \alpha_1, \beta_4\}$ and $\varPhi(S)$ is a root system of type $A_4$.
        
            \item Minimal roots: $\alpha_4$ is the minimal root in $\varPhi(S)$.
        
            \item Eliminated edges: $\{\widetilde\alpha_4, \beta_2\}$ and $\{\widetilde\alpha_4, \beta_3\}$.
        
            \item Emerging edge: $\{\alpha_4, \beta_4\}$.
        
            \item Checking relations:    
            \[\begin{cases}
                \alpha_4 \perp \beta_3, \alpha_2, \beta_1, \alpha_1 \quad (\alpha_4 \text{ - minimal root}) \\
                \alpha_4 \perp \widetilde\alpha_3, \beta_2 \quad (\beta_3 + \beta_1 \perp \widetilde\alpha_3, \text{ and } \widetilde\alpha_4 + \alpha_2 \perp \beta_2) \\
                (\alpha_4, \beta_4) = -1 \quad (\alpha_4 \text{ - minimal root})
            \end{cases}\]
        \end{itemize}

        \begin{figure}[!ht]
            \centering
            \includegraphics[scale=0.45]{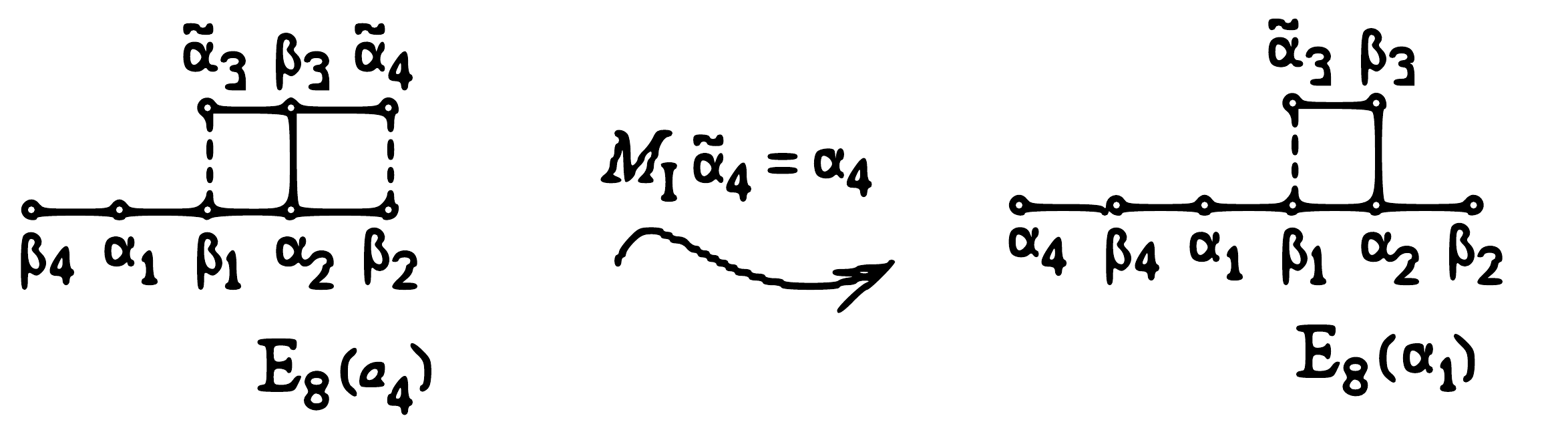}
            \caption{ Case (12). Mapping $M_I: E_8(a_4) \longmapsto E_8(a_1)$.}
            \label{map_E8a4_to_E8a1}
        \end{figure}

        \begin{table}
            \centering
            \caption{$\{ E_8(a_1), E_8\}$, $\{ E_8(a_2), E_8\}$,  $\{ E_8(a_3), E_8(a_2) \}$, $\{ E_8(a_4), E_8(a_1)\}$}
            \label{tab_part_root_syst_3}
            \begin{tabular} {|c|c|c|}
            \hline
             & & \\
             & $\widetilde\Gamma$-basis $\widetilde{S}$ and $\Gamma$-basis $S$
                         & Mapping $M_I$ \\
             & & \\
           \hline
              $\begin{array}{c}
               \footnotesize{\bf E_8} \\
             \end{array}$
               &
               $\begin{array}{c}
               \includegraphics[scale=0.35]{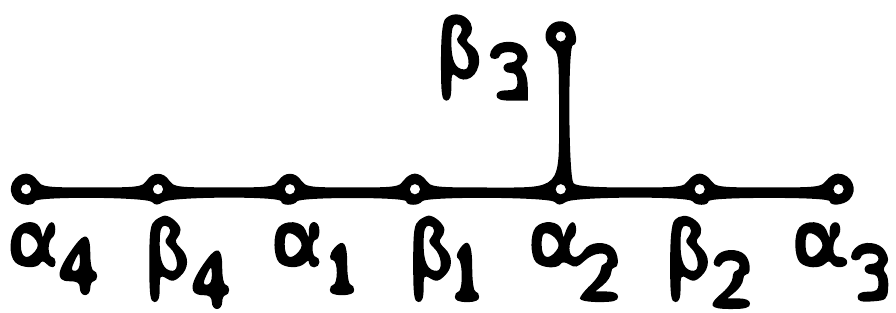}  \\
               \end{array}$
               &
               $\begin{array}{c}
               \\
               S = \{\alpha_1, \alpha_2, \alpha_3, \alpha_4, \beta_1, \beta_2, \beta_3, \beta_4 \} \\ \\
               \end{array}$
               \\
              \hline
              $\begin{array}{c}
              \\
              (9)
              \\
              \\
              {\bf E_8(a_1)}
              \\
             \end{array}$
              &
              $\begin{array}{c}
                 \\
                 \{\widetilde\Gamma, \Gamma\} = \{E_8(a_1), E_8\} \\
                 \\
               \includegraphics[scale=0.35]{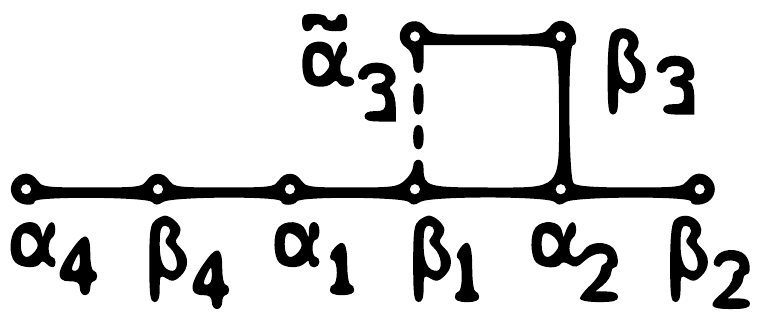}  \\
               \\
               \end{array}$
                &
                 $\begin{array}{l}
                 \\
                 \widetilde{S} = \{\alpha_1, \alpha_2, \widetilde\alpha_3, \alpha_4, \beta_1, \beta_2, \beta_3, \beta_4 \}, \\
                 M_I\tau = \tau \text{ for } \tau \in \{\alpha_1, \alpha_2, \alpha_4, \beta_1, \beta_2, \beta_3, \beta_4 \} \\
                 M_I\widetilde\alpha_3 = \alpha_3 = -(\widetilde\alpha_3 + \beta_3 + \alpha_2 + \beta_2)
                 \\
                 \end{array}$
                 \\
              \hline
              $\begin{array}{c}
              \\
              \footnotesize{(10)}
              \\
              \\
              \footnotesize{\bf E_8(a_2)}
              \\
             \end{array}$
              &
              $\begin{array}{c}
                 \\
                 \{\widetilde\Gamma, \Gamma\} = \{E_8(a_2), E_8\} \\ \\
               \includegraphics[scale=0.35]{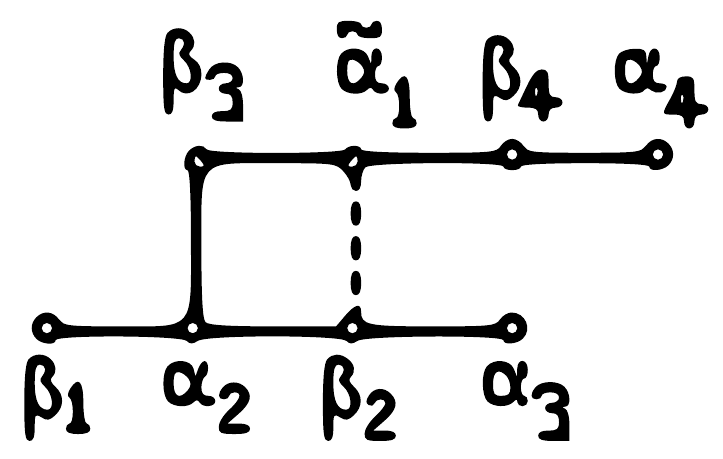}  \\
               \\
               \end{array}$
                &
                 $\begin{array}{l}
                 \\
                 \widetilde{S} = \{\widetilde\alpha_1, \alpha_2, \alpha_3, \alpha_4, \beta_1,
                                 \beta_2, \beta_3, \beta_4 \}, \\
                 M_I\tau = \tau \text{ for } \tau \in \{\alpha_2, \alpha_3, \alpha_4, \beta_1, \beta_2, \beta_3, \beta_4\} \\
                 M_I\widetilde\alpha_1 = \alpha_1 = -(\widetilde\alpha_1 + \beta_3 + \alpha_2 + \beta_1)
                 \\
                 \end{array}$
                 \\
              \hline
              $\begin{array}{c}
              \\
              \footnotesize{(11)}
              \\
              \\
               \footnotesize{\bf E_8(a_3)} \\
             \end{array}$
              &
              $\begin{array}{c}
                \\
                 \{\widetilde\Gamma, \Gamma\} = \{E_8(a_3), E_8(a_2)\} \\ \\
               \includegraphics[scale=0.35]{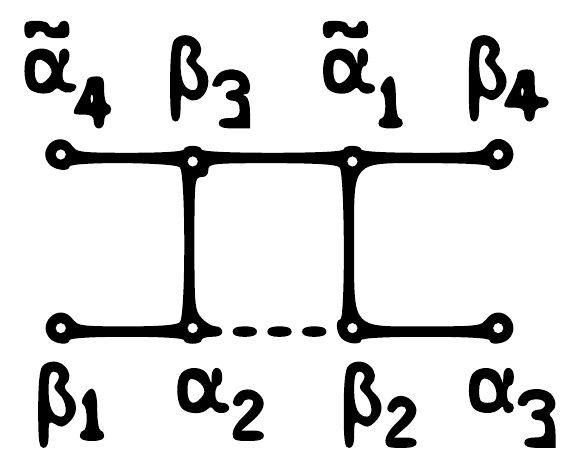}
               \\
               \end{array}$
                &
                 $\begin{array}{l}
                 \\
                 \widetilde{S} = \{\widetilde\alpha_1, \alpha_2, \alpha_3, \widetilde\alpha_4, \beta_1,
                                 \beta_2, \beta_3, \beta_4 \} \\
                 M_I\tau = \tau \text{ for } \tau \in \{\widetilde\alpha_1, \alpha_2, \alpha_3,
                                  \beta_1, \beta_2, \beta_3, \beta_4\} \\
                 M_I\widetilde\alpha_4 = \alpha_4 = \\
                     \qquad  -(3\widetilde\alpha_1 + 2\beta_2 + 2\beta_3 + 2\beta_4 + \alpha_3  + \widetilde\alpha_4)
                 \\
                 \end{array}$
                 \\
              \hline
              $\begin{array}{c}
              \\
              \footnotesize{(12)}
              \\
              \\
               \footnotesize{\bf E_8(a_4)} \\
             \end{array}$
              &
              $\begin{array}{c}
                \\
                 \{\widetilde\Gamma, \Gamma\} = \{E_8(a_4), E_8(a_1)\} \\ \\
               \includegraphics[scale=0.45]{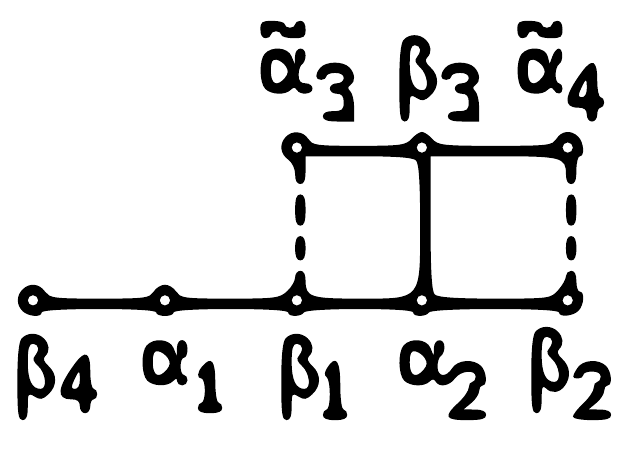}
               \\
               \end{array}$
                &
                 $\begin{array}{l}
                 \\
                 \widetilde{S} = \{\alpha_1, \alpha_2, \widetilde\alpha_3, \widetilde\alpha_4,
                     \beta_1, \beta_2, \beta_3, \beta_4 \} \\
                 M_I\tau = \tau \text{ for } \tau \in \{\alpha_1, \alpha_2, \widetilde\alpha_3,
                        \beta_1, \beta_2, \beta_3, \beta_4 \} \\
                 M_I\widetilde\alpha_4 = \alpha_4 = \\
                    \qquad   -(\widetilde\alpha_4 + \beta_3 + \alpha_2 + \beta_1 + \alpha_1 + \beta_4)
                 \\
                 \end{array}$
                 \\
              \hline
            \end{tabular}
        \end{table}

        \item
        Pair $\{E_8(a_5), E_8(a_4)\}$: $M_I$ maps $E_8(a_5)$-set to $E_8(a_4)$-set.
    
        \begin{itemize}    
            \item Mapping:
            $M_I\widetilde\beta_4 = \beta_4 = -(\widetilde\beta_4 + \widetilde\alpha_4 + \beta_3 + \alpha_2 + \beta_1 + \alpha_1)$
            
            \item Root systems:
            $S = \{\widetilde\beta_4, \widetilde\alpha_4, \beta_3, \alpha_2, \beta_1, \alpha_1\}$ and $\varPhi(S)$ is a root system of type $A_6$.
            
            \item Minimal roots:
            $\beta_4$ is the minimal root in $\varPhi(S)$.
            
            \item Eliminated edge:
            $\{\widetilde\alpha_4, \widetilde\beta_4\}$.
            
            \item Emerging edge:
            $\{\beta_4, \alpha_1\}$.
            
            \item Checking relations:
            \[\begin{cases}
                 \beta_4 \perp \beta_1, \alpha_2, \beta_3, \widetilde\alpha_4
                     \quad & (\beta_4 \text{ - minimal root}) \\
                 \beta_4 \perp \widetilde\alpha_3, \beta_2   \quad &
                   (\beta_3 + \beta_1 \perp \widetilde\alpha_3 \text{ and } \\
                 & \qquad  \alpha_2 + \widetilde\alpha_4 \perp \beta_2) \\
                 (\beta_4, \alpha_1) = -1 & (\beta_4 \text{ - minimal root})\\
          \end{cases}\]
        \end{itemize}

        \begin{figure}[!h]
            \centering
            \includegraphics[scale=0.45]{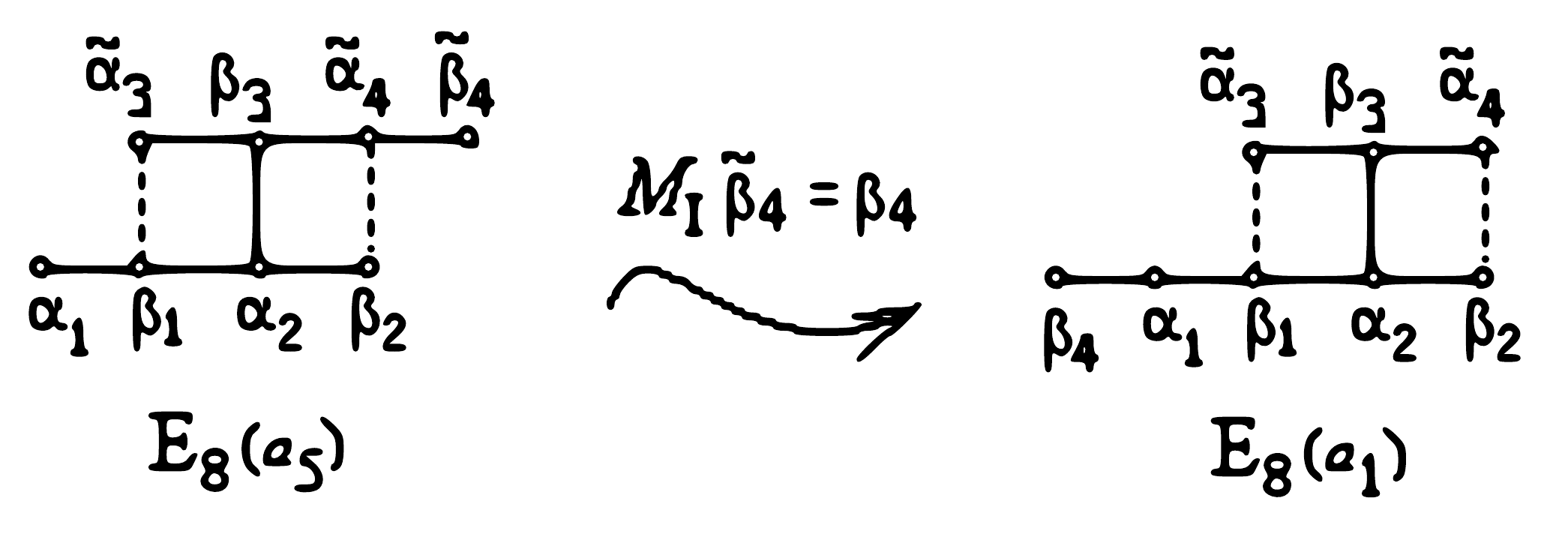}
            \caption{ Case (13). Mapping $M_I: E_8(a_5) \longmapsto E_8(a_4)$).}
            \label{map_E8a5_to_E8a4}
        \end{figure}

        \item
        Pair $\{E_8(a_6), E_8(a_4)\}$: $M_I$ maps $E_8(a_6)$-set to $E_8(a_4)$-set.

        \begin{itemize}
            \item Mapping: $M_I\widetilde\beta_4 = \beta_4 = -(\widetilde\beta_4 + 2\alpha_1 + 2\beta_1 + 2\alpha_2 + \beta_2 + \beta_3)$.
            
            \item Root systems: $S = \{\widetilde\beta_4, \alpha_1, \beta_1, \alpha_2, \beta_2, \beta_3\}$ and $\varPhi$ is a root system of type $D_6$.
            
            \item Minimal root: $\beta_4$ is the minimal root in $\varPhi$.
            
            \item Eliminated edges: $\{\widetilde\beta_4, \widetilde\alpha_3\}$ and $\{\widetilde\beta_4, \alpha_1\}$.
            
            \item Emerging edge: $\{\beta_4, \alpha_1\}$.
            
            \item Checking relations: 
            \[\begin{cases}
                  \beta_4 \perp \beta_1, \alpha_2, \beta_2, \beta_3, \widetilde\beta_4 & (\beta_4 \text{ - minimal root}) \\
                  \beta_4 \perp \widetilde\alpha_3 & ((\beta_4, \widetilde\alpha_3) = (\widetilde\beta_4 + \beta_3 + 2\beta_1, \widetilde\alpha_3) = -1 - 1 + 2 = 0) \\
                  \beta_4 \perp \widetilde\alpha_4 & (\beta_3 + \beta_2 \perp \alpha_4) \\
                  (\beta_4, \alpha_1) = -1 & (\beta_4 \text{ - minimal root}) \\
            \end{cases}\]
        \end{itemize}

        \begin{figure}[!h]
            \centering
            \includegraphics[scale=0.45]{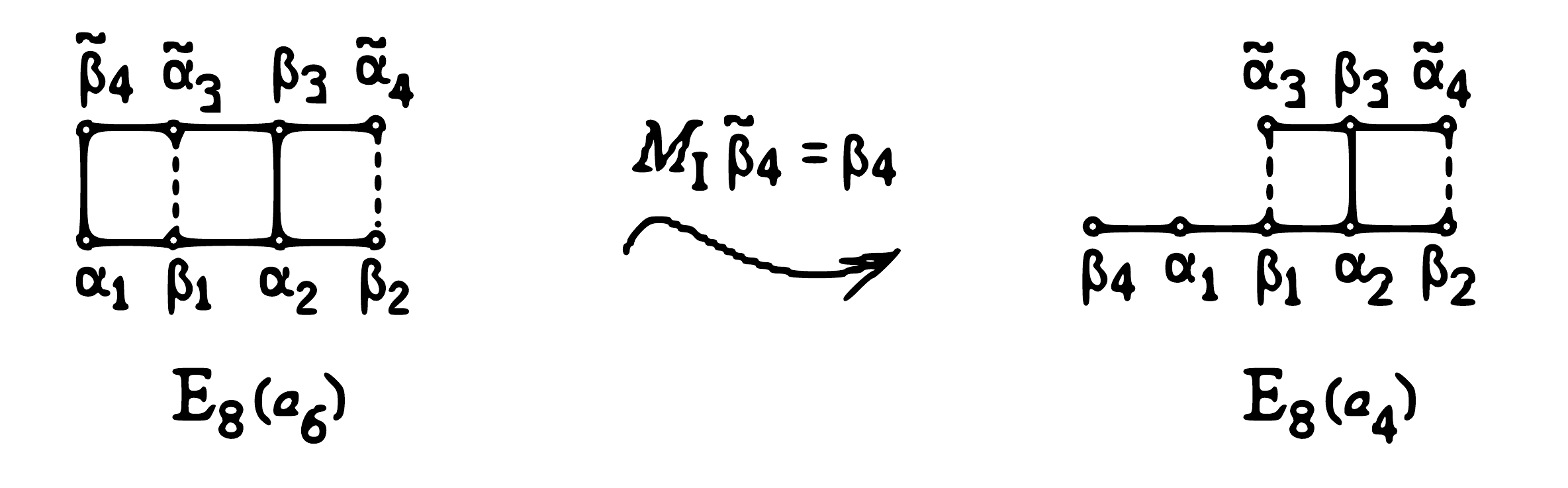}
            \caption{ Case (14). Mapping $M_I: E_8(a_6) \longmapsto E_8(a_4)$.}
            \label{map_E8a6_to_E8a4}
        \end{figure}

        \item
        Pair $\{E_8(a_7), E_8(a_5)\}$: $M_I$ maps $E_8(a_7)$-set to $E_8(a_5)$-set.

        \begin{itemize}
            \item Mapping: $M_I\widetilde\alpha_1 = \alpha_1 = -(2\beta_1 +  \widetilde\alpha_1 + \alpha_2 + \widetilde\alpha_3)$.

            \item Root systems: $S = \{\beta_1, \widetilde\alpha_1, \alpha_2, \widetilde\alpha_3 \}$ and $\varPhi(S)$ is a root system of type $D_4$.

            \item Minimal root: $\alpha_1$ is the minimal root in $\varPhi(S)$.

            \item Eliminated edges: $\{\widetilde\alpha_1, \widetilde\beta_1\}$ and $\{\widetilde\alpha_1, \beta_2\}$.

            \item Emerging edge: $\{\beta_1, \alpha_1 \}$.

            \item Checking relations:
            \[\begin{cases}
                \alpha_1 \perp \alpha_2, \widetilde\alpha_3 & (\alpha_1 \text{ - minimal root}) \\
                \alpha_1 \perp \widetilde\beta_4  & (\text{disconnected}) \\
                (\alpha_1, \beta_1) = -1  & (\alpha_1 \text{ - minimal root}) \\
                \alpha_1 \perp \beta_2, \beta_3 & (\alpha_2 + \widetilde\alpha_1 \perp \beta_2, \text{ and } \alpha_2 + \widetilde\alpha_3 \perp \beta_3)
            \end{cases}\]
        \end{itemize}

        \begin{figure}
            \centering
            \includegraphics[scale=0.45]{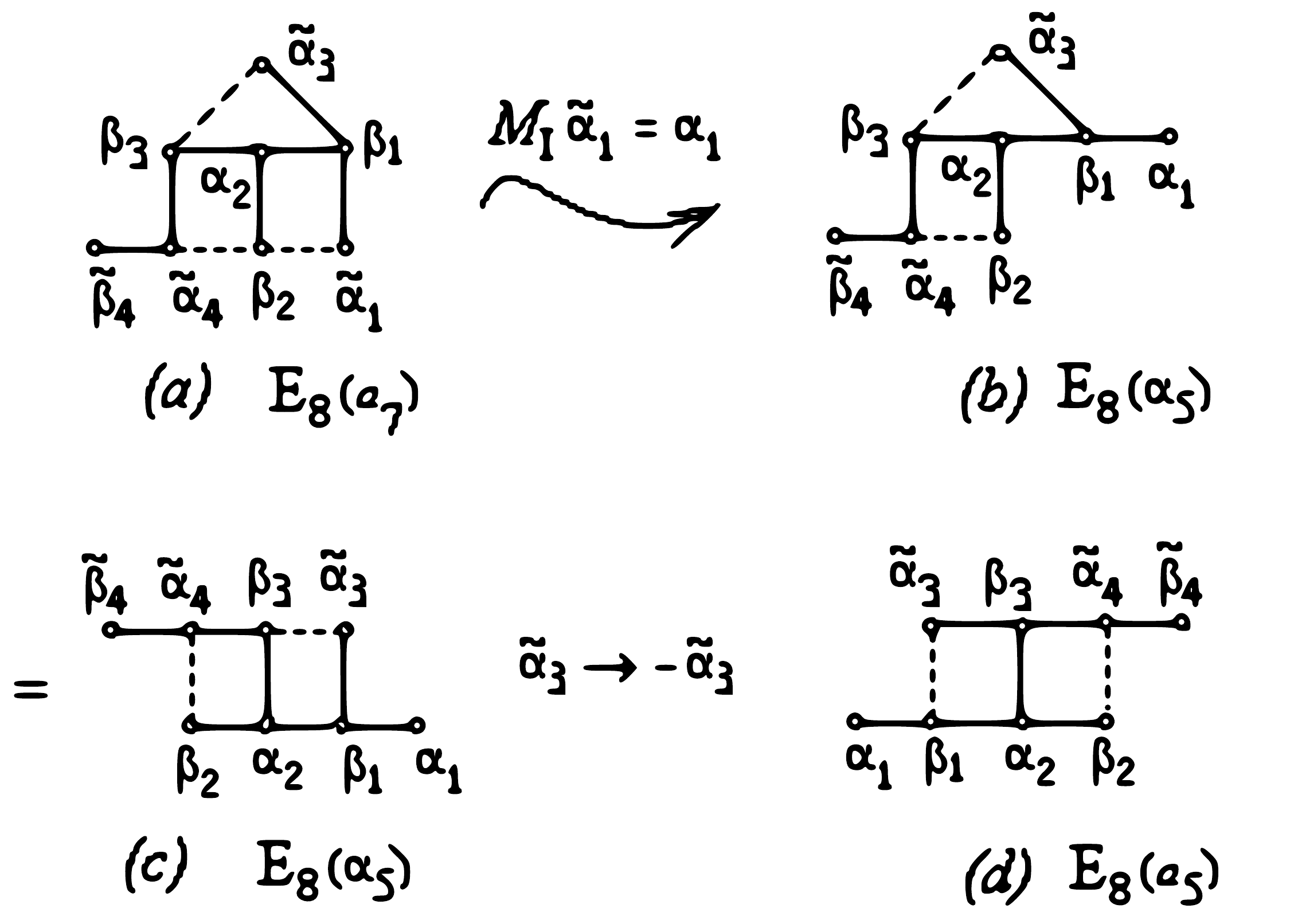}
            \caption{ Case (15). Mapping $M_I: E_8(a_7) \longmapsto E_8(a_5)$.
            Diagrams (c) and (d) are equivalent (by similarity $\widetilde\alpha_3 \longmapsto -\widetilde\alpha_3$)}
            \label{map_E8a7_to_E8a5}
        \end{figure}

        \item
        Pair $\{E_8(a_8), E_8(a_7)\}$: $M_I$ maps $E_8(a_8)$-set to $E_8(a_7)$-set.

        \begin{itemize}
            \item Mapping: $M_I\beta_4 = \widetilde\beta_4 = -(2\widetilde\alpha_4 + \beta_4 + \beta_2 + \beta_3)$.
            
            \item Root systems: $S = \{\widetilde\alpha_4, \beta_4, \beta_2, \beta_3\}$ and $\varPhi(S)$ is a root system of type $D_4$.
            
            \item Minimal root: $\beta_4$ is the minimal root in $\varPhi(S)$.
            
            \item Eliminated edges: $\{\beta_4, \widetilde\alpha_4\}$, $\{\beta_4, \widetilde\alpha_1\}$ and $\{\beta_4, \widetilde\alpha_3\}$.
            
            \item Emerging edge: $\{\widetilde\beta_4, \widetilde\alpha_4\}$.
            
            \item Checking relations:
            \[\begin{cases}
                \widetilde\beta_4 \perp  \beta_2, \beta_3 & (\beta_4 \text{ - minimal root}) \\
                \widetilde\beta_4 \perp  \widetilde\alpha_1, \widetilde\alpha_3 &
                     (\widetilde\alpha_1 \perp \beta_2 + \beta_4    \text{ and } \\
                      & \quad \beta_3 + \beta_4 \perp \widetilde\alpha_3) \\
                (\widetilde\beta_4, \widetilde\alpha_1) = -1 & (\beta_4 \text{ - minimal root})
            \end{cases}\]
        \end{itemize}
        
        \begin{figure}[!h]
            \centering
            \includegraphics[scale=0.45]{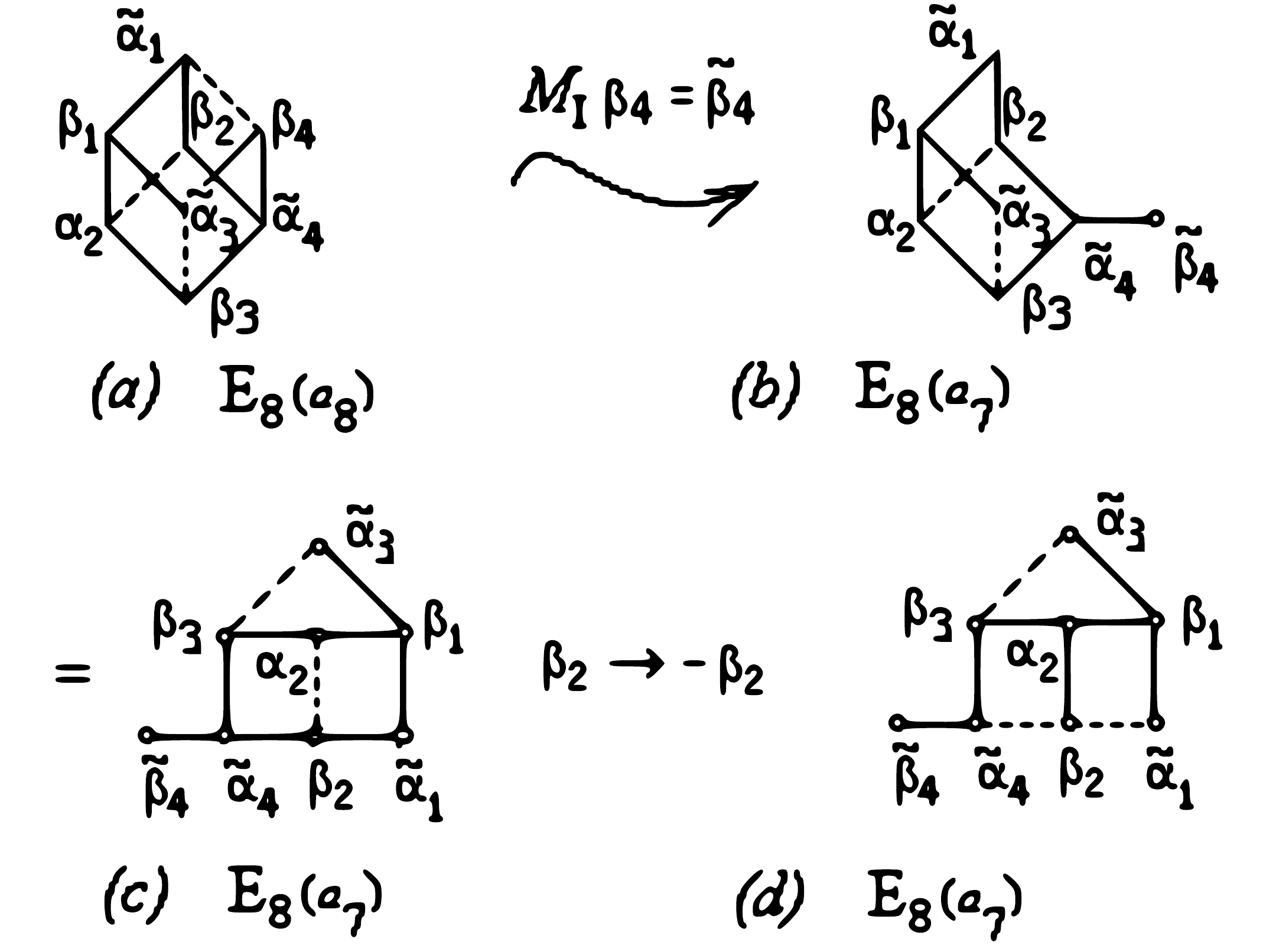}
            \caption{Case (16). Mapping $M_I: E_8(a_8) \longmapsto E_8(a_7)$.
            Diagrams (c) and (d) are equivalent (by similarity $\widetilde\beta_2 \longmapsto -\widetilde\beta_2$) }
            \label{map_E8a8_to_E8a7}
        \end{figure}

        \begin{table}
           \centering
           \caption{$\{ E_8(a_5), E_8(a_4) \}$, $\{ E_8(a_6), E_8(a_4)\}$,  $\{ E_8(a_7), E_8(a_5) \}$, $\{ E_8(a_8), E_8(a_7)\}$}
           \label{tab_part_root_syst_4}
           \begin{tabular} {|c|c|c|}
            \hline
               & & \\
               & $\widetilde\Gamma$-basis $\widetilde{S}$ and $\Gamma$-basis $S$
                         & Mapping $M_I$ \\
               & & \\
              \hline
              $\begin{array}{c}
              \\
              \footnotesize{(13)}
              \\
              \\
              \footnotesize{\bf E_8(a_5)}
              \\
             \end{array}$
              &
              $\begin{array}{c}
                \\
                 \{\widetilde\Gamma, \Gamma\} = \{E_8(a_5), E_8(a_4)\} \\ \\
               \includegraphics[scale=0.35]{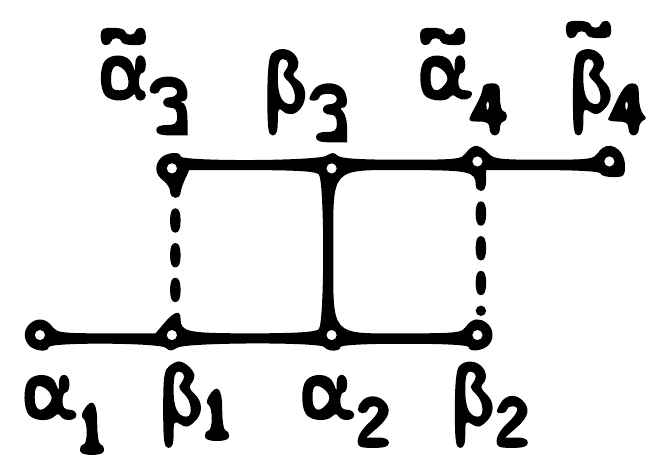}  \\
               \end{array}$
                &
                 $\begin{array}{l}
                 \\
                 \widetilde{S} = \{\alpha_1, \alpha_2, \widetilde\alpha_3, \widetilde\alpha_4,
                     \beta_1, \beta_2, \beta_3, \widetilde\beta_4 \}, \\
                 M_I\tau = \tau \text{ for } \tau \in \{\alpha_1, \alpha_2, \widetilde\alpha_3, \widetilde\alpha_4,
                           \beta_1,  \beta_2, \beta_3 \} \\
                 M_I\widetilde\beta_4 = \beta_4  = \\
                     \qquad  -(\widetilde\beta_4 + \widetilde\alpha_4 + \beta_3  + \alpha_2 + \beta_1 + \alpha_1) 
                     \\
                 \end{array}$
                 \\
              \hline
              $\begin{array}{c}
              \\
              \footnotesize(14)
              \\
              \\
              \footnotesize{{\bf E_8(a_6)}}
              \\
             \end{array}$
              &
              $\begin{array}{c}
                 \\
                 \{\widetilde\Gamma, \Gamma\} = \{E_8(a_6), E_8(a_4)\} \\ \\
               \includegraphics[scale=0.35]{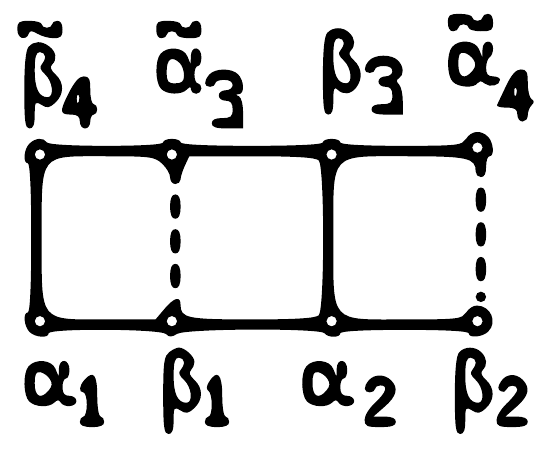}  \\
               \end{array}$
                &
                 $\begin{array}{l}
                 \\
                 \widetilde{S} = \{\alpha_1, \alpha_2, \widetilde\alpha_3, \widetilde\alpha_4,
                     \beta_1, \beta_2, \beta_3, \widetilde\beta_4 \}, \\
                 M_I\tau = \tau \text{ for } \tau \in \{\alpha_1, \alpha_2, 
                   \widetilde\alpha_3, \widetilde\alpha_4, \beta_1,  \beta_2, \beta_3 \} \\
                 M_I\widetilde\beta_4 = \beta_4 = \\
                     \qquad   -(\widetilde\beta_4 + 2\alpha_1 + 2\beta_1 + 2\alpha_2 + \beta_2 + \beta_3)
                 \\
                 \end{array}$
                 \\
              \hline
              $\begin{array}{c}
              \\
              \footnotesize{(15)}
              \\
              \\
              \footnotesize{\bf E_8(a_7)}
              \\
             \end{array}$
              &
              $\begin{array}{c}
                \\
                 \{\widetilde\Gamma, \Gamma\} = \{E_8(a_7), E_8(a_5)\} \\ \\
               \includegraphics[scale=0.45]{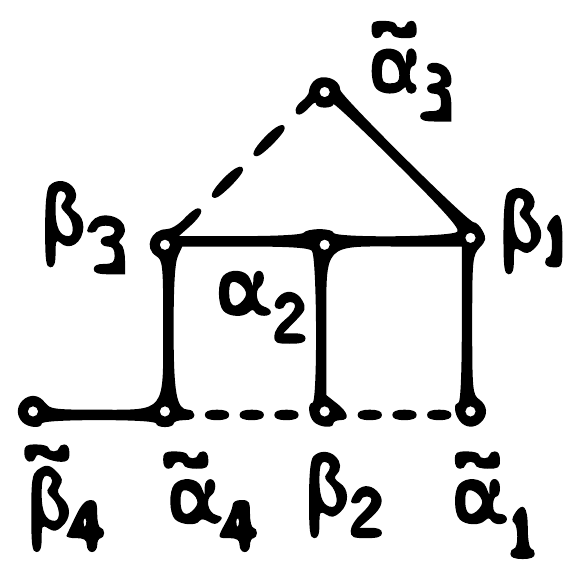}  \\
               \end{array}$
                &
                 $\begin{array}{l}
                 \\
                 \widetilde{S} = \{\widetilde\alpha_1, \alpha_2, \widetilde\alpha_3, \widetilde\alpha_4,
                     \beta_1, \beta_2, \beta_3, \widetilde\beta_4 \}, \\
                 M_I\tau = \tau \text{ for } \tau \in \{\alpha_2, \widetilde\alpha_3, \widetilde\alpha_4,
                     \beta_1, \beta_2, \beta_3, \widetilde\beta_4 \} \\
                 M_I\widetilde\alpha_1 = \alpha_1 = -(2\beta_1 +  \widetilde\alpha_1 + \alpha_2 + \widetilde\alpha_3)
                 \\
                 \end{array}$
                 \\
              \hline
              $\begin{array}{c}
               \\
               \footnotesize{(16)}
               \\
               \\
               \footnotesize{\bf E_8(a_8)}
               \\
             \end{array}$
              &     
              $\begin{array}{c}
                 \\
                  \{\widetilde\Gamma, \Gamma\} = \{E_8(a_8), E_8(a_7)\} \\ \\
               \includegraphics[scale=0.35]{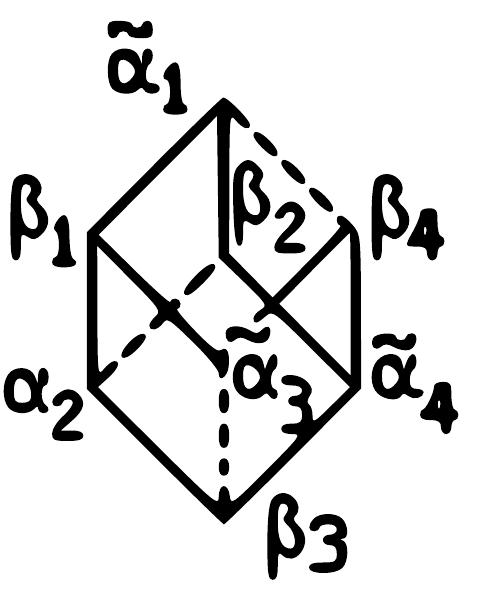}  \\
               \end{array}$
                &
                 $\begin{array}{l}
                 \\
                 \widetilde{S} = \{\widetilde\alpha_1, \alpha_2, \widetilde\alpha_3, \widetilde\alpha_4,
                    \beta_1, \beta_2, \beta_3, \beta_4 \}, \\
                 M_I\tau = \tau  \\
                 \qquad  \text{ for } \tau \in
                    \{\widetilde\alpha_1, \alpha_2, \widetilde\alpha_3, \widetilde\alpha_4,
                       \beta_1, \beta_2, \beta_3\}        \\
                 M_I\beta_4 = \widetilde\beta_4 = -(2\widetilde\alpha_4 + \beta_4 + \beta_2 + \beta_3)
                 \\
                 \end{array}$
                 \\
              \hline
            \end{tabular}
        \end{table}
    \end{enumerate}

    \item
    Let us prove that $M_I$ is an involution.
    There exists a certain root $\widetilde\alpha \in \widetilde{S}$ such that
    \begin{equation*}
      \begin{cases}
          M_I \tau_i = \tau_i \text{ for all }
               \tau_i \in \widetilde{S}, \tau_i \neq \widetilde\alpha, \\
          M_I \widetilde\alpha = \alpha =
             -\widetilde\alpha + \sum t_i \tau_i,
             \text{ where sum is taken over all }
               \tau_i \in \widetilde{S}, \tau_i \neq \widetilde\alpha.
      \end{cases}
    \end{equation*}
    The image $M_I \widetilde\alpha = \alpha$ is the root in $S$.
    Then,
    \begin{equation*}
        M^2_I \widetilde\alpha
        = -M\widetilde\alpha + \sum t_i \tau_i
        = \widetilde\alpha - \sum t_i \tau_i + \sum t_i \tau_i
        = \widetilde\alpha
        \quad  \text{ and } \quad
        M^2_I\widetilde\alpha
        = M_I\alpha
        = \widetilde\alpha.
    \end{equation*}
    In other words,
    \begin{equation*}
        M_I : \widetilde\alpha \longmapsto \alpha
        \quad \textup{ and } \quad
        M_I : \alpha \longmapsto \widetilde\alpha.
    \end{equation*}

\end{enumerate}

\subsection{Relation of partial Cartan matrices}

  Consider pairs $\{ \widetilde\Gamma, \Gamma \}$ out of the adjacency list \eqref{eq_pairs_trans}.
  The transition matrix $M_I$ maps the ${\widetilde\Gamma}$-basis $\widetilde{S}$
  to the ${\Gamma}$-basis  $S$:
  \begin{equation*}
      M \widetilde\tau_i = \tau_i, \text{ where } \widetilde\tau_i \in \widetilde{S},  \tau_i \in S.
  \end{equation*}
  If the matrix $M$ does not change a certain root $\widetilde\tau_i$,
  the designation of this root and the corresponding node
  is the same for $\widetilde\Gamma$-basis and $\Gamma$-basis, namely: $M\tau_i = \tau_i$.
  The transition matrix $M$ relates the partial Cartan matrices $B_\Gamma$ and  $B_{\widetilde\Gamma}$
  as follows:
  \begin{equation}
    \label{eq_trans_matr}
      {}^t{M} \cdot B_{\widetilde\Gamma} \cdot {M} = B_\Gamma.
  \end{equation}
 Eq. \eqref{eq_trans_matr} is the relation of partial Cartan matrices $B_{\Gamma}$
given in different bases.

\subsubsection{Example: $\{E_6(a_1), E_6 \}$}

For the pair $\{E_6(a_1), E_6 \}$ (case (3) in Appendix~\ref{sec_trans}),
\[
   M =
     \left [
       \begin{array}{cccccc}
         1 & 0 & 0 & 0 & 0 & -1 \\
         0 & 1 & 0 & 0 & 0 &  0 \\
         0 & 0 & 1 & 0 & 0 & -1 \\
         0 & 0 & 0 & 1 & 0 & -1 \\
         0 & 0 & 0 & 0 & 1 &  0 \\
         0 & 0 & 0 & 0 & 0 & -1 \\
      \end{array}
   \right ]
   \quad \textup{ and } \quad
    B_{E_6(a_1)} =
     \left [
       \begin{array}{cccccc}
         2 &  0 &  0 & -1 &  0 & 0 \\
         0 &  2 &  0 & -1 & -1 & 1 \\
         0 &  0 &  2 & -1 &  0 & -1 \\
        -1 & -1 & -1 &  2 &  0 & 0 \\
         0 & -1 &  0 &  0 &  2 & 0 \\
         0 &  1 & -1 &  0 &  0 & 2 \\
      \end{array}
     \right ].
\]
Then,
\[
    {}^tM B_{E_6(a_1)} M =
     \left [
       \begin{array}{cccccc}
         2 & 0 & 0 & -1 & 0 & -1 \\
         0 & 2 & 0 & -1 & -1 &  0 \\
         0 & 0 & 2 & -1 & 0 & 0 \\
         -1 & -1 & -1 & 2 & 0 & 0 \\
         0 & -1 & 0 & 0 & 2 &  0 \\
         -1 & 0 & 0 & 0 & 0 & 2 \\
      \end{array}
      \right ].
\]
The matrix ${}^tM B_{E_6(a_1)}M$ is the Cartan matrix $B_{E_6}$ for Dynkin diagram $E_6$.

\begin{appendix}

\section{Dynkin-Minchenko completion procedure}
   \label{sect_A}

\subsection{Enhanced Dynkin diagram $\Delta(E_6)$}

Extra nodes $m_1$ and $m_2$ for the enhanced Dynkin diagram $\Delta(E_6)$ are as follows (see Fig.~\ref{fig_enhan_E6}):
\begin{equation}
  \label{eq_E6_extra_1}
  \begin{split}
    & m_1 = 2\alpha_4 + \alpha_2 + \alpha_3 +\alpha_5, \\
    & m_1 \text{ is the maximal element in } \{\alpha_4, \alpha_2, \alpha_3, \alpha_5 \},
  \end{split}
\end{equation}
\begin{equation}
  \label{eq_E6_extra_2}
  \begin{split}
    & m_2 = 2m_1 - \alpha_4 + \alpha_1 + \alpha_6, \quad \\
    & m_2 \text{ is the maximal element in } \{m_1, -\alpha_4, \alpha_1, \alpha_6 \}.
  \end{split}
\end{equation}
From \eqref{eq_E6_extra_1} and \eqref{eq_E6_extra_2} we get
\begin{equation*}
     m_2 = 3\alpha_4 + 2\alpha_2 + 2\alpha_3 + 2\alpha_5 + \alpha_1 + \alpha_6.
\end{equation*}
Then, $m_2$ is also the maximal element in
the $E_6$-set $\{\alpha_1, \alpha_2, \alpha_3, \alpha_4, \alpha_4, \alpha_6 \}$.

\begin{figure}
    \centering
    \includegraphics[scale=0.3]{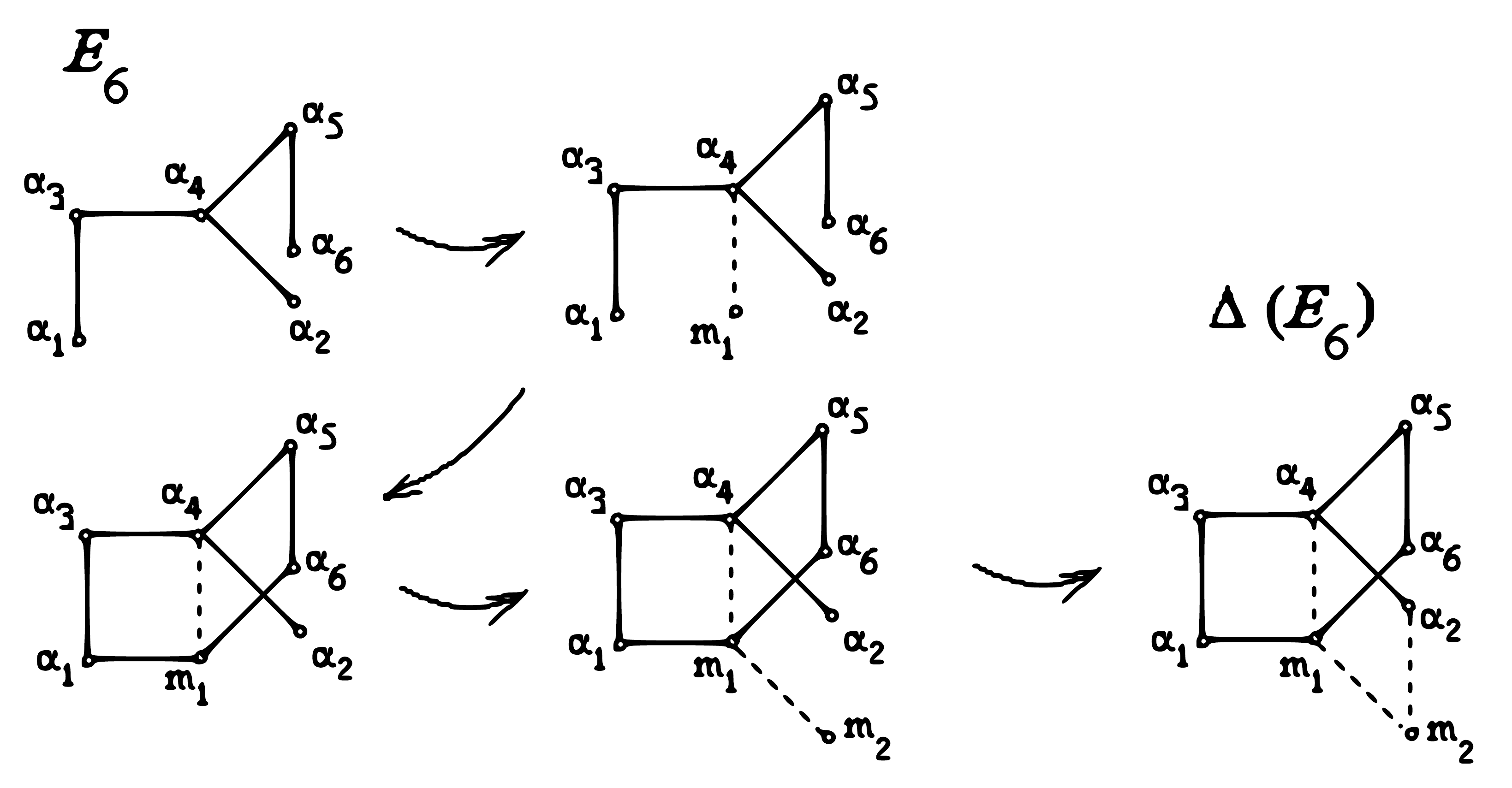}
    \caption{Enhanced Dynkin diagrams $\Delta(E_6)$}
    \label{fig_enhan_E6}
\end{figure}

\subsection{Enhanced Dynkin diagram $\Delta(E_7)$}

As in the case $E_6$, the extra nodes $m_1$ and $m_2$ are as follows:
\begin{equation*}
  \begin{split}
    & m_1 = 2\alpha_4 + \alpha_2 + \alpha_3 +\alpha_5, \\
    & m_2 = 3\alpha_4 + 2\alpha_2 + 2\alpha_3 + 2\alpha_5 + \alpha_1 + \alpha_6.
  \end{split}
\end{equation*}
Further, the extra node $m_3$ is as follows:
\begin{equation*}
  \begin{split}
    m_3 = & 2\alpha_6 + \alpha_5 + \alpha_7 + m_1 = \\
          & \alpha_2 + \alpha_3 + 2\alpha_4 + 2\alpha_5 + 2\alpha_6 + \alpha_7.
  \end{split}
\end{equation*}
Here, $m_3$ is the maximal element in the $D_6$-set
$\{\alpha_2, \alpha_3, \alpha_4, \alpha_4 \alpha_6 \}$.
The extra node $m_4$:
\begin{equation*}
  \begin{split}
    m_4 = & 2\alpha_1 + \alpha_3 + m_1 + m_3 = \\
          & 4\alpha_4 + 3\alpha_5 +  3\alpha_3 + 2\alpha_2 + 2\alpha_6 + 2\alpha_1 + \alpha_7.
  \end{split}
\end{equation*}
The node $m_4$ is the maximal element in the $E_7$-set (see Fig. \ref{fig_enhan_E7}):
\begin{equation*}
   \{\alpha_1, \alpha_2, \alpha_3, \alpha_4, \alpha_5, \alpha_6, \alpha_7 \}.
\end{equation*}

\begin{figure}
    \centering
    \includegraphics[scale=0.3]{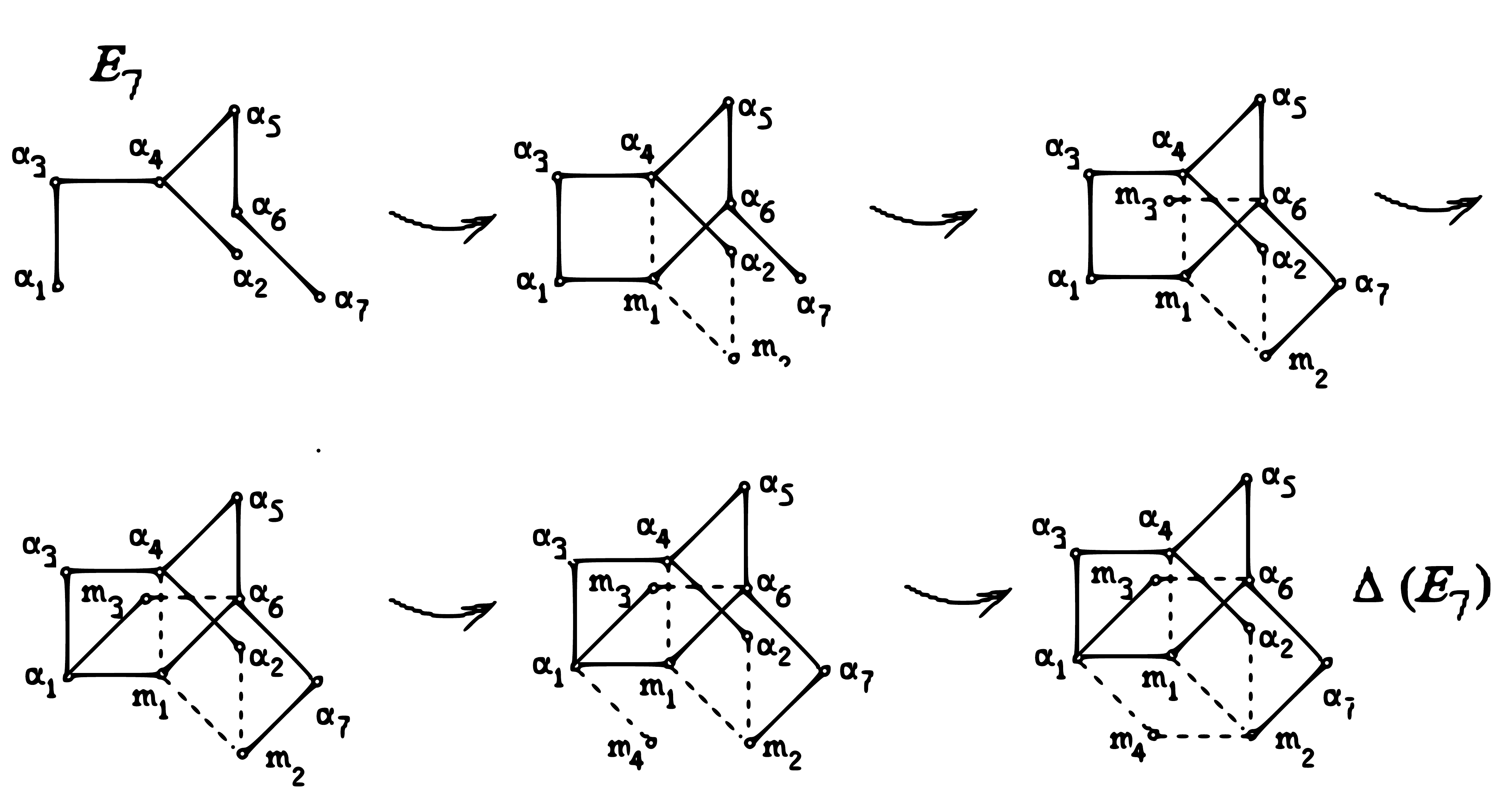}
    \caption{Enhanced Dynkin diagrams $\Delta(E_7)$}
    \label{fig_enhan_E7}
\end{figure}

\section{Transition  matrices $M_I$, cases (1) - (16)}
  \label{sec_trans}

\begin{enumerate}[(1)]
    \item $\{ D_4(a_1), D_4 \}$:
    \[\begin{array}{ccccc}
         \alpha_1 & \alpha_2 & \alpha_3 & \alpha_4 & ~ \\
         1 & 0 & -1 & 0 & \alpha_1 \\
         0 & 1 & -1 & 0 & \alpha_2 \\
         0 & 0 & -1 & 0 & \widetilde\alpha_3 \\
         0 & 0 &  0 & 1 & \alpha_4
    \end{array}\]

    \item $\{ D_l(a_k), D_l \}$:
    \[\begin{array}{cccccc}
         \tau_1  & \tau_2 & \dots & \tau_{l-1} & \beta_{k+1} & ~ \\
            1    & 0      & \dots & 0  & -1  & \tau_1 \\
            0    & 1      & \dots & -1 & -2  & \tau_2 \\
           \dots & \dots  & \dots & \dots & \dots & \dots \\
            0    & 0      & \dots & 0 &  -1 & \tau_{k+1} \\
            0    & 0      & \dots & 0 &  -1 & \overline\tau_{k+1} \\
           \dots & \dots  & \dots & \dots & \dots & \dots\\
            0    & 0      & \dots & 0 &  0 & \tau_{l-1}
      \end{array}\]

    \item $\{ E_6(a_1), E_6 \}$:
    \[\begin{array}{ccccccc}
         \alpha_1  & \alpha_2 & \alpha_3 & \beta_1 & \beta_2 & \beta_3 & \\
         1 & 0 &   0 &   0 & 0 & -1 & \alpha_1 \\
         0 & 1 &   0 &   0 & 0 & 0  & \alpha_2 \\
         0 & 0 &   1 &   0 & 0 & -1 & \alpha_3 \\
         0 & 0 &   0 &   1 & 0 & -1 & \beta_1   \\
         0 & 0 &   0 &   0 & 1 & 0  & \beta_2  \\
         0 & 0 &   0 &   0 & 0 & -1 & \widetilde\beta_3
    \end{array}\]
    
    \item $\{ E_6(a_2), E_6(a_1) \}$:
    \[\begin{array}{ccccccc}
         \alpha_1  & \alpha_2 & \alpha_3 & \beta_1 & \beta_2 & \widetilde\beta_3 & ~ \\
         1 & 0 & 0 & 0 & 0 & 0 & \alpha_1 \\
         0 & 1 & 0 & 0 & -1 & 0 & \alpha_2 \\
         0 & 0 & 1 & 0 & -1 & 0 & \alpha_3 \\
         0 & 0 & 0 & 1 & -1 & 0 & \beta_1 \\
         0 & 0 & 0 & 0 & -1 & 0 & \widetilde\beta_2 \\
         0 & 0 & 0 & 0 & 0 & 1 & \widetilde\beta_3
    \end{array}\]
    
    \item $\{ E_7(a_1), E_7 \}$:
    \[\begin{array}{cccccccc}
         \alpha_1  & \alpha_2 & \alpha_3 & \beta_1 & \beta_2 & \beta_3 & \beta_4 & ~ \\
         1 & 0 &   0 &   0 & 0 & 0 & 0 & \alpha_1 \\
         0 & 1 &  -1 &   0 & 0 & 0 & 0 & \alpha_2 \\
         0 & 0 &  -1 &   0 & 0 & 0 & 0 & \widetilde\alpha_3 \\
         0 & 0 &   0 &   1 & 0 & 0 & 0 & \beta_1 \\
         0 & 0 &  -1 &   0 & 1 & 0 & 0 & \beta_2 \\
         0 & 0 &  -1 &   0 & 0 & 1 & 0 & \beta_3 \\
         0 & 0 &   0 &   0 & 0 & 0 & 1 & \beta_4
    \end{array}\]
    
    \item $\{ E_7(a_2), E_7 \}$:
    \[\begin{array}{cccccccc}
         \alpha_1  & \alpha_2 & \alpha_3 & \beta_1 & \beta_2 & \beta_3 & \beta_4 & ~ \\
        -1 & 0 &   0 &  0 & 0 & 0 & 0 & \widetilde\alpha_1 \\
        -1 & 1 &   0 &  0 & 0 & 0 & 0 & \alpha_2 \\
         0 & 0 &   1 &  0 & 0 & 0 & 0 & \alpha_3 \\
        -1 & 0 &   0 &  1 & 0 & 0 & 0 & \beta_1 \\
         0 & 0 &   0 &  0 & 1 & 0 & 0 & \beta_2 \\
        -1 & 0 &   0 &  0 & 0 & 1 & 0 & \beta_3 \\
         0 & 0 &   0 &  0 & 0 & 0 & 1 & \beta_4
    \end{array}\]
    
    \item $\{ E_7(a_3), E_7(a_1) \}$:
    \[\begin{array}{cccccccc}
         \alpha_1  & \alpha_2 & \widetilde\alpha_3 & \beta_1 & \beta_2 & \beta_3 & \beta_4 & ~ \\
         1 & 0 &  0 &  0 & 0 & 0 & -1 & \alpha_1 \\
         0 & 1 &  0 &  0 & 0 & 0 & -1 & \alpha_2 \\
         0 & 0 &  1 &  0 & 0 & 0 & 0 & \widetilde\alpha_3 \\
         0 & 0 &  0 &  1 & 0 & 0 & -1 & \alpha_4 \\
         0 & 0 &  0 &  0 & 1 & 0 & -1 & \beta_1 \\
         0 & 0 &  0 &  0 & 0 & 1 & 0 & \beta_2 \\
         0 & 0 &  0 &  0 & 0 & 0 & -1 & \beta_3
    \end{array}\]
    
    \item $\{ E_7(a_4), E_7(a_3) \}$:
    \[\begin{array}{cccc cccc}
         \alpha_1  & \alpha_2 & \widetilde\alpha_3 & \alpha_4 & \beta_1 & \beta_2 & \beta_3 & ~ \\
         -1 & 0 &   0 &  0 & 0 & 0 & 0 & \widetilde\alpha_1 \\
         -1 & 1 &   0 &  0 & 0 & 0 & 0 & \alpha_2 \\
         -1 & 0 &   1 &  0 & 0 & 0 & 0 & \widetilde\alpha_3 \\
          0 & 0 &   0 &  1 & 0 & 0 & 0 & \alpha_4 \\
         -2 & 0 &   0 &  0 & 1 & 0 & 0 & \beta_1 \\
          0 & 0 &   0 &  0 & 0 & 1 & 0 & \beta_2 \\
          0 & 0 &   0 &  0 & 0 & 0 & 1 & \beta_3
    \end{array}\]

    \item $\{ E_8(a_1), E_8 \}$:
    \[\begin{array}{ccccccccc}
         \alpha_1  & \alpha_2 & \alpha_3 & \alpha_4 & \beta_1 & \beta_2 & \beta_3 & \beta_4 & ~ \\
         1 & 0 &   0 &   0 & 0 & 0 & 0 & 0 & \alpha_1 \\
         0 & 1 &  -1 &   0 & 0 & 0 & 0 & 0 & \alpha_2 \\
         0 & 0 &  -1 &   0 & 0 & 0 & 0 & 0 & \widetilde\alpha_3 \\
         0 & 0 &   0 &   1 & 0 & 0 & 0 & 0 & \alpha_4 \\
         0 & 0 &   0 &   0 & 1 & 0 & 0 & 0 & \beta_1 \\
         0 & 0 &  -1 &   0 & 0 & 1 & 0 & 0 & \beta_2 \\
         0 & 0 &  -1 &   0 & 0 & 0 & 1 & 0 & \beta_3 \\
         0 & 0 &   0 &   0 & 0 & 0 & 0 & 1 & \beta_4
    \end{array}\]

    \item $\{ E_8(a_2), E_8 \}$:
    \[\begin{array}{ccccccccc}
         \alpha_1  & \alpha_2 & \alpha_3 & \alpha_4 & \beta_1 & \beta_2 & \beta_3 & \beta_4 & ~ \\
        -1 & 0 &  0 &   0 & 0 & 0 & 0 & 0 & \widetilde\alpha_1 \\
        -1 & 1 &   0 &   0 & 0 & 0 & 0 & 0 & \alpha_2 \\
         0 & 0 &   1 &   0 & 0 & 0 & 0 & 0 & \alpha_3 \\
         0 & 0 &   0 &   1 & 0 & 0 & 0 & 0 & \alpha_4 \\
        -1 & 0 &   0 &   0 & 1 & 0 & 0 & 0 & \beta_1 \\
         0 & 0 &   0 &   0 & 0 & 1 & 0 & 0 & \beta_2 \\
        -1 & 0 &   0 &   0 & 0 & 0 & 1 & 0 & \beta_3 \\
         0 & 0 &   0 &   0 & 0 & 0 & 0 & 1 & \beta_4
    \end{array}\]

    \item $\{ E_8(a_3), E_8(a_2) \}$:
    \[\begin{array}{ccccccccc}
         \widetilde\alpha_1  & \alpha_2 & \alpha_3 & \alpha_4 & \beta_1 & \beta_2 & \beta_3 & \beta_4 & ~ \\
         1 & 0 &   0 &   -3 & 0 & 0 & 0 & 0 & \widetilde\alpha_1 \\
         0 & 1 &   0 &   0 & 0 & 0 & 0 & 0 & \alpha_2 \\
         0 & 0 &   1 &   -1 & 0 & 0 & 0 & 0 & \alpha_3 \\
         0 & 0 &   0 &   -1 & 0 & 0 & 0 & 0 & \widetilde\alpha_4 \\
         0 & 0 &   0 &   0 &  1 & 0 & 0 & 0 & \beta_1 \\
         0 & 0 &   0 &   -2 & 0 & 1 & 0 & 0 & \beta_2 \\
         0 & 0 &   0 &   -2 & 0 & 0 & 1 & 0 & \beta_3 \\
         0 & 0 &   0 &   -2 & 0 & 0 & 0 & 1 & \beta_4
     \end{array}\]

     \item $\{ E_8(a_4), E_8(a_1) \}$:
     \[\begin{array}{ccccccccc}
         \alpha_1  & \alpha_2 & \widetilde\alpha_3 & \alpha_4 & \beta_1 & \beta_2 & \beta_3 & \beta_4 \\
         1 & 0 &   0 &  -1 & 0 & 0 & 0 & 0 & \alpha_1 \\
         0 & 1 &   0 &  -1 & 0 & 0 & 0 & 0 & \alpha_2 \\
         0 & 0 &   1 &   0 & 0 & 0 & 0 & 0 & \widetilde\alpha_3 \\
         0 & 0 &   0 &  -1 & 0 & 0 & 0 & 0 & \widetilde\alpha_4 \\
         0 & 0 &   0 &  -1 & 1 & 0 & 0 & 0 & \beta_1 \\
         0 & 0 &   0 &   0 & 0 & 1 & 0 & 0 & \beta_2 \\
         0 & 0 &   0 &  -1 & 0 & 0 & 1 & 0 & \beta_3 \\
         0 & 0 &   0 &  -1 & 0 & 0 & 0 & 1 & \beta_4
      \end{array}\]

    \item $\{ E_8(a_5), E_8(a_4) \}$:
    \[\begin{array}{ccccccccc}
         \alpha_1  & \alpha_2 & \widetilde\alpha_3 & \widetilde\alpha_4 & \beta_1 & \beta_2 & \beta_3 & \beta_4 & ~ \\
         1 & 0 &   0 &   0 & 0 & 0 & 0 & -1 & \alpha_1 \\
         0 & 1 &   0 &   0 & 0 & 0 & 0 & -1 & \alpha_2 \\
         0 & 0 &   1 &   0 & 0 & 0 & 0 & 0 & \widetilde\alpha_3 \\
         0 & 0 &   0 &   1 & 0 & 0 & 0 & -1 & \widetilde\alpha_4 \\
         0 & 0 &   0 &   0 & 1 & 0 & 0 & -1 & \beta_1 \\
         0 & 0 &   0 &   0 & 0 & 1 & 0 & 0 & \beta_2 \\
         0 & 0 &   0 &   0 & 0 & 0 & 1 & -1 & \beta_3 \\
         0 & 0 &   0 &   0 & 0 & 0 & 0 & -1 & \widetilde\beta_4
    \end{array}\]

    \item $\{ E_8(a_6), E_8(a_4) \}$:
    \[\begin{array}{ccccccccc}
         \alpha_1  & \alpha_2 & \widetilde\alpha_3 & \widetilde\alpha_4 & \beta_1 & \beta_2 & \beta_3 & \beta_4 & ~ \\
         1 & 0 &   0 &   0 & 0 & 0 & 0 & -2 & \alpha_1 \\
         0 & 1 &   0 &   0 & 0 & 0 & 0 & -2 & \alpha_2 \\
         0 & 0 &   1 &   0 & 0 & 0 & 0 & 0 & \widetilde\alpha_3 \\
         0 & 0 &   0 &   1 & 0 & 0 & 0 & 0 & \widetilde\alpha_4 \\
         0 & 0 &   0 &   0 & 1 & 0 & 0 & -2 & \beta_1 \\
         0 & 0 &   0 &   0 & 0 & 1 & 0 & -1 & \beta_2 \\
         0 & 0 &   0 &   0 & 0 & 0 & 1 & -1 & \beta_3 \\
         0 & 0 &   0 &   0 & 0 & 0 & 0 & -1 & \widetilde\beta_4
      \end{array}\]

    \item $\{ E_8(a_7), E_8(a_5) \}$:
    \[\begin{array}{ccccccccc}
         \alpha_1  & \alpha_2 & \widetilde\alpha_3 &
         \widetilde\alpha_4 & \beta_1 & \beta_2 & \beta_3 & \widetilde\beta_4 & ~ \\
         -1 & 0 &   0 &   0 & 0 & 0 & 0 & 0 & \widetilde\alpha_1 \\
         -1 & 1 &   0 &   0 & 0 & 0 & 0 & 0 & \alpha_2 \\
         -1 & 0 &   1 &   0 & 0 & 0 & 0 & 0 & \widetilde\alpha_3 \\
         0 & 0 &   0 &   1 & 0 & 0 & 0 & 0 & \widetilde\alpha_4 \\
         -2 & 0 &   0 &   0 & 1 & 0 & 0 & 0 & \beta_1 \\
         0 & 0 &   0 &   0 & 0 & 1 & 0 & 0 & \beta_2 \\
         0 & 0 &   0 &   0 & 0 & 0 & 1 & 0 & \beta_3 \\
         0 & 0 &   0 &   0 & 0 & 0 & 0 & 1 & \widetilde\beta_4
      \end{array}\]

    \item $\{ E_8(a_8), E_8(a_7) \}$:
    \[\begin{array}{ccccccccc}
         \widetilde\alpha_1  & \alpha_2 & \widetilde\alpha_3 &
         \widetilde\alpha_4 & \beta_1 & \beta_2 & \beta_3 & \widetilde\beta_4 & ~ \\
         1 & 0 &   0 &   0 & 0 & 0 & 0 & 0 & \widetilde\alpha_1 \\
         0 & 1 &   0 &   0 & 0 & 0 & 0 & 0 & \alpha_2 \\
         0 & 0 &   1 &   0 & 0 & 0 & 0 & 0 & \widetilde\alpha_3 \\
         0 & 0 &   0 &   1 & 0 & 0 & 0 & -2 & \widetilde\alpha_4 \\
         0 & 0 &   0 &   0 & 1 & 0 & 0 & 0 & \beta_1 \\
         0 & 0 &   0 &   0 & 0 & 1 & 0 & -1 & \beta_2 \\
         0 & 0 &   0 &   0 & 0 & 0 & 1 & -1 & \beta_3 \\
         0 & 0 &   0 &   0 & 0 & 0 & 0 & -1 & \beta_4
      \end{array}\]
\end{enumerate}

\section{For discussion}

\subsection{Conjectures regarding the Dynkin-Minchenko completion procedure}
  \label{sec_conjectures}

The numerical relation in  Remark~\ref{obs_corresp_DM} motivates to the following assumption:
\begin{conjecture} \label{conj_1}
There is a correspondence between Carter diagrams  (with cycles)
and extra nodes in the Dynkin-Minchenko completion procedure.
\end{conjecture}

One more conjecture on relationship between
enhanced Dynkin diagrams and Carter diagrams is as follows:
\begin{conjecture} \label{conj_2}
If $\{\widetilde\Gamma, \Gamma \}$  is a homogeneous pair of Carter diagrams, then the signed enhanced Dynkin diagrams associated with $\widetilde\Gamma$ and $\Gamma$ coincide up to similarities.
\begin{equation*}
   \Delta(\widetilde\Gamma) = \Delta(\Gamma).
\end{equation*}
\end{conjecture}

\begin{remark}
Conjecture \ref{conj_2} is easily verified for Carter diagrams $E_6(a_1)$ and $E_6(a_2)$:
    \begin{itemize}
        \item Case $E_6(a_1)$. Extra nodes $m_1$ and $m_2$ are as follows:
        \begin{equation*}
          m_1 = 2\beta_1 + \alpha_1 + \alpha_2 + \alpha_3, \quad
          m_2 = 2\alpha_2 + \beta_1 + \beta_2 - \beta_3.
        \end{equation*}
        
        \begin{figure}[!h]
            \centering
            \includegraphics[scale=0.25]{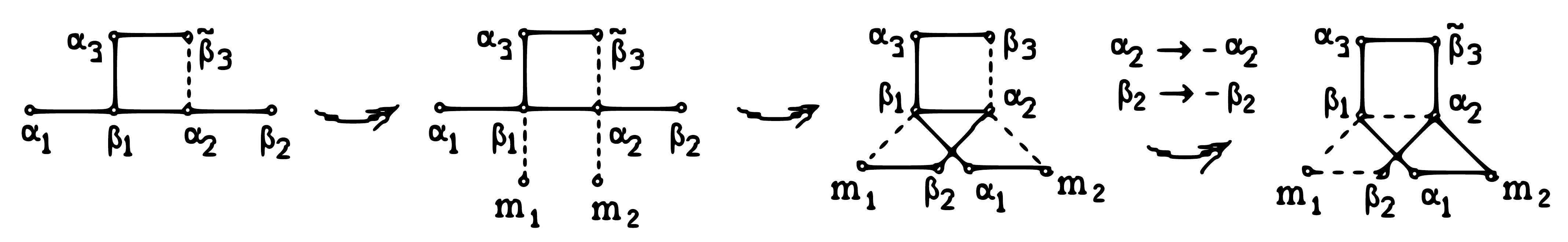}
            \caption{Enhanced Dynkin diagram $\Delta(E_6(a_1))$}
            \label{fig_enhan_E6a1}
        \end{figure}
        
        Orthogonality relations are as follows:
        \begin{equation*}
        \begin{split}
           & m_1 \perp \alpha_1, \alpha_2, \alpha_3, \widetilde\beta_3, \qquad   (m_1, \beta_2)  = (\alpha_2, \beta_2) = -1,\\
           & (m_1, m_2) = (m_1, \beta_1 + \beta_2) = 1 - 1 = 0.  \\
        \end{split}
        \end{equation*}

        \item Case $E_6(a_2)$. Here, extra nodes $m_1$ and $m_2$ are as follows:
        \begin{equation*}
          m_1 = 2\alpha_3 + \widetilde\beta_3 + \widetilde\beta_2 + \beta_1, \quad
          m_2 = 2\beta_1 + \alpha_1 + \alpha_2 + \alpha_3.
        \end{equation*}

        \begin{figure}
            \centering
            \includegraphics[scale=0.25]{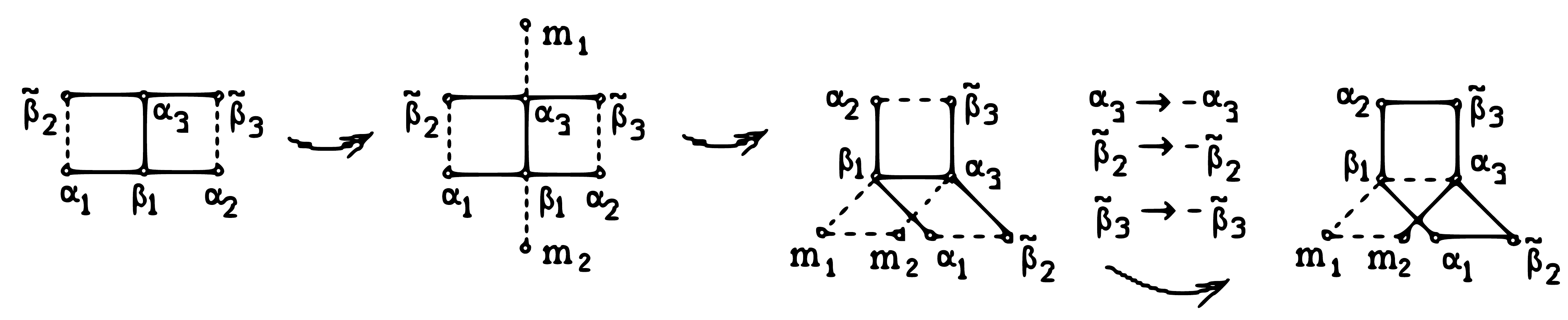}
            \caption{Enhanced Dynkin diagram $\Delta(E_6(a_2))$}
            \label{fig_enhan_E6a2}
        \end{figure}
        
        Orthogonality relations:
        \begin{equation*}
            m_1 \perp \widetilde\beta_3, \widetilde\beta_2,  \beta_1, \alpha_1,  \alpha_2, \qquad
            (m_1, m_2) = (m_1, \alpha_3) = 1.
        \end{equation*}
    \end{itemize}
\end{remark}

Further, according to Vavilov-Mingin, see \cite[Theorem 1]{19},
\begin{equation*}
 E_7(a_i) \subset \Delta(E_7)   \text{ and } E_8(a_i) \subset \Delta(E_8), \quad 1 \le i \le 8,
\end{equation*}
therefore
\begin{equation*}
   \Delta(E_7(a_i)) \subset \Delta(E_7) \text{ and } \Delta(E_8(a_i)) \subset \Delta(E_8).
\end{equation*}
Thus, to prove Conjecture \ref{conj_2}, it suffices to check only the reverse inclusions:
\begin{equation*}
   E_7 \subset \Delta(E_7(a_i)) \text{ and } E_8 \subset \Delta(E_8(a_i)).
\end{equation*}

\subsection{Adjacency, complexity and eigenvalues}
 \label{sec_eigenv}
 Let us define the \textit{complexity} of the Carter diagram as $Nc + Ke$, where
 $N$ is the number of cycles and $K$ is the number of endpoints.
 Assume that one cycle contributes to complexity as two endpoints.
 One can select another proportion.
 In Table \ref{tab_complexity},
 Carter diagrams located side by side are the pairs from the adjacency list \eqref{eq_pairs_trans}.
 The Carter diagrams from the adjacency list can be transformed to each other using
 the \textit{transition matrix} $M_I$ constructed in Theorem \ref{th_invol},
 see Section \ref{sec_prod_trans_matr}.
 Denote by $Max$-$E$ the maximal eigenvalue of a partial Cartan matrix.

\begin{table}[!h]
    \centering
    \renewcommand{\arraystretch}{1.3}
    \begin{tabular} {|c|c|c|c|c|c|c|c|c|}
        \hline
        & $E_8(a_8)$  & $E_8(a_7)$ & $E_8(a_5)$  & $E_8(a_4)$  & $E_8(a_1)$ & $E_8$ & $E_8(a_2)$ & $E_8(a_3)$ \\
        \hline
        $Nc+Ke$ & $6c$   & $3c + 1e$  & $2c + 2e$   & $2c + 1e$   & $1c + 2e$  & $3e$ & $1c + 3e$ & $1c + 4e$ \\
        \hline
        $2N+K$ & $12$   & $7$        & $6$         & $5$   & $4$  & $3$        & $5$ & $6$ \\
        \hline
        $Max$-$E$ & $3.73$      & $3.93$      & $3.956$     & $3.969$    & $3.982$    & $3.989$ & $3.975$ & $3.93$ \\
        \hline
    \end{tabular}
    \caption{ Two chains arranged in ascending order of $Max$-$E$. }
    \label{tab_complexity}
\end{table}

In Table \ref{tab_complexity}, there are two chains which are arranged in ascending
order of the \textit{maximal eigenvalues} of the partial Cartan matrices
and in descending order of the \textit{complexity} parameter, see \eqref{eq_2_ascend_chains}.
\begin{equation}
  \label{eq_2_ascend_chains}
\begin{split}
 & E_8(a_8) \xrightarrow{\hspace{0.4cm}}
 E_8(a_7) \xrightarrow{\hspace{0.4cm}} E_8(a_5) \xrightarrow{\hspace{0.4cm}}
 E_8(a_4) \xrightarrow{\hspace{0.4cm}} E_8(a_1) \xrightarrow{\hspace{0.4cm}} E_8, \\
 & E_8(a_3) \xrightarrow{\hspace{0.4cm}} E_8(a_2) \xrightarrow{\hspace{0.4cm}} E_8. \\
 \end{split}
\end{equation}
The arrows in \eqref{eq_2_ascend_chains} point in the direction of increasing of maximal eigenvalue.
It is not so clear the place of $E_8(a_6)$ in the homogeneous class $\{E, 8\}$:

\begin{itemize}
\item  The complexity parameter for the diagram $E_8(a_6)$ is equal to $3c$.
\item  By alternative transition matrices from Section \ref{sec_altern}, the possible adjacent diagram for
$E_8(a_6)$ can be $E_8(a_5)$, $E_8(a_4)$, $E_8(a_1)$.
\item  The maximum eigenvalue of the partial Cartan matrix for $E_8(a_6)$ is $3.902$.
\end{itemize}

Based on these three factors, $E_8(a_6)$
can be placed, for example, between $E_8(a_5)$ and $E_8(a_4)$, or in the separated pair $\{ E_8(a_5), E_8(a_6) \}$.

\subsection{C.~M.~Ringel: Invariants with value $2$, $4$, $8$}
  \label{sec_Ringel}

In the survey article \cite{13}, C.~M.~Ringel provided several notes regarding representations of
Dynkin quivers. As Ringel writes: {\it``they shed some new light on properties of Dynkin and
Euclidean quivers''}. The following is Ringel's $2-4-8$ assertion regarding
Dynkin quivers $E_n$, $n = 6,7,8$.

Let $\overline\Gamma$ be the extended Dynkin diagram (=Euclidean quiver)
for the Dynkin diagram $\Gamma$. If $\overline\Gamma$ is constructed from $\Gamma$
by adding the new edge to the vertex $y$ of $\Gamma$, then $y$ is said to be the exceptional vertex.
For the Dynkin quivers $E_n$, $n = 6,7,8$, let $x$ be the neighbor of $y$ and $\Gamma' = \Gamma\backslash\{y\}$,
see Table \ref{tab_invariants}.

\begin{table}
    \centering
    \begin{tabular} {||c||c|c|c||}
        \hline
        $\Gamma$    & $E_6$  & $E_7$  & $E_8$ \\
        \hline
        $\Gamma'$   & $A_5$  & $D_6$  & $E_7$  \\
        \hline
        $n$         & $2$    &  $4$   & $8$ \\
        \hline
    \end{tabular}
    \caption{\hspace{1mm} Ringel $2-4-8$ assertion}
    \label{tab_invariants}
\end{table}

Let $P'(x)$ (resp. $I'(x)$) be the indecomposable projective (resp. injective) representation of $\Gamma'$ corresponding to vertex $x$. Let $\tau'$ be the Auslander-Reiten translation in the category of finite-dimensional representations
rep($\Gamma$). Then, $P'(x)$ and $I'(x)$ belong to the same $\tau'$-orbit and there is the following
$2-4-8$ assertion:
\begin{equation*}
   P'(x) = (\tau')^n I'(x),
\end{equation*}
where $n$ is the \textit{Ringel invariant} given in the Table \ref{tab_invariants}, see \cite[Part 3]{13}.
For example, for the Auslander-Reiten $A_5$, $n = 2$, see  Fig. \ref{AR_quiver_A5}, and for the Auslander-Reiten $D_6$, $n = 4$, see Fig.~\ref{AR_quiver_D6}.

\begin{figure}
    \centering
    \includegraphics[width=.66\textwidth]{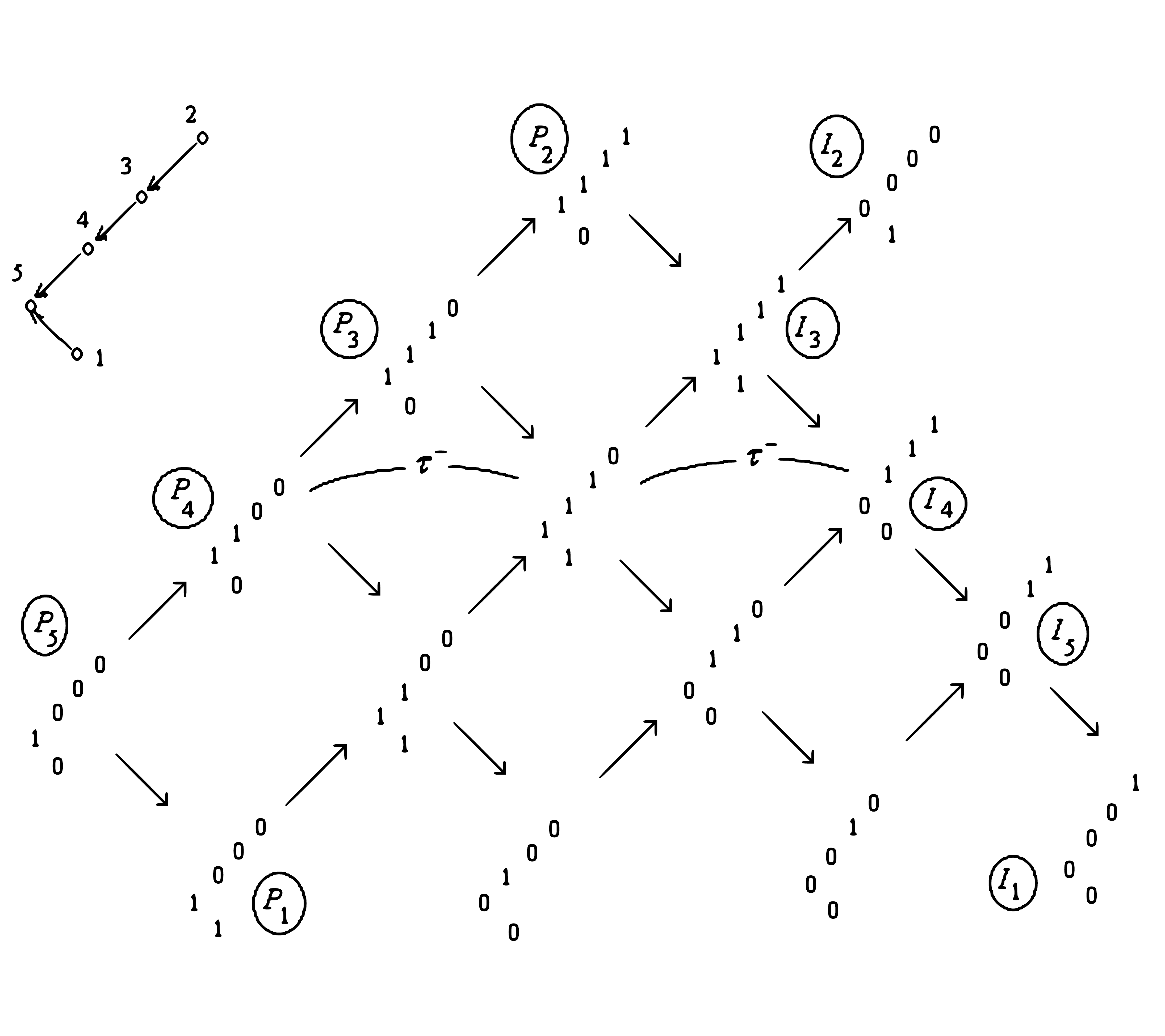}
    \caption{Auslander-Reiten quiver $A_5$: $\tau^- = s_5s_4s_3s_2s_1$, $\tau^+ = s_1s_2s_3s_4s_5$, $\tau^{-2} P_4 = I_4$.}
    \label{AR_quiver_A5}
\end{figure}

\begin{figure}
    \centering
    \includegraphics[width=\textwidth]{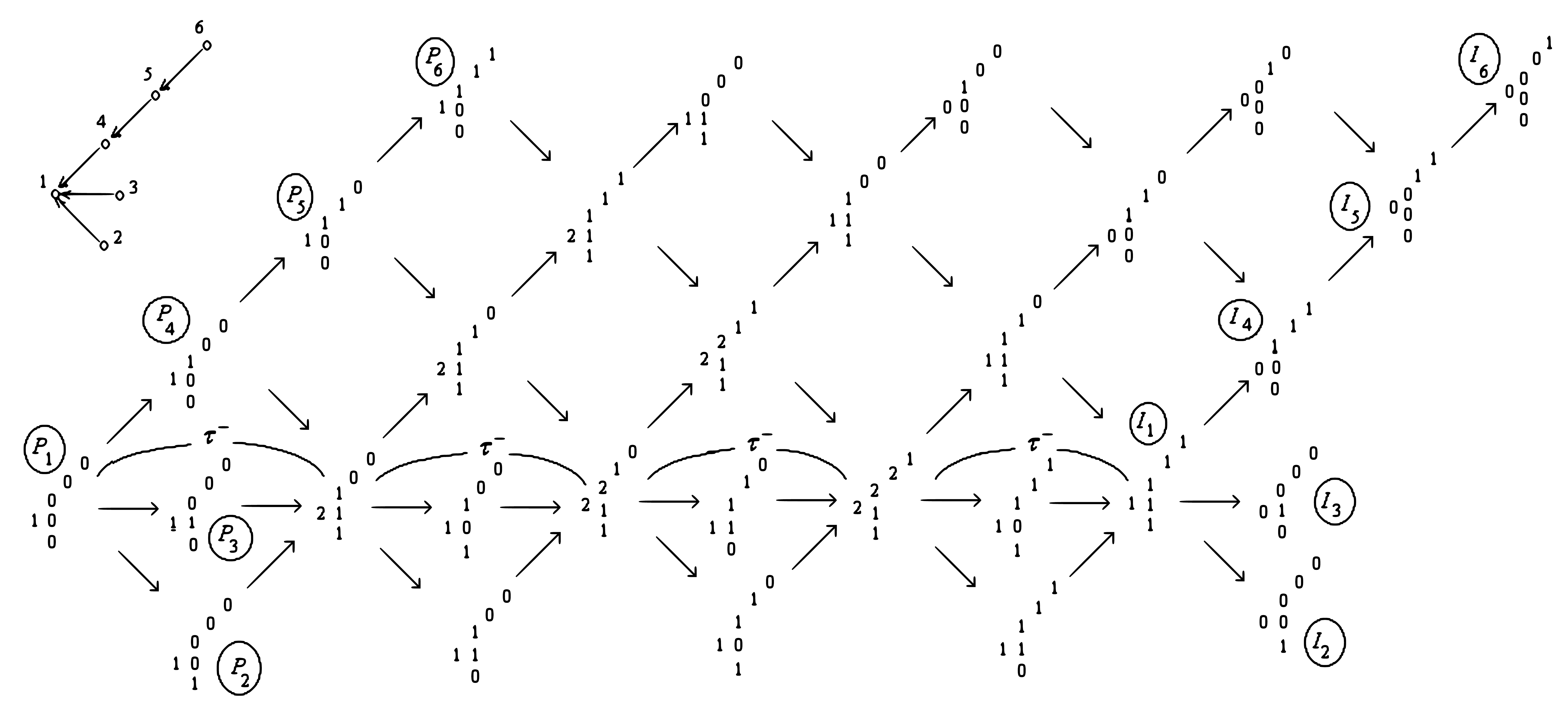}
    \caption{Auslander-Reiten quiver $D_6$: $\tau^- = s_1s_2s_3s_4s_5s_6$, $\tau^+ = s_6s_5s_4s_3s_2s_1$. For each $i$, we have $\tau^{-4} P_i = I_i$.}
    \label{AR_quiver_D6}
\end{figure}

\subsection{B.~Rosenfeld: Isometry groups of the projective planes}
  \label{sec_Baez}

J.~Baez in \cite{1} points out another connection between the invariants $2$, $4$, $8$
and the diagrams $E_6$, $E_7$ and $E_8$. This connection was discovered by Rosenfeld, see \cite{14}:
\textit{``$\dots$ Boris Rosenfeld had the remarkable idea $\dots$ that
the exceptional Lie groups $E_6$, $E_7$ and $E_8$ may be
considered as the isometry groups of the projective planes over the following $3$
algebras, respectively:''}

\begin{itemize}
 \item the bioctonions $\mathbb{C} \otimes \mathbb{O}$,
 \item the quateroctonions $\mathbb{H} \otimes \mathbb{O}$,
 \item the octoctonions $\mathbb{O} \otimes \mathbb{O}$.
\end{itemize}
 Any real finite-dimensional division algebra over the reals must be only one of these:
 $\mathbb{R}$, $\mathbb{C}$, $\mathbb{H}$, $\mathbb{O}$\footnotemark[1].
 The real numbers $\mathbb{R}$,  complex numbers $\mathbb{C}$, the quaternions $\mathbb{H}$, and the octonions $\mathbb{O}$
 are division algebras of dimensions, respectively: $1$, $2$, $4$, or $8$.
\footnotetext[1] {This theorem is due to Kervaire and Bott-Milnor, 1958.}

\end{appendix}

\EditInfo{July 5, 2022}{June 6, 2023}{Dimitry Leites and Sofiane Bouarroudj}


{\small
\begin{thebibliography}{10}

\bibitem{1}
J.~C. Baez.
\newblock The octonions.
\newblock {\em Bull. Amer. Math. Soc. (N.S.)}, 39(2):145--205, 2002.

\bibitem{2}
K.~Bongartz.
\newblock Algebras and quadratic forms.
\newblock {\em J. London Math. Soc. (2)}, 28(3):461--469, 1983.

\bibitem{3}
S.~Brenner.
\newblock Quivers with commutativity conditions and some phenomenology of
  forms.
\newblock In {\em Proceedings of the {I}nternational {C}onference on
  {R}epresentations of {A}lgebras ({C}arleton {U}niv., {O}ttawa, {O}nt.,
  1974)}, Carleton Math. Lecture Notes, No. 9, pages Paper No. 5, 25. Carleton
  Univ., Ottawa, Ont., 1974.

\bibitem{6}
R.~W. Carter.
\newblock Conjugacy classes in the {W}eyl group.
\newblock {\em Compositio Math.}, 25:1--59, 1972.

\bibitem{7}
D.~Chapovalov, M.~Chapovalov, A.~Lebedev, and D.~Leites.
\newblock The classification of almost affine (hyperbolic) {L}ie superalgebras.
\newblock {\em J. Nonlinear Math. Phys.}, 17(suppl. 1):103--161, 2010.

\bibitem{8}
E.~B. Dynkin and A.~N. Minchenko.
\newblock Enhanced {D}ynkin diagrams and {W}eyl orbits.
\newblock {\em Transform. Groups}, 15(4):813--841, 2010.

\bibitem{5}
P.~Gabriel and A.~V. Roiter.
\newblock {\em Representations of finite-dimensional algebras}.
\newblock Springer-Verlag, Berlin, 1997.
\newblock Translated from the Russian, With a chapter by B. Keller, Reprint of
  the 1992 English translation.

\bibitem{4}
V.~A. Gritsenko and V.~V. Nikulin.
\newblock On the classification of {L}orentzian {K}ac-{M}oody algebras.
\newblock {\em Uspekhi Mat. Nauk}, 57(5(347)):79--138, 2002.

\bibitem{9}
J.~E. Humphreys.
\newblock {\em Reflection groups and {C}oxeter groups}, volume~29 of {\em
  Cambridge Studies in Advanced Mathematics}.
\newblock Cambridge University Press, Cambridge, 1990.

\bibitem{10}
V.~G. Kac.
\newblock Infinite root systems, representations of graphs and invariant
  theory.
\newblock {\em Invent. Math.}, 56(1):57--92, 1980.

\bibitem{11}
J.~McKee and C.~Smyth.
\newblock Integer symmetric matrices having all their eigenvalues in the
  interval {$[-2,2]$}.
\newblock {\em J. Algebra}, 317(1):260--290, 2007.

\bibitem{12}
R.~Mulas and Z.~Stani\'{c}.
\newblock Star complements for {$\pm2$} in signed graphs.
\newblock {\em Spec. Matrices}, 10:258--266, 2022.

\bibitem{13}
C.~M. Ringel.
\newblock Representation theory of {D}ynkin quivers. {T}hree contributions.
\newblock {\em Front. Math. China}, 11(4):765--814, 2016.

\bibitem{14}
B.~Rosenfeld.
\newblock {\em Geometry of {L}ie groups}, volume 393 of {\em Mathematics and
  its Applications}.
\newblock Kluwer Academic Publishers Group, Dordrecht, 1997.

\bibitem{15}
D.~Simson.
\newblock Integral bilinear forms, {C}oxeter transformations and {C}oxeter
  polynomials of finite posets.
\newblock {\em Linear Algebra Appl.}, 433(4):699--717, 2010.

\bibitem{16}
R.~Stekolshchik.
\newblock {\em Notes on {C}oxeter transformations and the {M}c{K}ay
  correspondence}.
\newblock Springer Monographs in Mathematics. Springer-Verlag, Berlin, 2008.

\bibitem{17}
R.~Stekolshchik.
\newblock Classification of linkage systems, arxiv:1406.3049v2, 2014.

\bibitem{18}
R.~Stekolshchik.
\newblock Equivalence of {C}arter diagrams.
\newblock {\em Algebra Discrete Math.}, 23(1):138--179, 2017.

\bibitem{19}
N.~Vavilov and V.~Migrin.
\newblock Enhanced {D}ynkin diagrams done right.
\newblock {\em Zap. Nauchn. Sem. S.-Peterburg. Otdel. Mat. Inst. Steklov.
  (POMI)}, 500(Voprosy Teorii Predstavleni\u{\i} Algebr i Grupp. 37):11--29,
  2021.

\bibitem{20}
N.~Vavilov and V.~Migrin.
\newblock Colourings of exceptional uniform polytopes of types $e_6$ and $e_7$.
\newblock {\em Zap. Nauchn. Sem. S.-Peterburg. Otdel. Mat. Inst. Steklov.
  (POMI)}, 517(Teoriya Predstavlenii, Dinamicheski Sistemy, Kombinatornye
  Metody. XXXIV):36--54, 2022.

\end{thebibliography}
}
\end{document}